\documentclass[11pt]{amsart}
\usepackage{amsmath,amssymb,latexsym,esint,cite,mathrsfs}
\usepackage{verbatim,wasysym}
\usepackage[left=2.3cm,right=2.3cm,top=2.65cm,bottom=2.65cm]{geometry}
\usepackage{tikz,enumitem,graphicx, subfig, microtype, color}
\usepackage{epic,eepic}

\usepackage[colorlinks=true,urlcolor=blue, citecolor=red,linkcolor=blue,
linktocpage,pdfpagelabels, bookmarksnumbered,bookmarksopen]{hyperref}
\usepackage[hyperpageref]{backref}
\usepackage[english]{babel}
\allowdisplaybreaks
\numberwithin{equation}{section}
\newtheorem{theorem}{Theorem}[section]

\newtheorem{lemma}[theorem]{Lemma}
\theoremstyle{definition}

\newtheorem{definition}[theorem]{Definition}
\newtheorem{remark}[theorem]{Remark}

\newcommand{\R}{\mathbb{R}}

\newcommand{\supp}{{\rm supp}{\hspace{.05cm}}}
\begin{document}
\title
[Normalized solutions for Chern-Simons-Schr\"{o}dinger systems]
{Normalized solutions to the Chern-Simons-Schr\"{o}dinger system: the
	supercritical case}

 \author[L.\ Shen]{Liejun Shen}

\author[M.\ Squassina]{Marco Squassina}

 \address{Liejun Shen, \newline\indent Department of Mathematics, Zhejiang Normal University, \newline\indent
	Jinhua, Zhejiang, 321004, People's Republic of China}
\email{\href{mailto:ljshen@zjnu.edu.cn}{ljshen@zjnu.edu.cn}}

\address{Marco Squassina, \newline\indent
	Dipartimento di Matematica e Fisica \newline\indent
	Universit\`a Cattolica del Sacro Cuore, \newline\indent
	Via della Garzetta 48, 25133, Brescia, Italy}
\email{\href{mailto:marco.squassina@unicatt.it}{marco.squassina@unicatt.it}}


\thanks{L.\ Shen was partially supported by NSFC (12201565). M.\  Squassina  is  member  of  Gruppo  Nazionale  per
	l'Analisi  Matematica,  la Probabilita  e  le  loro  Applicazioni  (GNAMPA)  of  the  Istituto  Nazionale  di  Alta  Matematica  (INdAM)}

\begin{abstract}
We are concerned with the existence of normalized solutions for a class of generalized Chern-Simons-Schr\"{o}dinger type problems
with supercritical exponential growth
\[
  \left\{
  \begin{array}{ll}
\displaystyle    -\Delta u +\lambda u+A_0 u+\sum\limits_{j=1}^2A_j^2 u=f(u), \\
 \displaystyle     \partial_1A_2-\partial_2A_1=-\frac{1}{2}|u|^2,~\partial_1A_1+\partial_2A_2=0,\\
 \displaystyle    \partial_1A_0=A_2|u|^2,~ \partial_2A_0=-A_1|u|^2,\\
 \displaystyle    \int_{\R^2}|u|^2dx=a^2,
  \end{array}
\right.
\]
where $a\neq0$, $\lambda\in\R$ is known as the Lagrange multiplier
and $f\in \mathcal{C}^1(\R)$ denotes the nonlinearity that fulfills the
 supercritical exponential growth in the Trudinger-Moser sense at infinity.
Under suitable assumptions, combining the constrained minimization approach together with the homotopy stable family and elliptic regularity theory,
 we obtain that the problem has at least a ground state solution.
\end{abstract}
\maketitle

\begin{center}
	\begin{minipage}{8.5cm}
		\small
		\tableofcontents
	\end{minipage}
\end{center}

\smallskip

\section{Introduction and main results}

In this article, we investigate the existence of solutions for the following generalized Chern-Simons-Schr\"{o}dinger
(\textbf{CSS} in short) system/equation with supercritical exponential growth
\begin{equation}\label{mainequation1}
	\left\{
	\begin{array}{ll}
		\displaystyle    -\Delta u +\lambda u+A_0 u+\sum\limits_{j=1}^2A_j^2 u= f(u), \\
		\displaystyle     \partial_1A_2-\partial_2A_1=-\frac{1}{2}|u|^2,~\partial_1A_1+\partial_2A_2=0,\\
		\displaystyle    \partial_1A_0=A_2|u|^2,~ \partial_2A_0=-A_1|u|^2,
	\end{array}
	\right.
\end{equation}
under the constraint
\begin{equation}\label{mainequation1a}
	\int_{\R^2}|u|^2dx=a^2,
\end{equation}
where $a\neq0$, $\lambda\in\R$ is known as the Lagrange multiplier
and $f\in \mathcal{C}^0(\R)$ denotes the nonlinearity that fulfills the
supercritical exponential growth in the Trudinger-Moser sense at infinity which would be specified later.

In recent years, the following time-dependent CSS system in two spatial dimension
\begin{equation}\label{CSS1}
	\begin{cases}
		\text{i}D_0\psi+(D_1D_1+D_2D_2)\psi+g(|\psi|^2)\psi=0, \\
		\partial_0A_1-\partial_1A_0=-\text{Im}(\overline{\psi} D_2\psi),\\
		\partial_0A_2-\partial_2A_0=\text{Im}(\overline{\psi} D_1\psi), \\
		\partial_1A_2-\partial_2A_1=-\frac{1}{2}|\psi|^2 ,\\
	\end{cases}
\end{equation}
is usually exploited to describe the non-relativistic dynamics behavior
of massive number of particles in Chern-Simons gauge fields,
where i stands for the imaginary unit,
$\partial_0=\frac{\partial}{\partial t}$, $\partial_1=\frac{\partial}{\partial x_1}$, $\partial_2=\frac{\partial}{\partial x_2}$
for $(t,x_1,x_2)\in\R^{1+2}$, $\psi:\R^{1+2}\to \mathbb{C}$ is the complex scalar field,  $A_j:\R^{1+2}\to\R$ denotes the gauge field,
$D_j=\partial_j +\text{i}A_j$ is the covariant derivative for $j=0, 1, 2$
and the function $g$ is the nonlinearity.

This model plays an important role in
the study of the high-temperature superconductor, Aharovnov-Bohm scattering, and quantum Hall effect,
we refer the reader to \cite{Jackiw1,Jackiw2,Jackiw3}. Moreover, there are some further physical motivations
for considering CSS system \eqref{CSS1}, see e.g.
\cite{Dunne,Huh1,LS1,LS2}.

For $j=0,1,2$, we always deal with $A_j(t,x)=A_j(x)$ for all $(t,x_1,x_2)\in\R^{1+2}$.
If the standing wave ansatz $\psi(t,x)=e^{\text{i}\lambda t}u(x)$ with a given $\lambda\in\R$
for $u:\R^2\to\R$, then \eqref{CSS1} reduces to
\begin{equation}\label{CSS2a}
	\left\{
	\begin{array}{ll}
		\displaystyle  -\Delta u+\lambda u+A_0 u+\sum_{j=1}^2A_j^2 u=f(u),  \\
		\displaystyle     \partial_1A_2-\partial_2A_1=-\frac{1}{2}|u|^2,\\
		\displaystyle    \partial_1A_0=A_2|u|^2,~ \partial_2A_0=-A_1|u|^2,
	\end{array}
	\right.
\end{equation}
where $f(u)=g(|u|^2)u$. Let us suppose that the gauge field $A_j$ satisfies the Coulomb gauge condition
$\sum_{j=0}^2\partial_jA_j= 0$, then
\eqref{CSS2a} becomes the original CSS equation \eqref{mainequation1}, namely
\begin{equation}\label{CSS2b}
	\left\{
	\begin{array}{ll}
		\displaystyle    -\Delta u+\lambda u+A_0 u+\sum_{j=1}^2A_j^2 u=f(u),  \\
		\displaystyle    \partial_1A_0=A_2|u|^2,~ \partial_2A_0=-A_1|u|^2,\\
		\displaystyle  \partial_1A_2-\partial_2A_1=-\frac{1}{2}|u|^2,~
		\partial_1A_1+\partial_2A_2 =0.
	\end{array}
	\right.
\end{equation}
Since $\partial_1A_0=A_2|u|^2$ and $\partial_2A_0=-A_1|u|^2$
in \eqref{CSS2b}, there holds
\[
\Delta A_0 =\partial_1\big(A_2|u|^2\big)-\partial_2\big(A_1|u|^2\big),
\]
which leads to
\begin{equation}\label{CSS2d}
	A_0[u](x)=\frac{x_1}{2\pi |x|^2}\ast \big(A_2|u|^2\big)
	-\frac{x_2}{2\pi |x|^2}\ast \big(A_1|u|^2\big).
\end{equation}
Analogously, adopting $\partial_1A_2-\partial_2A_1=-\frac{1}{2}|u|^2$ and
$\partial_1A_1+\partial_2A_2 =0$ in \eqref{CSS2b} to deduce that
\[
\Delta A_1= \partial_2\bigg(\frac{|u|^2}{2}\bigg)~ \text{and}~
\Delta A_2= -\partial_1\bigg(\frac{|u|^2}{2}\bigg).
\]
From which,
the components $A_j$ for $j=1,2$ in \eqref{CSS2b} can be represented as
\begin{equation}\label{CSS2e1}
	A_1[u](x)=\frac{x_2}{2\pi|x|^2}\ast\bigg(\frac{|u|^2}{2}\bigg)
	=-\frac{1}{4\pi }\int_{\R^2}\frac{(x_2-y_2)u^2(y)}{|x-y|^2}dy,
\end{equation}
\begin{equation}\label{CSS2e2}
	A_2[u](x) =-\frac{x_1}{2\pi|x|^2}\ast\bigg(\frac{|u|^2}{2}\bigg)
	=\frac{1}{4\pi }\int_{\R^2}\frac{(x_1-y_1)u^2(y)}{|x-y|^2}dy.
\end{equation}
When there is no misunderstanding, for simplicity we shall write shortly $A_j$ in place of $A_j[u]$, for $j=0,1,2$.
Further properties of $A_j$, for $j=0,1,2$, will be introduced in Section \ref{Section2} below.

Actually, if $u$ is radially symmetric in
the standing wave ansatz $\psi(t,x)=e^{\text{i}\lambda t}u(x)$,
CSS system \eqref{CSS1} becomes a single equation.
In \cite{Byeon}, Byeon-Huh-Seok considered the standing waves  of type
\begin{equation}\label{CSS2}
	\begin{gathered}
		\psi(t,x)=u(|x|)e^{\text{i}\lambda t},~  A_0(t,x)=k(|x|), \hfill\\
		A_1(t,x)=\frac{x_2}{|x|^2}h(|x|),~   A_2(t,x)=-\frac{x_1}{|x|^2}h(|x|),\hfill\\
	\end{gathered}
\end{equation}
where $k$ and $h$ are real value functions depending only on $|x|$. Note that \eqref{CSS2} satisfies the Coulomb gauge condition with $\varsigma=ct+n\pi$, where $n$ is
an integer and $c$ is a real constant. To look for solutions of CSS system \eqref{CSS1} of the type \eqref{CSS2}, it suffices to solve the
following semilinear elliptic equation
\begin{equation}\label{BHS}
	-\Delta u+ \lambda u+ \bigg(\int_{|x|}^\infty
	\frac{h(s)}{s}u^2(s)ds+\frac{h^2(|x|)}{|x|^2} \bigg)u=f(u)
	~\text{in}~\mathbb{R}^2,
\end{equation}
where $h(s)=\int_0^s\frac{r}{2}u^2(r)dr$.

Now, we have two kinds of CSS type equations \eqref{mainequation1} and \eqref{BHS} in hands
and there are two aspects to the studies for them.
On the one hand, one can choose the frequency $\lambda\in\R$ to be fixed.
In this situation, the existence, nonexistence and multiplicity of nontrivial solutions
have been considerably contemplated by a lot of mathematicians, see \cite{AP,CZT,KT,LOZ,LYY,PR,Shen,SSY,PSZZ,WT,DPS,JCZSA}
and the references therein for example.

On the other hand, one can find solutions for \eqref{mainequation1} or \eqref{BHS}
with an unknown frequency $\lambda\in\R$ which appears as a Lagrange multiplier.
As we all know, the mass of each solution to the Cauchy problem for system \eqref{CSS1} is conserved along time, namely
$\int_{\R^2}|\psi(t,\cdot)|^2dx= \int_{\R^2}|\psi(0,\cdot)|^2dx$
for any $t\in[0,T)$. In physics, there exists a crucial meaning concerning the mass
which is often adopted to represent the power supply in nonlinear
optics or the total number of atoms in Bose-Einstein condensation, see \cite{Esry,Frantzeskakis,Malomed} for example.
As a consequence, taking physical point of views into account,
it is interesting to seek for solutions to \eqref{CSS1} with prescribed mass.

This
naturally leads to the study of solutions to \eqref{mainequation1}-\eqref{mainequation1a} for $a>0$ given. In this scenario, solutions to
\eqref{mainequation1}-\eqref{mainequation1a}
are referred to as normalized solutions. Indeed, solutions to \eqref{mainequation1}-\eqref{mainequation1a} correspond to critical points of the underlying energy
functional $E$ restricted on $S(a)$, where
\begin{equation}\label{functional}
	E(u)\triangleq\frac12\int_{\R^2}[|\nabla u|^2 + (A_1^2+A_2^2)|u|^2]dx-\int_{\R^2}F(u)dx
\end{equation}
and
\[
S(a)\triangleq\bigg\{u\in H^1(\R^2):\int_{\R^2}|u|^2dx=a^2\bigg\}.
\]
Moreover, from mathematical perspectives, the consideration of normalized solutions turns out to be also
meaningful since it benefits from understanding dynamical properties of stationary solutions to
equations like \eqref{CSS1}.

In recent years, the considerations for normalized solutions for
\eqref{mainequation1} or \eqref{BHS} have received more and more attentions.
Speaking precisely, for $f(u)=|u|^{p-2}u$,
Byeon \emph{et al.} \cite{Byeon} derived the existence of normalized solution
for each $a\neq0$ if $p\in(2,3]$ and sufficiently small $|a|$ if $p\in(3,4)$. Afterwards,
the authors in \cite{Yuan,LX} generalized the results in \cite{Byeon} to $p>4$.
For the particular case $p=4$, Gou and Zhang \cite{GZ} exhibited some very interesting results.
There exist some other related results on normalized solution in
\cite{LC,Luo,YCS1,YCS2,YTC} and the references therein.

It should be mentioned that the research interest could date back to the Schr\"{o}dinger equations with prescribed $L^2$-norm.
In \cite{Jeanjean1997}, with the help of a minimax approach and compactness argument,
Jenajean contemplated the existence of solutions for the following Schr\"{o}dinger
problem
\begin{equation}\label{Jeanjean}
	\left\{
	\begin{array}{ll}
		\displaystyle   -\Delta u+\lambda u=
		g(u) &\text{in}~ \R^N, \\
		\displaystyle     \int_{\R^N}|u|^2dx=a^2>0.
	\end{array}
	\right.
\end{equation}
Later on, there are some complements and generalizations in \cite{JeanjeanLu2}.
In \cite{Soave1}, for $g(t)=\mu|t|^{q-2}t+|t|^{p-2}t$ with $2<q\leq 2+\frac4N\leq p<2^*$,
Soave considered the existence of solutions for problem \eqref{Jeanjean},
where $2^*=\frac{2N}{N-2}$ if $N\geq3$ and $2^*=\infty$ if $N=2$.
For this type of combined nonlinearities, Soave \cite{Soave2} proved
the existence of ground state and excited solutions when $p=2^*$.
As to more results for problem \eqref{Jeanjean}, see e.g.
\cite{BartschSoave,JeanjeanLu1,JeanjeanLu3,LiXinfu,WeiWu} and the references therein.

Whereas, the planar case for critical problem \eqref{Jeanjean}
are definitely different from $N\geq3$.
In reality, one can observe that $2^*=\infty$ if $N=2$ and
$H^1(\R^2)\not\hookrightarrow L^\infty(\R^2)$ which makes the problems special and quite delicate.
So, it is not easy to deal with the nonlinearity involving
a (super)critical exponential growth trivially.
Due to the Trudinger-Moser type inequality, we can say that a function $f$ possesses the
\emph{critical exponential growth} at infinity if there exists a constant $\alpha_{0}>0$ such that
\begin{equation}\label{definitioncriticalgrowth}
	\lim _{t \rightarrow+\infty} \frac{|f(t)|}{e^{\alpha t^{2}}}= \begin{cases}0,
		& \forall \alpha>\alpha_{0}, \\ +\infty, & \forall \alpha<\alpha_{0}.\end{cases}
\end{equation}
The above definition was introduced by Adimurthi and Yadava in \cite{AYA}, see also de Figueiredo, Miyagaki
and Ruf \cite{Figueiredo} for example.
As to the \emph{subcritical exponential growth} at infinity for $f$, it says that
\begin{equation}\label{definitionsubcriticalgrowth}
	\lim _{t \rightarrow+\infty} \frac{|f(t)|}{e^{\alpha t^{2}}}=0,~\forall \alpha>0.
\end{equation}

Because of the appearance of the critical exponential growth in \eqref{definitioncriticalgrowth},
the results for problem \eqref{Jeanjean} with $N=2$ are fewer than
those for $N\geq3$ involving Sobolev critical growth. Very recently, Alves, Ji and Miyagaki \cite{AJM}
firstly studied problem \eqref{Jeanjean} in the case $N=2$ with $g$ satisfying \eqref{definitioncriticalgrowth}.
Based on the ideas in \cite{AJM}, the authors investigated the existence of normalized solutions for
\eqref{mainequation1} and \eqref{BHS} in \cite{YCS1} and \cite{YTC}, respectively.
In this article, we continue to study the existence of normalized solutions for \eqref{mainequation1}
with supercritical exponential growth which is definitely differently from \eqref{definitioncriticalgrowth}.

In order to depict clearly the ideas how to handle the supercritical exponential case,
we are going to derive the existence of normalized solutions for CSS equation \eqref{mainequation1}
with (sub)critical exponential growth. Let us impose the assumptions on $f$ as follows.
\begin{itemize}
	\item[$(f_1)$] $f \in \mathcal{C}^1(\mathbb{R})$ and
	there is a constant $\chi\geq4$ such that $f(s)=o(s^{\chi-1})$ as $s\to0$;
	\item[$(f_2)$] There is a constant $\theta>4$ such that $0<\theta F(s)\leq  f(s)s$ for all $s\neq0$;
	\item[$(f_3)$] The function $\bar{F}(s)=f(s)s-2F(s)$ satisfies
	\[
	\bar{F}(s)/|s|^{4}~\text{is strictly increasing in}~(0,+\infty).
	\]
\end{itemize}

Since we are interested in positive solutions for Eq. \eqref{mainequation1}, without loss of generality, we
suppose that $f(s)\equiv0$ for all $s\in(-\infty,0]$ until the end of the present paper.

Now, we can state the first main result.

\begin{theorem}\label{maintheorem1}
	Suppose that $f$ satisfies \eqref{definitionsubcriticalgrowth} and $(f_1)-(f_3)$,
	then there is a small $a^*>0$ such that problems \eqref{mainequation1}-\eqref{mainequation1a}
	possess a couple weak solution $(u_0,\lambda_0)\in H^1(\R^2)\times\R^+$ for all $a\in(0,a^*]$,
	where $\lambda_0>0$, $u_0(x)>0$ for all $x\in\R^2$ and $E(u_0)=m(a)$ with
	\begin{equation}\label{minimization}
		m(a)\triangleq \inf_{u\in \mathcal{M}(a)}E(u),~\mathcal{M}(a)=\big\{u\in S(a):J(u)=0\big\}.
	\end{equation}
	Here the energy functional $J:S(a)\to\R$ is defined by
	\[
	J(u)= \int_{\R^2}[|\nabla u|^2 + (A_1^2+A_2^2)|u|^2]dx-\int_{\R^2}[f(u)u-2F(u)]dx.
	\]
\end{theorem}

Given a $u\in S(a)$, it
follows from $(f_1)-(f_2)$ that
\[
E(tu(t\cdot))=\frac{t^2}2\int_{\R^2}[|\nabla u|^2 + (A_1^2+A_2^2)|u|^2]dx-t^{-2}\int_{\R^2}F(tu)dx\to-\infty
\]
as $t\to+\infty$ and so we cannot find critical points for $E$ restricted on $S(a)$ directly.
According to the discussions in Section \ref{Section2} below, we see that $\mathcal{M}(a)$
is a natural constraint manifold.

Next, we deal with the critical exponential growth case. To the end, we additionally suppose that
\begin{description}
	\item[$(f_4)$] there exist some constants $s_0>0$, $M_0>0$ and $\vartheta\in (0, 1]$ such that
	$$0<s^\vartheta F(s)\leq M_0f(s) ~\text{for all}~ s\geq s_0;$$
	\item[$(f_5)$] $\liminf\limits_{|s|\to+\infty}F(s)e^{-\alpha_0s^2}\geq\beta_0>0$, where $\beta_0$ is an arbitrary but fixed constant.
\end{description}

Our second main result reads in the following.

\begin{theorem}\label{maintheorem2}
Suppose that $f$ satisfies \eqref{definitioncriticalgrowth} and $(f_1)-(f_5)$,
then there is a small $a_*>0$ such that problems \eqref{mainequation1}-\eqref{mainequation1a}
admit a couple weak solution $(u_0,\lambda_0)\in H^1(\R^2)\times\R^+$ for all $a\in(0,a_*]$,
where $\lambda_0>0$, $u_0(x)>0$ for all $x\in\R^2$ and $E(u_0)=m(a)$.
\end{theorem}

\begin{remark}
In light of the nonlinearity $f$ in \eqref{mainequation1} is more general than its counterparts in
\cite{LX,LC,GZ}, so we can never simply repeat their methods to conclude Theorem \ref{maintheorem1}.
The critical exponential case has been studied in \cite{YCS1,YTC}, but there are
the following three important contributions which can be regarded as a partial motivation to
contemplate the problems behind this article.

(1) We do not require any compact condition which allows us to study a wider class of CSS equations.
Explaining more precisely, to overcome the loss of compactness, authors in \cite{YTC} look for solutions
in the radially symmetric subspace $H_r^1(\R^2)$ which implies that the imbedding $H^1_r(\R^2)\hookrightarrow L^p(\R^2)$
is compact for all $2<p<+\infty$; conversely, the authors in \cite{YCS1} find solutions in the work space
$X=\{u\in H^1(\R^2):\int_{\R^2}|x|^2|u|^2dx<+\infty\}$ which reveals that the imbedding $X\hookrightarrow L^p(\R^2)$
is compact for all $2\leq p<+\infty$.

(2) Although with the above two types of compact imbedding results in hands, authors in
\cite{YCS1,YTC} have to suppose that $a^2<1$ to recover the compactness. With the help of some specific calculations,
we would remove this unpleasant restriction. In other words, we only need to suppose that the mass $a^2$ is suitably small which is just
caused by the Chern-Simons term $\int_{\R^2} (A_1^2+A_2^2)|u|^2]dx$. Actually, we could show that
Theorems \ref{maintheorem1} and \ref{maintheorem2} holds true for all $a\neq0$ provided that it is absent.

(3) When taking the energy estimate, the assumption
\begin{description}
\item[$(f_5^\prime)$] There are constants $p>2$ and a sufficiently large $\varsigma>0$ such that $F(s)\geq\varsigma|s|^p$ for all $s\in\R$,
\end{description}
acts as a key role in \cite{YCS1,YTC}. Nevertheless, as pointed out in \cite{AS},
it covers the
essential feature of the critical exponential growth given in \eqref{definitioncriticalgrowth}
because $(f_5^\prime)$ is a global condition which requires $f$ to be $p$-superlinear growth for all
$t\in\R$ and the parameter $\varsigma$ must be large. Therefore, we shall depend on the slightly modified Moser
sequence functions introduced in \cite{AS} to reach the aim.
However, owing to the presence of the Chern-Simons term $\int_{\R^2} (A_1^2+A_2^2)|u|^2]dx$,
there are some additional challenges that we are necessary to bring in some technical calculations,
see e.g. Lemma \ref{estimateChern-Simons}.
\end{remark}

Finally, as the
 applications of Theorems \ref{maintheorem1} and \ref{maintheorem2},
 we are able to investigate the existence of normalized solutions for CSS equation
 \eqref{mainequation1} with supercritical exponential growth at infinity. Due to
\cite{AS1,AS2}, we suppose that  the nonlinearity $f$ has the form of the following particular type
 \begin{equation}\label{form}
  f(t)=h(t) e^{\bar{\alpha}_0 |t|^\tau},~\forall t\in\R,
\end{equation}
for some $\bar{\alpha}_0>0$
 and $\tau\geq2$. Hereafter, the $\mathcal{C}^1$ function $h$ that vanishes in $(-\infty,0])$ satisfies
\begin{description}
  \item[$(h_1)$] There is a constant $\chi\geq4$ such that $h(t)=o(t^{\chi-1})$ as $t\to0$;
  \item[$(h_2)$] There exists a $\theta>4$ such that $0<\theta H(t) \leq h(t)t$ for all $t\neq0$,
	where $H(t)=\int_{0}^{t}h(s)\,ds$;
\item[$(h_3)$] The function $\bar{H}(s)=h(s)s-2H(s)$ satisfies
\[
 \bar{H}(s)/|s|^{4}~\text{is strictly increasing in}~(0,+\infty).
\]
\item[$(h_4)$] There exist $\delta \in [8\theta^{-1},2)$ and $\gamma,M>0$ such that $0 \leq |h(t)| \leq Me^{\gamma |t|^{\delta}}$
for all $t\in\R$.
\end{description}

Motivated by \cite{AS1,AS2}, for the function $f$ defined in \eqref{form} together with $(h_2)$,
one could observe that it possesses the so-called \emph{\textbf{supercritical exponential growth}} at infinity when $\tau>2$,
for example
 \begin{equation}\label{definition2}
(\text{I})~\tau > 2~\text{is arbitrary if}~\bar{\alpha}_0>0~\text{is fixed;}~
 (\text{II})~\bar{\alpha}_0>0~\text{is arbitrary if}~\tau\geq2~\text{is fixed.}
 \end{equation}

 Let us exhibit the main result on this topic as follows.

 \begin{theorem}\label{maintheorem3}
 Suppose that the nonlinearity $f$ defined in \eqref{form}
requires $(h_1)-(h_4)$, then there is a small $\bar{a}^*>0$ such that,
for each $\tau>2$, there exists a $\bar{\alpha}^*_0=\bar{\alpha}_0(\tau)>0$ such that
problems \eqref{mainequation1}-\eqref{mainequation1a} possess a couple of weak solution
$(\bar{u}_0,\bar{\lambda}_0)$ for all fixed $\bar{\alpha}_0 \in (0, \bar{\alpha}_0^*)$ and $a\in(0,\bar{a}^*]$.
Moreover, if in addition
\begin{description}
 \item[$(h_5)$] there are some constants $\xi>0$ and $p>4$ such that $h(t)\geq \xi t^{p-1}$ for all $t\in[0,+\infty)$,
\end{description}
then there is a small $\bar{a}_*>0$ such that,
given a
$\bar{\alpha}_0>0$, there exist $\tau_*=\tau(\bar{\alpha}_0)>2$ and ${\xi}_0>0$
such that problems \eqref{mainequation1}-\eqref{mainequation1a} admit a couple of weak solution $(\bar{u}_0,\bar{\lambda}_0)$
 for every fixed $\tau \in [2,\tau_*)$,
  ${\xi}>{\xi}_0$ and $a\in(0,\bar{a}_*]$.
\end{theorem}

\begin{remark}
Up to the best knowledge of us so far, except \cite{Shen2},
there are not any existence results
for CSS equation with supercritical exponential critical growth.
As a consequence, it seems the first paper to consider the existence of normalized solutions for
supercritical CSS equation in $\R^2$. We emphasize here that although the authors in \cite{AS}
showed that the supercritical Schr\"{o}dinger equation with Stein-Weiss convolution parts
has at least a normalized solution,
one cannot derive Theorem \ref{maintheorem3} trivially
because of the Chern-Simons term $\int_{\R^2} (A_1^2+A_2^2)|u|^2]dx$.
What's more, the absence of compact imbedding results in this paper cannot also be ignored,
see \cite[Lemma 2.3]{AS} in detail.
\end{remark}

Let us briefly  sketch the proofs for our main results.
To conclude the proofs of Theorems \ref{maintheorem1} and \ref{maintheorem2},
we firstly introduce some interesting results for the Chern-Simons term $\int_{\R^2} (A_1^2+A_2^2)|u|^2]dx$ in the literature
and the Trudinger-Moser inequality developed by Cao \cite{Cao}
and so that the energy functional $E$, defined in \eqref{functional},
is of class $\mathcal{C}^1$ over $S(a)$. Then, we shall
investigate some properties of $\mathcal{M}(a)$
including it is natural constrain manifold, namely any minimizer of $m(a)$ in \eqref{minimization}
is indeed a weak solution of Eq. \eqref{mainequation1} with a suitable $\lambda\in\R$.
In the meantime, with the help of the homotopy stable family,
we can prove that there is a sequence $\{u_n\}\subset \mathcal{M}(a)$
which is a Palais-Smale sequence for $E$ restricted on $S(a)$ at the level $m(a)$.
To get the proof of Theorem \ref{maintheorem1},
combing the Vanishing lemma in \cite{Willem}
and Br\'{e}zis-Lieb lemma as well as the monotonicity of $m(a)$, we have that $E(u_0)=m(a)$
and $\lim\limits_{n\to\infty}|\nabla u_n|_2^2=|\nabla u_0|_2^2$,
where $u_0$ is the weak limit of $\{u_n\}$ in $H^1(\R^2)$
along a subsequence. On the other hand, we derive the strong convergence of
 $\{u_n\}$ in $H^1(\R^2)$ by showing that the Lagrange multiplier $\lambda_0$
associated with $u_0$ is positive. Compared with the proof of Theorem \ref{maintheorem1}, in order
to prove Theorem \ref{maintheorem2}, we make the following three additional adjustments:
(a) We find a upper bound by the Moser sequence functions for $m(a)$ which is adopted to verify that the weak limit $u_0\neq0$.
(b) By taking advantage of $(f_2)$ jointly with the monotonicity of $m(a)$, we still prove that $E(u_0)=m(a)$
and $\lim\limits_{n\to\infty}|\nabla u_n|_2^2=|\nabla u_0|_2^2$. Let us stress here that this idea is also suitable for the
subcritical case in Theorem \ref{maintheorem1} and it should be reviewed as one of the striking novelties
the the present article. (c) As an application of $\lim\limits_{n\to\infty}|\nabla u_n|_2^2=|\nabla u_0|_2^2$,
we shall derive that $\lambda_0>0$ which leads to that $\bar{a}\mapsto m(\bar{a})$  is strictly decreasing
on a right neighborhood of $\bar{a}$, where $\bar{a}^2=|u_0|_2^2\in(0,a^2]$.
According to it, we must have $\bar{a}^2=a^2$ which reveals the strong convergence of $\{u_n\}$ in $H^1(\R^2)$
and so the proof is done.

Concerning the proof of Theorem \ref{maintheorem3},
due to a variational method point of
view, we continue to search for critical points for the variational functional $E:S(a)\to \R$.
Unfortunately, we couldn't even show that $E$ is well-defined over $S(a)$
 directly which is caused by the appearance the nonlinearity $f$ involving the supercritical
exponential growth \eqref{definition2} in the Trudinger-Moser sense.
Hence, the most imperative starting point is to deduce that $E$ is of class $\mathcal{C}^1$. Given a
fixed constant $R > 0$, followed by \cite{AS1,AS2}, we shall consider an auxiliary equation which possesses a
(sub)critical exponential growth. Speaking it clearly, introducing a cutoff function $f^{R,\bar{\delta}}$  which is
given by
\begin{equation} \label{fR}
f^{R, \bar{\delta} }(t)=
\left\{
\begin{array}{ll}
0,&    t\leq0,\\
	 h(t) e^{\bar{\alpha}_0 t^\tau},&    0\leq t \leq R,\\
	h(t) e^{\bar{\alpha}_0 R^{\tau-\bar{\delta}}t^{\bar{\delta}}}, &  t \geq R,
\end{array}
\right.
\end{equation}
and
\[
 \bar{\delta}\triangleq\left\{
   \begin{array}{ll}
    \delta, & \text{if the Case I in \eqref{definition2} is considered},\\
    2, & \text{if the Case II in \eqref{definition2} is considered}.
    \end{array}
 \right.
\]
Here, the function $h$ appears in \eqref{form} and the constant $\delta\in(0,2)$ is proposed in the condition $(h_4)$.

In consideration of such a nonlinearity $f^{R,\bar{\delta}}$, let us study the following auxiliary equation
\begin{equation}\label{mainequation2}
    \left\{
  \begin{array}{ll}
\displaystyle    -\Delta u +\lambda u+A_0 u+\sum\limits_{j=1}^2A_j^2 u= f^{R,\bar{\delta}}(u), \\
 \displaystyle     \partial_1A_2-\partial_2A_1=-\frac{1}{2}|u|^2,~\partial_1A_1+\partial_2A_2=0,\\
 \displaystyle    \partial_1A_0=A_2|u|^2,~ \partial_2A_0=-A_1|u|^2,
  \end{array}
\right.
\end{equation}
Using $(h_4)$, it is ready to observe that $f^{R,\bar{\delta}}$
admits the
subcritical or critical exponential growth at infinity for every fixed $R > 0$.
Consequently, with the help of Theorems \ref{maintheorem1} and \ref{maintheorem2},
we are able to establish the existence of normalized solutions for Eq. \eqref{mainequation2}
by finding critical points of the energy functional $E^{R,\bar{\delta}}$ which is defined by
\begin{equation}\label{functional2}
E^{R,\bar{\delta}}(u)=\frac12\int_{\R^2}[|\nabla u|^2 + (A_1^2+A_2^2)|u|^2]dx-\int_{\R^2}F^{R,\bar{\delta}}(u)dx
\end{equation}
over a suitable subset of $S(a)$,
where $R>0$ is a fixed constant determined later and $F^{R,{\bar\delta}}(t)=\int_0^tf^{R,{\bar\delta}}(s)ds$ for all $t\in\R$.
Furthermore, it could simply contemplate that if a couple of weak solution
$(\bar{u}_0^R,\bar{\lambda}_0^R)$ of problems \eqref{mainequation2} and \eqref{mainequation1a}
satisfying $|\bar{u}_0^R|_\infty\leq R$, then it is a
couple of weak solution of problems \eqref{mainequation1}-\eqref{mainequation1a}.
Have this in mind, we invite the reader to acquaint that one should
construct such a couple of weak solution
$(\bar{u}_0^R,\bar{\lambda}_0^R)$
to conclude the proof of Theorem \ref{maintheorem3}. As a consequence, it is necessary to get the following result.

 \begin{theorem}\label{maintheorem4}
 Suppose that the nonlinearity $f$ defined in \eqref{form}
requires $(h_1)-(h_4)$, then for every fixed $R>0$,
there is a small $\bar{a}^*_R>0$ such that,
for each $\tau>2$, there exists a $\bar{\alpha}^*_0=\bar{\alpha}_0(\tau)>0$ such that
problems \eqref{mainequation2} and \eqref{mainequation1a}  admit a couple of weak solution
$(\bar{u}_0^R,\bar{\lambda}_0^R)$ for every fixed $\bar{\alpha}_0 \in (0, \bar{\alpha}_0^*)$ and $a\in(0,\bar{a}^*_R]$.
Moreover, if in addition
\begin{description}
 \item[$(h_5)$] there are some constants $\xi>0$ and $p>4$ such that $h(t)\geq \xi t^{p-1}$ for all $t\in[0,+\infty)$,
\end{description}
then there is a small $\bar{a}_*^R>0$ such that,
given a
$\bar{\alpha}_0>0$, there exist $\tau_*=\tau(\bar{\alpha}_0)>2$ and ${\xi}_0^R>0$
such that problems \eqref{mainequation2} and \eqref{mainequation1a} have a couple of weak solution $(\bar{u}_0^R,\bar{\lambda}_0^R)$
 for every fixed $\tau \in [2,\tau_*)$,
  ${\xi}>{\xi}_0^R$ and $a\in(0,\bar{a}_*^R]$.
\end{theorem}

Again the results in Theorem \ref{maintheorem3} and \ref{maintheorem4} are new for the normalized solutions under the
Chern-Simons-Schr\"{o}dinger system setting. As described above, one of the key ingredients in the proof of Theorem \ref{maintheorem3}
is to derive the $L^\infty$-estimate for $\bar{u}_0^R$. It is widely known that
both the elliptic regular theory and Nash-Moser iteration procedure are the effective tools
and we shall argue as \cite{Shen2} by using the former one to reach the aim.
Alternatively, we have to exhibit some nontrivial calculations to show that
$(A_0+A_1^2+A_2^2)\in L^\infty(\R^2)$ when $\bar{u}_0^R\in H^1(\R^2)$, which differs from
the counterparts in \cite{Shen2}.

The outline of the present paper is organized as follows. In Section \ref{Section2}, we introduce the variational settings
and preliminaries. In Sections \ref{Section3} and \ref{Section4}, we will exhibit the proofs
of Theorems \ref{maintheorem1} and \ref{maintheorem2}, respectively. Section \ref{Section5}
is mainly devoted to the supercritical exponential case for problems \eqref{mainequation1}-\eqref{mainequation1a}.
\\\\
 \textbf{Notations.} From now on in this paper, otherwise mentioned, we use the following notations:
\begin{itemize}
	\item   $C,C_1,C_2,\bar{C}_1,\bar{C}_2,\cdots$ denote any positive constant, whose value is not relevant.
	 	\item      Let $(X,\|\cdot\|_X)$ be a Banach
space with dual space $(X^{-1},\|\cdot\|_{X^{-1}})$.
 \item  $|\,\,\,|_p$ denotes the usual norm of the Lebesgue space $L^{p}(\mathbb{R}^2)$, for every $p \in [1,+\infty]$,
	$\Vert\,\,\,\Vert_{H^{i}(\R^2)}$ denotes the usual norm of the Sobolev space $H^{i}(\mathbb{R}^2)$ for $i=1,2$.
\item $o_{n}(1)$ denotes a real sequence with $o_{n}(1)\to 0$ as $n \to +\infty$
and $\R^+\triangleq[0,+\infty)$.
\item ``$\to$" and ``$\rightharpoonup$" stand for the strong and weak convergence in the related function spaces, respectively.
\item The tangent space $S(a)$ at $u\in H^1(\R^2)$ is defined as
\[
\mathbb{T}_u=\left\{v\in H^1(\R^2):\int_{\R^2}uvdx=0\right\}.
\]
\item The norm of the $C^1$
restriction functional $E|_{S(a)}^\prime$ at $u\in H^1(\R^2)$ is defined by
\[
\|E|_{S(a)}^\prime\|_{ H^{-1}(\R^2) }=\sup_{v\in \mathbb{T}_u,\|v\|_{H^1(\R^2)}=1}E^\prime(u)[v].
\]
\end{itemize}

\section{Preliminary results}\label{Section2}

In this section, we are going to exhibit some preliminary results adopted to prove the main results.
To begin with, we shall give some useful observations. Owing to
 the second equation and the last two equations in \eqref{mainequation1}, for all $u\in
H^1(\mathbb{R}^2)$, one has
\begin{equation}\label{gauge0}
\begin{gathered}
\int_{\mathbb{R}^2}A_0|u|^2dx  =2\int_{\mathbb{R}^2}A_0(\partial_2A_1-\partial_1A_2)dx   \hfill\\
 \ \ \ \   =2\int_{\mathbb{R}^2}(A_2\partial_1A_0-A_1\partial_2A_0)dx
 =2\int_{\mathbb{R}^2}(A_1^2+A_2^2)|u|^2dx.\hfill\\
\end{gathered}
\end{equation}

It follows from
 the well-known Hardy-Littlewood-Sobolev inequality \cite[Theorem 4.3]{LM}, we can derive
the following estimates to the gauge fields $A_j$ for $j = 0, 1, 2$.

\begin{lemma}\label{gauge}
(see \cite[Propositions 4.2-4.3]{Huh2}) Assume $1<r<2$ and $\frac{1}{r}
-\frac{1}{\widehat{r}}=\frac{1}{2}$, then
\[
|A_j|_{\widehat{r}}\leq C_r|u|_{2r}^2~ \text{for}~ j=1,2,~
|A_0|_{\widehat{r}}\leq C_r|u|_{2r}^2|u|_4^2,
\]
where $C_r>0$ is a constant dependent of $r$.
\end{lemma}

With Lemma \ref{gauge} in hand, one can easily conclude that
\begin{equation}\label{gauge1}
  |A_ju|_2\leq |A_j|_{\widehat{r}}|u|_{\frac{r}{r-1}}
\leq C_r|u|_{2r}^2|u|_{\frac{r}{r-1}}\leq \bar{C}_r\|u\|_{H^1(\R^2)}^3,
~ \text{for}~ j=1,2,
\end{equation}
because $2r>2$ and $r/(r-1)>2$, where $\bar{C}_r>0$  depends only on  $r>1$.
We also need the following Br\'{e}zis-Lieb type lemma for the Chern-Simons term.
\begin{lemma}\label{BLCS}
(see \cite[Lemma 2.4]{GZ})
If $u_n\rightharpoonup u$ in $H^1(\R^2)$ and $u_n\to u$ a.e. in $\R^2$, then one has
$A_j [u_n]\to A_j[u]$ a.e. for $j=1,2$,
\begin{equation}\label{BLCS1}
 \left\{
   \begin{array}{ll}
  \displaystyle    \lim_{n\to\infty}\int_{\R^2}A_0[u_n]u_n\psi dx=\int_{\R^2}A_0[u ]u \psi dx,~\forall \psi\in H^1(\R^2),\\
   \displaystyle      \lim_{n\to\infty}\int_{\R^2}A_j^2[u_n]u_n\psi dx=\int_{\R^2}A_j^2[u ]u \psi dx,
   ~\forall \psi\in H^1(\R^2)~\text{\emph{with}}~j=1,2,
   \end{array}
 \right.
\end{equation}
 and
\begin{equation}\label{BLCS2}
\lim_{n\to\infty}\int_{\R^2}\big[A_j^2[u_n]|u_n|^2-A_j^2[u_n-u]|u_n-u|^2\big]dx
=\int_{\R^2}A_j^2[u]|u|^2dx,~\text{for}~j=1,2.
\end{equation}
\end{lemma}

Caused by the presence of (sub)critical exponential growth at infinity associated with the nonlinearity $f$,
  let us introduce the following type of Trudinger-Moser inequality due to Cao \cite{Cao}.
\begin{lemma}\label{TM}
For any $u\in H^1(\R^2)$ and for all $\alpha>0$, there holds
 \begin{equation}\label{TM1}
 \int_{\R^2}(e^{\alpha|u|^2}-1)dx<+\infty.
 \end{equation}
Moreover, if
$|\nabla u|_2\leq 1$, $|u|_2\leq \bar{M}$ for some $\bar{M}\in(0,+\infty)$ and $\alpha<4\pi$, then there exists a constant $C=C(\bar{M},\alpha)>0$ such that
\begin{equation}\label{TM2}
\int_{\R^2}(e^{\alpha|u|^2}-1)dx\leq C.
\end{equation}
\end{lemma}

Combining \eqref{definitionsubcriticalgrowth} and $(f_1)$,
for all $\varepsilon>0$ and $\alpha>0$, there is a constant $C_\varepsilon>0$ such that
\begin{equation}\label{growth1}
  |f(s)|\leq \varepsilon |s|^{\chi-1}+C_\varepsilon|s|^{q-1}(e^{\alpha |s|^2}-1),~\forall s\in\R,
\end{equation}
where $q>2$ can be arbitrarily chosen later. Using $(f_2)$, there holds
\begin{equation}\label{growth2}
  |F(s)|\leq \varepsilon |s|^{\chi }+C_\varepsilon|s|^{q }(e^{\alpha |s|^2}-1),~\forall s\in\R.
\end{equation}
Similarly, if the nonlinearity $f$ has a critical exponential growth at infinity with the critical exponent $\alpha_0$
appearing in \eqref{definitioncriticalgrowth}. Then fix $q>2$ as above,
for all $\varepsilon>0$ and $\alpha>\alpha_0$, there is a constant $C_\varepsilon>0$ such that
\begin{equation}\label{growth3}
  |f(s)|\leq \varepsilon |s|^{\chi-1}+C_\varepsilon|s|^{q-1}(e^{\alpha |s|^2}-1),~\forall s\in\R,
\end{equation}
and
\begin{equation}\label{growth4}
  |F(s)|\leq \varepsilon |s|^{\chi }+C_\varepsilon|s|^{q }(e^{\alpha |s|^2}-1),~\forall s\in\R.
\end{equation}

Now, recallng the imbedding $H^1(\R^2)\hookrightarrow L^p(\R^2)$ is continuous for all $2\leq p<+\infty$
and adopting \eqref{gauge1}, we can apply \eqref{growth2} and \eqref{growth4}
in \eqref{TM1} to verify that the energy functional $E$, defined in \eqref{functional},
is of $\mathcal{C}^1$ class on $S(a)$ whose derivative can be computed as
\[
E^\prime(u)[v]=
\int_{\R^2} \big[\nabla u \nabla v
+(A_1^2+A_2^2+A_0)uv\big]dx - \int_{\R^2}f( u)vdx
\]
for any $u,v\in S(a)$. In particular, with the help of \eqref{gauge0}, one has
\[
E^\prime(u)[u]= \int_{\R^2}\big[|\nabla u|^2 +3(A_1^2+A_2^2)|u|^2\big]dx
- \int_{\R^2}f(u)udx.
\]
Therefore if $(u,\lambda)\in S(a)\times\R$ is a critical point of $E$, then the
quintuplet $\big(u,A_0[u],A_1[u],A_2[u],\lambda\big)$ is a (weak)
 solution of system \eqref{mainequation1}.

Next, we shall introduce the so-called Poho\u{z}aev identity for Eq. \eqref{mainequation1}
which is more general but easier than its counterpart in \cite[Propostion 2.3]{Byeon}.

\begin{lemma}\label{Pohozaev}
Suppose that $f$ satisfies \eqref{definitionsubcriticalgrowth} and $(f_1)$, or
\eqref{definitioncriticalgrowth} and $(f_1)$. Let $u\in H^1(\R^2)$ be a
nontrivial weak solution of Eq. \eqref{mainequation1} with some suitable $\lambda>0$,
then it satisfies the so-called Poho\u{z}aev identity below
\begin{equation}\label{Pohozaev1}
 \lambda\int_{\R^2}|u|^2dx
 +2\int_{\R^2}(A_1^2+A_2^2)|u|^2dx-2\int_{\R^2}F(u)dx
=0.
\end{equation}
\end{lemma}

\begin{proof}
We postpone its
 detailed proof in Lemma A.1 in the Appendix below.
\end{proof}

As a direct consequence of Lemma \ref{Pohozaev}, we know that each nontrivial weak solution of Eq. \eqref{mainequation1}
is contained in the set $\mathcal{M}(a)$. Indeed, let $u$ be a test function on Eq. \eqref{mainequation1}, there holds
\begin{equation}\label{Nehari}
\int_{\R^2}\big[|\nabla u|^2 +3(A_1^2+A_2^2)|u|^2\big]dx+ \lambda\int_{\R^2}|u|^2dx
- \int_{\R^2}f(u)udx
=0.
\end{equation}
One immediately derives $J(u)=0$ if \eqref{Nehari} minuses \eqref{Pohozaev1}. In other words, the set $\mathcal{M}(a)$
is a natural constraint to find normalized ground state solutions for problems \eqref{mainequation1}
and \eqref{mainequation1a}.

In the sequel, we mainly focus on establishing some properties for $\mathcal{M}(a)$.
Because the subcritical case for the nonlinearity $f$ is much simpler than the critical one,
let us only consider that $f$ possesses the critical exponential growth at infinity.
From now on until the end of this section, without loss of generality,
 we would always suppose that $f$
satisfies \eqref{definitioncriticalgrowth} and $(f_1)-(f_3)$ for simplicity.

\begin{lemma}\label{unique} Let $a\neq0$.
For every $u\in S(a)$, there exists a unique $t_u>0$ such that $t_uu(t_u\cdot)\in \mathcal{M}(a)$
and $E(u_{t_u})=\max\limits_{t>0}E(u_t)$, where $u_t=tu(t\cdot)$ for all $t>0$.
Moreover, the map $S(a)\to (0,+\infty)$ defined by $u\mapsto t_u$ is continuous.
\end{lemma}

\begin{proof}
Let $u\in S(a)$ be fixed and define the function
 $\varsigma(t)=E(u_{t})$ for any $t>0$. Since $u_t\in S(a)$, we simply have that
\begin{align*}
\varsigma^\prime(t)=0  & \Longleftrightarrow t\int_{\R^2}[|\nabla u|^2 + (A_1^2+A_2^2)|u|^2]dx-t^{-3}\int_{\R^2}[f(tu)tu-2F(tu)]dx=0\\
& \Longleftrightarrow \frac1t J(u_t)=0  \Longleftrightarrow  J(u_t)=0  \Longleftrightarrow u_t\in \mathcal{M}(a).
\end{align*}
We claim that $\varsigma(t)>0$ for some sufficiently small $t>0$ and $\lim\limits_{t\to+\infty}\varsigma(t)=-\infty$.
Firstly, let us recall the celebrated Gagliardo-Nirenberg inequality
\begin{equation}\label{GN}
  |u|_p^p\leq C_{p}|\nabla u|_2^{p-2}|u|_2^2,~\forall u\in H^1(\R^2)~\text{and}~2<p<+\infty.
\end{equation}
Let $t_1>0$ small enough satisfy $t_1^2|\nabla u|_2^2<\frac{\pi}{\alpha}$ with $\alpha>\alpha_0$ appearing in \eqref{growth4},
then we can define $\bar{u}=\sqrt{\frac{\alpha}{\pi}}t_1u$ which shows that $|\nabla \bar{u}|_2^2\leq1$
and $|\bar{u}|_2^2=\frac{\alpha}{\pi}t_1^2a^2<+\infty$, using \eqref{TM2} to obtain
\begin{equation}\label{unique1}
\int_{\R^2} (e^{2\alpha t_1^2|u|^2}-1)dx
=\int_{\R^2} (e^{2\pi  |\bar{u}|^2}-1)dx\leq C<+\infty
\end{equation}
which together with the Gagliardo-Nirenberg inequality \eqref{GN} implies that
\begin{align*}
 \varsigma(t_1) & \geq \frac{t^2_1}2\int_{\R^2} |\nabla u|^2dx-\varepsilon  t_1^{\chi-2 }\int_{\R^2} |u|^\chi dx
-C_\varepsilon  t_1^{q-2}\int_{\R^2}|u|^{q }(e^{\alpha t_1^2|u|^2}-1)dx\\
   & \geq\frac{t_1^2}2\int_{\R^2} |\nabla u|^2dx
-\varepsilon a^2t_1^{\chi-2 }\bigg(
\int_{\R^2}|\nabla u|^{2}dx\bigg)^{\frac{\chi-2}{2}}-\bar{C}_\varepsilon
at_1^{q-2}\bigg(\int_{\R^2}|\nabla u|^{2}dx\bigg)^{\frac{q-1}{2}},
\end{align*}
where we need to require $q>4$ in \eqref{growth4} in this situation. Therefore, we choose $\varepsilon=\frac1{4a^2}$
if $\chi=4$ and $\varepsilon=1$ if $\chi>4$ to see there is a sufficiently small $t_2\in(0,t_1)$ such that $\varsigma(t_2)>0$.
According to $(f_2)$, it is very clear to show that $\varsigma(t)\to-\infty$ as $t\to\infty$ since $\theta>4$. Thanks to the claim,
it permits us to find a such $t_u>0$ and its uniqueness follows directly from $(f_3)$.

Invoking from the above arguments, the map $u\mapsto t_u$ is well-defined.
Suppose that $\{u_n\}\subset S(a)$ is a sequence satisfying $u_n\to u_0$ in $H^1(\R^2)$,
then we only need to conclude that $t_{u_n}\to t_{u_0}$ in $(0,+\infty)$ along a subsequence,
where $t_{u_n}u_n(t_{u_n}\cdot), t_{u_0}u_0(t_{u_0}\cdot)\in \mathcal{M}(a)$.
Note that $t_{u_0}$ exists since $|u_0|_2^2=a^2>0$ which shows that $u_0\neq0$.
On the one hand, we claim that $\{t_n\}$ is uniformly bounded from above. Indeed,
since $J(t_{u_n}u_n(t_{u_n}\cdot))=0$ and $(f_2)$, there holds
\[
\int_{\R^2}[|\nabla u_n|^2 + (A_1^2+A_2^2)|u_n|^2]dx= \int_{\R^2}\left[\frac{f(t_nu_n)u_n-2F(t_nu_n)}{(t_nu_n)^4}\right]|u_n|^4dx
\geq (\theta-2)\int_{\R^2} \frac{ F(t_nu_n)}{(t_nu_n)^4} |u_n|^4dx.
\]
Using $(f_2)$ again which implies that $\lim\limits_{s\to+\infty}F(s)/s^4=+\infty$,
then the right hand hand goes to $+\infty$ as $n\to\infty$.
However, the left hand can be controlled by a universal constant independent of $n\in \mathbb{N}^+$
by \eqref{gauge1}. Thereby, the claim concludes.
On the other hand, we claim that $\{t_n\}$ is uniformly bounded from below by a positive constant.
Otherwise, we should suppose that $t_{u_n}\to0$ up to a subsequence. By
exploiting a very similar calculations in \eqref{unique1},
we can certify that
\[
\lim_{n\to\infty}\int_{\R^2}[|\nabla u_n|^2 + (A_1^2+A_2^2)|u_n|^2]dx=0
\]
which is impossible and the claim is true. So, passing to a subsequence if necessary, there is $t_0>0$
such that $t_{u_n}\to t_0$ as $n\to\infty$. Owing to the Lemma A.2 in the Appendix below,
we immediately conclude that $J(u_{t_0})=0$ by $J(t_{u_n}u_n(t_{u_n}\cdot))=0$.
This actually indicates that $t_{0}u_0(t_{0}\cdot)\in \mathcal{M}(a)$. Recalling the above fact that
$t_{u_0}u_0(t_{u_0}\cdot)\in \mathcal{M}(a)$, we must have $t_0=t_{u_0}$
determined by
the uniqueness of the map $u_0\mapsto t_{u_0}$ which means that $t_{u_n}\to t_{u_0}$ as $n\to\infty$. The proof is completed.
 \end{proof}

 \begin{lemma}\label{positive} Let $a\neq0$. Then, for all $u\in \mathcal{M}(a)$,
there is a constant $\varrho>0$ such that
$|\nabla u|_2^2\geq\varrho$. In particular, it holds that
the minimization $m(a)>0$.
\end{lemma}

\begin{proof}
Arguing it indirectly, we suppose that there is a sequence $\{u_n\}\subset \mathcal{M}(a)$ such that
$|\nabla u_n|_2^2\to0$ as $n\to\infty$. Clearly, without loss of generality, we can assume that
$\sup_{n\in \mathbb{N}^+}|\nabla u_n|_2^2\leq K$ with some $K\in(0,\frac{2\pi}{\alpha})$,
where $\alpha>\alpha_0$ comes from \eqref{growth3}. Denoting $\bar{u}_n=K^{-\frac12}u_n$,
then $\sup_{n\in \mathbb{N}^+}|\nabla \bar{u}_n|_2^2\leq1$ and
$|\bar{u}_n|^{2}_2=K^{-1}a^2<+\infty$. Adopting \eqref{TM2} with $2K\alpha<4\pi$, we obtain
\[
\sup_{n\in \mathbb{N}^+}\int_{\R^2}(e^{2\alpha|u_n|^2}-1) dx=
\sup_{n\in \mathbb{N}^+}\int_{\R^2}(e^{2K\alpha|\bar{u}_n|^2}-1) dx\leq C<+\infty.
\]
Choosing $q>3$ in \eqref{growth3} and using the Gagliardo-Nirenberg inequality \eqref{GN}, one has
\begin{align}\label{positive1}
\nonumber \int_{\R^2}f(u_n)u_ndx &\leq \varepsilon\int_{\R^2}|u_n|^\chi dx+C_\varepsilon\int_{\R^2}|u_n|^q(e^{\alpha|u_n|^2}-1) dx \\
\nonumber   & \leq  \varepsilon a^2\bigg(\int_{\R^2}|\nabla u_n|^2 dx\bigg)^{\frac{\chi-2}{2}}+
  C_\varepsilon\bigg(\int_{\R^2}|u_n|^{2q}dx\bigg)^{\frac12}\bigg( \int_{\R^2}(e^{2\alpha|u_n|^2}-1) dx\bigg)^{\frac12}\\
  & \leq  \varepsilon a^2\bigg(\int_{\R^2}|\nabla u_n|^2 dx\bigg)^{\frac{\chi-2}{2}}+
  \bar{C}_\varepsilon a\bigg(\int_{\R^2}|\nabla u_n|^{2}dx\bigg)^{\frac{q-1}2}.
\end{align}
Since $\{u_n\}\subset \mathcal{M}(a)$ which gives that $J(u_n)=0$, one sees that
\[
\int_{\R^2}|\nabla u_n|^2 dx\leq  \varepsilon a^2\bigg(\int_{\R^2}|\nabla u_n|^2 dx\bigg)^{\frac{\chi-2}{2}}+
  \bar{C}_\varepsilon a\bigg(\int_{\R^2}|\nabla u_n|^{2}dx\bigg)^{\frac{q-1}2}
\]
which is absurd when $\chi>4$ and $q>3$. It still remains valid when $\chi=4$
by choosing $\varepsilon=\frac1{2a^2}>0$. So, we derive the first part of this lemma.
According to $(f_2)$, for all $u\in \mathcal{M}(a)$,
\[
E(u)=E(u)-\frac1{\theta-2}J(u)\geq\frac{\theta-4}{2(\theta-2)}\int_{\R^2}|\nabla u|^2 dx
\]
showing the desired result. The proof is completed.
\end{proof}

\begin{lemma}\label{bounded}
Let $\{u_n\}\subset \mathcal{M}(a)$ be a minimizing sequence of $m(a)$, then $\{u_n\}$ is uniformly
 bounded in $H^1(\R^2)$. Moreover, we further
have that
\begin{equation}\label{bounded1}
\left\{\int_{\R^2}F(u_n)dx\right\}~\text{and}~\left\{\int_{\R^2}f(u_n)u_ndx\right\}~\text{are uniformly bounded in}~n\in \mathbb{N}^+.
\end{equation}
\end{lemma}

\begin{proof}
Since $J(u_n)=0$ for all $\{u_n\}\subset \mathcal{M}(a)$, we conclude from $(f_2)$ that
\begin{align*}
  m(a)+o_n(1) &=E(u_n)=E(u_n)-\frac{1}{\theta-2}J(u_n) \\
 & = \frac{\theta-4}{2(\theta-2)}\int_{\R^2}[|\nabla u_n|^2+(A_1^2+A_2^2)|u_n|^2]dx+\frac{1}{\theta-2}\int_{\R^{2}}\left[f(u_n)u_n-\theta F(u_n)\right]dx\\
 &\geq \frac{\theta-4}{2(\theta-2)}\int_{\R^2}[|\nabla u_n|^2+(A_1^2+A_2^2)|u_n|^2]dx
\end{align*}
and
\[
  m(a)+o_n(1) =E(u_n)-\frac{1}{2}J(u_n)  =  \frac{1}{2}\int_{\R^{2}}\left[f(u_n)u_n-4 F(u_n)\right]dx
  \geq \frac{\theta-4}{2}\int_{\R^{2}} F(u_n) dx.
\]
Since $F(s)\geq0$ for all $s\in\R$, exploiting the definition of $J$,
we can finish the proof of this lemma.
\end{proof}

Next, we are going to show that any minimizer of the minimization problem $m(a)=\inf_{u\in \mathcal{M}(a)}E(u)$
is a solution of Eq. \eqref{mainequation1} with a suitable $\lambda\in\R$.
In other words, we prove that each critical point of $E|_{\mathcal{M}(a)}$ is indeed a critical point of
$E|_{S(a)}$. Taking this purpose into account, we need to introduce the following result.

\begin{lemma}\label{CHANG}(see \cite[Corollary 4.1.2]{Chang})
 Let $X$ be a real Banach space and $U\subset X$ be an open set. Suppose that $\tilde{h},g_1,\cdots,g_m:U\to\mathbb{R}^1$ are
 $C^1$ functions and $x_0\in \mathcal{M}$ satisfying $\tilde{h}(x_0)=\inf\limits_{x\in \mathcal{M}}\tilde{h}(x)$ with
$$
\mathcal{M}=\{x\in U: g_i(x)=0,~i=1,2,\cdots,m\}.
$$
If
$\{g_i^{\prime}(x_0)\}_{i=1}^m$ is linearly independent, then there exist $\lambda_1,\cdots,\lambda_m\in \R$ such that
$$
\tilde{h}^{\prime}(x_0)+\sum_{i=1}^m\lambda_ig_i^{\prime}(x_0)=0.
$$
\end{lemma}

\begin{lemma}\label{mnifold}
 Let $a\neq0$, then $\mathcal{M}(a)$ is a $\mathcal{C}^1$-manifold of codimension 2 in $H^1(\R^2)$.
 Moreover, it is a $\mathcal{C}^1$-manifold of codimension 1 in $S(a)$.
\end{lemma}

\begin{proof}
Given a $u\in \mathcal{M}(a)$, one sees that $J(u)=0$ and $P(u)\triangleq\int_{\R^2}u^2dx-a^2=0$.
Since $f\in \mathcal{C}^1$, we see that $J$ and $P$ are of $C^1$ class.
There exist the following two possibilities: either (i) $J$ and $P$ are linearly dependent; or (ii)
$J$ and $P$ are linearly independent.

If (i) holds true, then for all fixed $u\in \mathcal{M}(a)$, it means that there is a constant $\lambda^*\in\R$ satisfying
\[
2\int_{\R^2}[ \nabla u\nabla\psi  + (A_0+A_1^2+A_2^2)u\psi]dx+2\lambda^*\int_{\R^2}u\psi dx
=\int_{\R^2}[f^\prime(u)u-f(u)]\psi dx
\]
for all $\psi\in H^1(\R^2)$,
that is,
\[ \left\{
  \begin{array}{ll}
\displaystyle    -\Delta u +\lambda^* u+A_0 u+\sum\limits_{j=1}^2A_j^2 u=\frac12 [f^\prime(u)u-f(u)], \\
 \displaystyle     \partial_1A_2-\partial_2A_1=-\frac{1}{2}|u|^2,~\partial_1A_1+\partial_2A_2=0,\\
 \displaystyle    \partial_1A_0=A_2|u|^2,~ \partial_2A_0=-A_1|u|^2,
  \end{array}
\right.
\]
Using Lemma \ref{Pohozaev}, we deduce that $u$ must satisfy the equality
\[
 \int_{\R^2}[|\nabla u|^2 + (A_1^2+A_2^2)|u|^2]dx=\frac12\int_{\R^2}[f^\prime(u)u^2-3f(u)u+4F(u)]dx.
\]
Recalling that $J(u)=0$ for $u\in \mathcal{M}(a)$, using the above equality, there holds
\[
\int_{\R^2}[f^\prime(u)u^2-5f(u)u+8F(u)]dx=0
\]
which is impossible by Lemma A.3 in the Appendix below. So, we have that (i) cannot occur.

When (ii) holds true, then by Lemma \ref{CHANG}, there are $\lambda_1,\lambda_2\in\R$ such that
\[
E^\prime(u)+\lambda_1J^\prime(u)+\lambda_2u=0~\text{in}~H^{-1}(\R^2).
\]
To derive the proof of this lemma, we shall show that $\lambda_1=0$. First of all, we know that
\[\begin{gathered}
(1+2\lambda_1)\int_{\R^2}[ \nabla u\nabla\psi  + (A_0+A_1^2+A_2^2)u\psi]dx+ \lambda_2\int_{\R^2}u\psi dx\hfill\\
\ \ \ \  \ \ \ \ =\int_{\R^2}[\lambda_1f^\prime(u)u-(\lambda_1-1)f(u)]\psi dx\hfill\\
\end{gathered}\]
for all $\psi\in H^1(\R^2)$, that is,
\[ \left\{
  \begin{array}{ll}
\displaystyle    -(1+2\lambda_1)\Delta u +\lambda_2 u+(1+2\lambda_1)A_0 u+(1+2\lambda_1)
\sum\limits_{j=1}^2A_j^2 u=\lambda_1f^\prime(u)u-(\lambda_1-1)f(u), \\
 \displaystyle     \partial_1A_2-\partial_2A_1=-\frac{1}{2}|u|^2,~\partial_1A_1+\partial_2A_2=0,\\
 \displaystyle    \partial_1A_0=A_2|u|^2,~ \partial_2A_0=-A_1|u|^2,
  \end{array}
\right.
\]
It follows from Lemma \ref{Pohozaev} again that
\[
 (1+2\lambda_1)\int_{\R^2}[|\nabla u|^2 + (A_1^2+A_2^2)|u|^2]dx
 = \int_{\R^2}[\lambda_1f^\prime(u)u^2-(3\lambda_1-1)f(u)u+2(2\lambda_1-1)F(u)]dx.
\]
Since $J(u)=0$ for $u\in \mathcal{M}(a)$ combined with the above equality, we have
\[
\lambda_1\int_{\R^2}[f^\prime(u)u^2-5f(u)u+8F(u)]dx=0
\]
jointly with Lemma A.3 below again yields that $\lambda_1=0$. The proof is completed.
\end{proof}

To look for minimizer of $m(a)$, we shall construct a Palais-Smale sequence for $E$ restricted on $S(a)$
at the level $m(a)$ by using the minimax principle based on the
homotopy stable family of compact subsets of $S(a)$.
Due to \cite{Ghoussoub}, we have its definition as follows.

\begin{definition}\label{homotop}
Let $B$ be a closed subset of a set $Y\subset H^1(\R^2)$. We say that a class $\mathcal{G}$
of compact subsets of $Y$ is a \emph{homotopy stable family} with the closed boundary $B$ provided that
\begin{itemize}
  \item every set in $\mathcal{G}$ is contained in $B$;
  \item for any $A\in \mathcal{G}$ and each function $\eta\in \mathcal{C}([0,1]\times Y,Y)$ satisfying $\eta(t,x)=x$ for all
$(t,x)\in (\{0\}\times Y)\cup([0,1]\times B)$, then $\eta(\{1\}\times A)\in \mathcal{G}$.
\end{itemize}
\end{definition}

We borrow some ideas developed in \cite{BartschSoave2,SzulkinWeth} to consider the functional
$\Psi:S(a)\to\R$ defined by
\[
\Psi= E\circ \zeta,
\]
where the map $\zeta:S(a)\to \mathcal{M}(a)$ is defined as $u\mapsto t_uu(t_u\cdot)$
and $t_u>0$ is determined by Lemma \ref{unique}.
Since $u\mapsto t_u$ is a continuous map, motivated by \cite[Proposition 2.9]{SzulkinWeth},
 it allows us to show that $\Psi$ is of $\mathcal{C}^1$ class over $S(a)$
and its derivative is computed as
\begin{equation}\label{derivative}
\Psi^\prime(u)[v]  =E^\prime(\zeta(u))[v_{t_u}]
\end{equation}
for any $u\in S(a)$ and $v\in \mathbb{T}_u$, see Lemma A.4 in the Appendix in detail.

\begin{lemma}\label{PSsequence}
 Let $\mathcal{G}$ be a homotopy stable family of compact subsets of $S(a)$ with closed boundary $B$ and define
\[
m_\mathcal{G}\triangleq\inf_{A\in \mathcal{G}}\max_{u\in A}\Psi(u).
\]
Suppose that $B$ is contained in a connected component of $\mathcal{M}(a)$ and
\[
\max\{\sup\Psi(B),0\}<m_\mathcal{G}<+\infty,
\]
then there is a Palais-Smale sequence $\{u_n\}\subset\mathcal{M}(a)$ for $E$ restricted on $S(a)$
at the level $m_\mathcal{G}$.
\end{lemma}

\begin{proof}
The idea can date bake to \cite{BartschSoave2} and we will exhibit the detailed proof for the convinence
of the reader. According to the definition of $m_\mathcal{G}$,
there is a minimizing sequence $\{A_n\}\subset \mathcal{G}$ such that
\[
\max_{u\in A_n}\Psi(u)=m_\mathcal{G}+o_n(1).
\]
Define a map $\eta:[0,1]\times S(a)\to S(a)$ by $\eta(t,u)\triangleq (1-t+tt_u)u((1-t+tt_u)\cdot)$,
where $t_u>0$ comes from Lemma \ref{unique} and so $\eta\in \mathcal{C}([0,1]\times S(a))$.
Moreover, adopting Lemma \ref{unique} again, one can see that
$\eta(t,u)=u$ for any $(t,u)\in(\{0\}\times S(a))\cup([0,1]\times B)$ since $B\subset \mathcal{M}(a)$ which
yields that $t_u\equiv1$.
In light of
 $\mathcal{G}$ is a homotopy stable family of compact subsets of $S(a)$ with closed boundary $B$,
then $D_n\triangleq \eta(\{1\}\times A_n)=\{u_{t_u}:u\in A_n\}\in \mathcal{G}$ by Definition \ref{homotop}.
 Using Definition \ref{homotop} again, one knows that $D_n\subset B\subset \mathcal{M}(a)$,
and for each $v\in D_n$, there is $u\in A_n$ such that $v=u_{t_u}$.
In view of Lemma \ref{unique}, we have that
$\Psi(v)=\Psi(u_{t_u})=E(\zeta(u))=\Psi(u)$ which shows that
\[
\max_{v\in D_n}\Psi(v)=\max_{u\in A_n}\Psi(u).
\]
Therefore, there is another minimizing sequence $\{D_n\}\subset \mathcal{M}(a)$ such that
 \[
\max_{v\in D_n}\Psi(v)=\max_{u\in A_n}\Psi(u)=m_\mathcal{G}+o_n(1).
\]
Thanks to the equivalent minimax principle \cite[Theorem 3.2]{Ghoussoub},
it permits us to derive a Palais-Smale sequence $\{\bar{u}_n\}\subset S(a)$ for $\Psi$ restricted on $S(a)$
at the level $m_\mathcal{G}$ and $\text{dist}_{H^1(\R^2)}(\bar{u}_n,D_n)=o_n(1)$.

Taking advantage of $\bar{u}_n$, we shall prove that $u_n\triangleq\zeta(\bar{u}_n)\in\mathcal{M}(a)$
is a Palais-Smale sequence for $\Psi$ restricted on $S(a)$
at the level $m_\mathcal{G}$. Since $E(u_n)=E(\zeta(\bar{u}_n))=\Psi(\bar{u}_n)=m_\mathcal{G}+o_n(1)$
and $=m_\mathcal{G}<+\infty$, due to Lemma \ref{bounded}, one obtains that $\{u_n\}$ is uniformly bounded in
$n\in \mathbb{N}^+$. In particular, $\{|\nabla u_n|_2^2\}$ is uniformly bounded from above by a universal constant.
 Exploiting Lemma \ref{positive}, one concludes that $\{|\nabla u_n|_2^2\}$ is uniformly bounded from below by a universal constant.
Then, we
denote $\bar{t}_n\triangleq t_{\bar{u}_n}>0$ such that
$t_{\bar{u}_n}\bar{u}_n (t_{\bar{u}_n} \cdot)\in\mathcal{M}(a)$ and so
\[
\bar{t}_n^2=\frac{|\nabla u_n|_2^2}{|\nabla \bar{u}_n|_2^2}\in[C_1,C_2],
\]
for some positive constants $C_1,C_2$ independent of $n\in \mathbb{N}^+$,
where we have depended on
  Lemma \ref{positive} and $\text{dist}_{H^1(\R^2)}(\bar{u}_n,D_n)=o_n(1)$
with $\{D_n\}\subset \mathcal{M}(a)$.

In view of the definition of $\mathbb{T}_u$, since
\[
\int_{\R^2}uvdx=\int_{\R^2}u_{t_u}v_{t_u}dx
\]
which implies that $v_{t_u}\in \mathbb{T}_{u_{t_u}}$ if, and only if, $v\in \mathbb{T}_u$,
then we know that the map $\mathbb{T}_u\to \mathbb{T}_{u_{t_u}}$, defined by $v\mapsto v_{t_u}$,
is an isomorphism by Lemma \ref{unique} with the inverse $v\mapsto v_{t_u^{-1}}$.
So,
\begin{align*}
  \|E|_{S(a)}^\prime(u_n)\|_{H^{-1}(\R^2)}  &=\sup_{v\in \mathbb{T}_{u_n},\|v\|_{H^1(\R^2)}=1}E^\prime(u_n)[v]
=\sup_{v\in \mathbb{T}_{u_n},\|v\|_{H^1(\R^2)}=1}E^\prime(\zeta(\bar{u}_n))[(v_{\bar{t}_n^{-1}})_{\bar{t}_n}]  \\
    &  =\sup_{v\in \mathbb{T}_{u_n},\|v\|_{H^1(\R^2)}=1}\Psi^\prime ( \bar{u}_n )[ v_{\bar{t}_n^{-1}}]
=\sup_{v_{\bar{t}_n^{-1}}\in \mathbb{T}_{\bar{u}_n},\|v_{\bar{t}_n^{-1}}\|_{H^1(\R^2)}=1}\Psi^\prime ( \bar{u}_n )[ v_{\bar{t}_n^{-1}}],
\end{align*}
where we have made full use of \eqref{derivative} in the third equality as well as
$v_{\bar{t}_n^{-1}}\in \mathbb{T}_{\bar{u}_n}$ if, and only if, $v\in \mathbb{T}_{u_n}$ in the last equality.
In addition, by $\bar{t}_n^2\in[C_1,C_2]$
and $\|v_{\bar{t}_n^{-1}}\|_{H^1(\R^2)}^2=t_n^{-2}|\nabla v|_2^2+|v|_2^2< C<+\infty$,
we derive that
\[
\|E|_{S(a)}^\prime(u_n)\|_{H^{-1}(\R^2)}\leq \bar{C}
\sup_{w\in \mathbb{T}_{\bar{u}_n},\|w\|_{H^1(\R^2)}=1}\Psi^\prime ( \bar{u}_n )[w].
\]
Recalling $\{\bar{u}_n\}\subset S(a)$ is a Palais-Smale sequence for $\Psi$ restricted on $S(a)$
at the level $m_\mathcal{G}$, thereby we see that $\{u_n\}\subset \mathcal{M}(a)$ is a Palais-Smale sequence for $E$ restricted on $S(a)$
at the level $m_\mathcal{G}$.
\end{proof}

As a corollary of Lemma \ref{PSsequence}, we are able to derive the result below.

\begin{lemma}\label{2PSsequence}
Let $a\neq0$. Then, there is a Palais-Smale sequence $\{u_n\}\subset \mathcal{M}(a)$ for $E$ restricted on $S(a)$
at the level $m(a)$.
\end{lemma}

\begin{proof}
Let $\mathcal{G}$ be all singletons in $S(a)$ and $B=\emptyset$, then it is simple to check that $\mathcal{G}$ is a homotopy stable
family of compact subsets of $S(a)$
without boundary. Moreover, we obtain
\[
m_{\mathcal{G}}=\inf_{A\in \mathcal{G}}\max_{u\in A}\Psi(u)=\inf_{u\in S(a)}\max_{t>0}E(u_t).
\]
If we show that $m_{\mathcal{G}}=m(a)$, then we are done by Lemma \ref{PSsequence}.
On the one hand, for each $u\in S(a)$, there is unique $t_u>0$ such that $t_uu(t_u\cdot)\in \mathcal{M}(a)$
and $\max_{t>0}E(u_t)=E(u_{t_u})$ by Lemma \ref{unique}, then
\[
\inf_{u\in S(a)}\max_{t>0}E(u_t)\geq\inf_{u\in \mathcal{M}(a)} E(u).
\]
On the other hand,
we claim that, for all $t>0$, there holds
\[
\varpi(t,s)\triangleq\frac{1-t^2}2[f(s)s-2F(s)]+t^{-2}F(ts)-F(s)\geq0,~\forall s\in\R.
\]
Indeed, it suffices to consider the case $s\geq0$. Using some elementary calculations, one sees that
\begin{align*}
 \frac{\partial}{\partial t}\varpi(t,s) &= t^{-3}[f(ts)ts-2F(ts)]-t[f(s)s-2F(s)]\\
   & =ts^4\left\{  \frac{ f(ts)ts-2F(ts) }{(ts)^4}-\frac{ f(s)s-2F(s) }{s^4}  \right\}\\
&
\left\{
  \begin{array}{ll}
   \geq0, &\text{if}~t\in[1,+\infty), \\
    <0, & \text{if}~t\in(0,1),
  \end{array}
\right.
\end{align*}
where we have adopted $(f_3)$ in the last inequality. Hence, we know that $t\mapsto \varpi(t,s)$
is decreasing in $(0,1)$ and increasing in $(1,+\infty)$ for all $s\in\R$, respectively.
It has that $\varpi(t,s)\geq\min_{t\in(0,+\infty)}\varpi(t,s)=\varpi(1,s)=0$ for all $s\in\R$ and the claim follows.
Due to the claim, for all $u\in S(a)$ and $t>0$,
\[
E(u)-E(u_t)-\frac{1-t^2}{2}J(u)=\int_{\R^2}\left[\frac{1-t^2}2[f(u)u-2F(u)]+t^{-2}F(tu)-F(u)\right]dx\geq0.
\]
Given a $u\in \mathcal{M}(a)$, we immediately conclude that $E(u)\geq \max_{t>0}E(u_t)\geq \inf_{u\in S(a)}\max_{t>0}E(u_t)$
and so
\[
\inf_{u\in \mathcal{M}(a)} E(u)\geq \inf_{u\in S(a)}\max_{t>0}E(u_t).
\]
As a consequence, we must have that $m_{\mathcal{G}}=m(a)$. The proof is completed.
\end{proof}

\begin{lemma}\label{2bounded}
Let $\{u_n\}\subset \mathcal{M}(a)$ be a Palais-Smale sequence for $E$ restricted on $S(a)$
at the level $m(a)$, then there is a uniformly bounded sequence $\{\lambda_n\}\subset \R$ such that
\begin{equation}\label{2bounded1}
   E^\prime(u_n)+\lambda_nu_n=o_n(1)~\emph{\text{in}}~H^{-1}(\R^2).
\end{equation}
\end{lemma}

\begin{proof}
Since $\{u_n\}$ is uniformly bounded in $H^1(\R^2)$ by Lemma \ref{bounded}
and $\|E|^\prime_{S(a)}(u_n)\|_{H^{-1}(\R^2)}\to0$, we then obtain
that $\|E^\prime(u_n)-([E^\prime(u_n)]u_n)u_n\|_{H^{-1}(\R^2)}\to0$ and it means that
\[
\int_{\R^2}[ \nabla u_n\nabla\psi  + (A_0+A_1^2+A_2^2)u_n\psi]dx+ \lambda_n\int_{\R^2}u_n\psi dx
-\int_{\R^2} f (u_n) \psi dx\to0
\]
for all $\psi\in H^1(\R^2)$ and the above formula is the desired result \eqref{2bounded1}, where
\begin{align}\label{lambda0}
\nonumber \lambda_n & \triangleq-\frac1{a^2}\left\{\int_{\R^2}[ |\nabla u_n|^2  + 3( A_1^2+A_2^2)u_n^2]dx
-\int_{\R^2} f (u_n)u_ndx\right\} \\
   & = \frac2{a^2}\left\{\int_{\R^2} |\nabla u_n|^2  dx+3\int_{\R^2} F(u_n)dx
-\int_{\R^2} f (u_n)u_ndx\right\}.
\end{align}
Here, we have exploited $J(u_n)=0$ for all $\{u_n\}\subset \mathcal{M}(a)$ in the second equality.
According to Lemma \ref{bounded}, the sequence $\{\lambda_n\}\subset \R$ is uniformly bounded.
The proof is completed.
\end{proof}

With Lemmas \ref{bounded}
and \ref{2bounded} in hands, passing to a subsequence if necessary, there exist a function $u_0\in H^1(\R^2)$
and a constant $\lambda_0\in\R$ such that
\begin{equation}\label{un}
u_n\rightharpoonup u_0 ~\text{in}~H^1(\R^2),~
u_n\to u_0 ~\text{in}~L^p_{\text{loc}}(\R^2)~\text{for all}~2\leq p<+\infty~\text{and}~
u_n\to u_0~\text{a.e. in}~ \R^2
\end{equation}
and
\begin{equation}\label{lambda}
\lambda_n\to\lambda_0.
\end{equation}

To conclude this section, we introduce the following result.

\begin{lemma}\label{solution}
Let $\{u_n\}\subset \mathcal{M}(a)$ be a Palais-Smale sequence for $E$ restricted on $S(a)$
at the level $m(a)$, up to s subsequence if necessary, there holds
\begin{equation}\label{solution1}
\lim_{n\to\infty}\int_{\R^2}f(u_n)\psi dx=\int_{\R^2}f(u_0)\psi dx,~\forall\psi\in C_0^\infty(\R^2),
\end{equation}
where $u_0\in H^1(\R^2)$ comes from \eqref{un}. In particular, $u_0$ is a solution of \eqref{mainequation1}
with $\lambda=\lambda_0$ in \eqref{lambda}.
\end{lemma}

\begin{proof}
Since $E(u_n)= m(a)+o_n(1)$ and \eqref{2bounded1} jointly with \eqref{solution1},
the proof is standard and we omit it here, see e.g. \cite{Figueiredo}.
\end{proof}

\section{The subcritical case}\label{Section3}

In this section, we give the proof of Theorem \ref{maintheorem1}.
For simplicity, we shall always suppose that the nonlinearity $f$ satisfies \eqref{definitionsubcriticalgrowth}
and $(f_1)-(f_3)$ throughout the present section.

First of all, since $f$ admits the subcritical exponential growth at infinity,
then we have the following Br\'{e}zis-Lieb type lemma associated with $f$.

\begin{lemma}\label{fBrezis-Lieb}
Let $\{u_n\}\in H^1(\R^2)$ satisfies \eqref{un}, going to a subsequence if necessary, we have
\begin{equation}\label{fBrezis-Lieb1}
\lim_{n\to\infty} \int_{\R^2}[F(u_n) -F(u_n-u_0)]dx = \int_{\R^2}F(u_0)dx
\end{equation}
\begin{equation}\label{fBrezis-Lieb2}
\lim_{n\to\infty} \int_{\R^2}[f(u_n)u_n -f(u_n-u_0)(u_n-u_0)]dx = \int_{\R^2}f(u_0)u_0 dx
\end{equation}
\end{lemma}

\begin{proof}
The proof is standard and we omit it here.
\end{proof}

As a direct consequence of Lemma \ref{fBrezis-Lieb}, we immediately derive the results below that play
crucial roles in this section.

\begin{lemma}\label{EBrezis-Lieb}
Let $\{u_n\}\in H^1(\R^2)$ satisfies \eqref{un}, going to a subsequence if necessary, we have
\begin{equation}\label{EBrezis-Lieb1}
  \lim_{n\to\infty} [E(u_n)-E(u_n-u_0)]=E(u_0)
\end{equation}
and
\begin{equation}\label{EBrezis-Lieb2}
\lim_{n\to\infty} [J(u_n)-J(u_n-u_0)]=J(u_0).
\end{equation}
\end{lemma}

\begin{proof}
Since $u_n\rightharpoonup u$ as $n\to\infty$, it infers from the basic property of Hilbert space that
\[\lim_{n\to\infty} \int_{\R^2}[|\nabla u_n|^2 -|\nabla u_n-\nabla u_0|^2]dx = \int_{\R^2}|\nabla u_0 |^2dx\]
which together with \eqref{BLCS1} and \eqref{fBrezis-Lieb1}-\eqref{fBrezis-Lieb2}
yields the desired results. The proof is completed.
\end{proof}

Before showing the detailed proof of Theorem \ref{mainequation1},
we need to certify that the minimization $m(a)$ is non-increasing with respect to $a$.

\begin{lemma}\label{non-increasing}
For all $0<a_1<a_2$, there holds that $m(a_1)\geq m(a_2)$.
\end{lemma}
\begin{proof}
According to the definition of $m(a_1)$, for any $\epsilon>0$,
there is a $u\in \mathcal{M}(a)$ such that
\[
E(u)\leq m(a_1)+\frac\epsilon4.
\]
Let us
consider a cutoff function $\psi\in C_0^\infty(\R^2,[0,1])$ such that $\psi(x)=1$
if $|x|\leq1$ and $\psi(x)=0$
if $|x|\geq2$ and define $u_\varepsilon(x)\triangleq \psi(\varepsilon x)u(x)\in H^1(\R^2)\backslash\{0\}$,
then $u_\varepsilon \to u$ in $H^1(\R^2)$ as $\varepsilon\to0^+$.
In view of Lemma \ref{unique}, there is a unique $t_\varepsilon\triangleq t_{u_\varepsilon}$ such that
$t_{u_\varepsilon} u_\varepsilon(t_{u_\varepsilon}\cdot)\in \mathcal{M}(a)$
and $t_\varepsilon\to t_u\triangleq t_0=1$ as $\varepsilon\to0^+$.

As a consequence, we conclude that $E((u_\varepsilon)_{t_\varepsilon})\to E(u)$ as $\varepsilon\to0^+$ with the help of
 Lemma A.2 in the Appendix below and then one can
fix a sufficiently small $\varepsilon_0>0$ to satisfy
\[
E((u_{\varepsilon_0})_{t_{\varepsilon_0}})\leq E(u)+\frac\epsilon4\leq m(a_1)+\frac 12\epsilon.
\]
Let $v\in C_0^\infty(\mathbb{R})$ satisfy $\supp v \subset B_{1+\frac{4}{\varepsilon_0}}(0)\backslash B_{\frac{4}{\varepsilon_0}}(0)$
and define
\[
v_{\varepsilon_0}=\frac{a_2^2- |u_{\varepsilon_0} |_2^2}{ |v |_2^2}v.
\]
We choose $w_t=u_{\varepsilon_0}+t v_{\varepsilon_0}(t\cdot )=u_{\varepsilon_0}+(v_{\varepsilon_0})_t$ for all $t\in(0,1)$,
then it is simple to calculate that
\[
\text{dist}_{\R^2}(\supp u_{\varepsilon_0},\supp (v_{\varepsilon_0})_t)\geq \frac2{\varepsilon_0}
\left(\frac2t-1\right)>\frac2{\varepsilon_0}>0
\]
and so $|w_t|_2^2=a_2^2$ which means that
$w_t\in S(a_2)$. It simply invokes from Lemma \ref{unique} that there exists a
$t_{w_t}>0$ such that $(w_t)_{t_{w_t}}\triangleq t_{w_t}w_t(t_{w_t}\cdot)\in \mathcal{M}(a_2)$.
We claim that $\{t_{w_t}\}$ is uniformly bounded from above in $t\in(0,1)$.
Indeed, $\|w_t\|^2_{H^1(\R^2)}=\|u_{\varepsilon_0}\|^2_{H^1(\R^2)}+\|(v_{\varepsilon_0})_t\|^2_{H^1(\R^2)}
\leq \|u_{\varepsilon_0}\|^2_{H^1(\R^2)}+\|v_{\varepsilon_0} \|^2_{H^1(\R^2)}<+\infty$, we argue as the proof of Lemma \ref{unique}
to reach the claim. Similarly, $\|w_t\|^2_{H^1(\R^2)}\geq\|u_{\varepsilon_0}\|^2_{H^1(\R^2)}>0$,
we can derive that $\{t_{w_t}\}$ is uniformly bounded from below by a positive constant in $t\in(0,1)$.
All in all, there are constants $T_1,T_2>0$ independent of $t\in(0,1)$ such that
$0<T_1\leq t_{w_t}\leq T_2<+\infty$. So
 \[
\text{dist}_{\R^2}\big(\supp (u_{\varepsilon_0})_{t_{w_t}},\supp ((v_{\varepsilon_0})_t)_{t_{w_t}}\big)\geq \frac2{t_{w_t}\varepsilon_0}
\left(\frac2t-1\right)>\frac{2 }{T_2\varepsilon_0}>0
\]
which indicates that
\[
\left\{
  \begin{array}{ll}
\displaystyle    \int_{\R^2}|\nabla(w_t)_{t_{w_t}}|^2dx= \int_{\R^2}|\nabla (u_{\varepsilon_0})_{t_{w_t}}|^2dx
+\int_{\R^2}|\nabla((v_{\varepsilon_0})_t)_{t_{w_t}}|^2dx,\\
\displaystyle    \int_{\R^2}F((w_t)_{t_{w_t}})dx=\int_{\R^2}F((u_{\varepsilon_0})_{t_{w_t}})dx+\int_{\R^2}F(((v_{\varepsilon_0})_t)_{t_{w_t}})dx,
  \end{array}
\right.
\]
and jointly with \eqref{CSS2e1} and \eqref{CSS2e2} yields that
\[
A_j[(w_t)_{t_{w_t}}]
=A_j[(u_{\varepsilon_0})_{t_{w_t}}]+A_j[((v_{\varepsilon_0})_{t})_{t_{w_t}}],~\text{for}~j=1,2.
\]
According to Lemma A.5 in the Appendix below, there holds
\begin{eqnarray*}
  & & \int_{\R^2}A_j^2[(w_t)_{t_{w_t}}]|(w_t)_{t_{w_t}}|^2dx \\
    &=&
\int_{\R^2}(
A_j^2[(u_{\varepsilon_0})_{t_{w_t}}]+A_j^2[((v_{\varepsilon_0})_{t})_{t_{w_t}}]
+2A_j[(u_{\varepsilon_0})_{t_{w_t}}]A_j[((v_{\varepsilon_0})_{t})_{t_{w_t}}]
)(|(u_{\varepsilon_0})_{t_{w_t}}|^2+|((v_{\varepsilon_0})_t)_{t_{w_t}}|^2)dx \\
&=&\int_{\R^2}(
A_j^2[(u_{\varepsilon_0})_{t_{w_t}}]|(u_{\varepsilon_0})_{t_{w_t}}|^2+A_j^2[((v_{\varepsilon_0})_{t})_{t_{w_t}}]|((v_{\varepsilon_0})_t)_{t_{w_t}}|^2)dx
+o_t(1),
\end{eqnarray*}
where $o_t(1)\to 0$ as $t\to0^+$. Consequently, by $(u_{\varepsilon_0})_{t_{\varepsilon_0}}\in \mathcal{M}(a)$,
we obtain that
\[
E\big((w_t)_{t_{w_t}}\big)=E\big((u_{\varepsilon_0})_{t_{w_t}} \big)
+E\big( ((v_{\varepsilon_0})_{t})_{t_{w_t}} \big)+o_t(1)
\leq E\big((u_{\varepsilon_0})_{t_{\varepsilon_0}} \big)
+E\big( ((v_{\varepsilon_0})_{t})_{t_{w_t}} \big)+o_t(1).
\]
On the other hand, one easily exploits $(f_1)$ to verify that
\[
E(((v_{\varepsilon_0})_{t})_{t_{w_t}})=\frac{t^2t^2_{w_t}}{2}
\int_{\R^2}[|\nabla v_{\varepsilon_0}|^2 + (A_1^2+A_2^2)|v_{\varepsilon_0}|^2]dx
-t^{-2}t^{-2}_{w_t}\int_{\R^2}F(tt_{w_t}v_{\varepsilon_0})dx=o_t(1).
\]
Finally, since $(w_t)_{t_{w_t}}\in \mathcal{M}(a_2)$, we
could fix $t>0$ which is so close to $0^+$ such that
\begin{align*}
   m(a_2) & \leq E((w_t)_{t_{w_t}})
\leq E((u_{\varepsilon_0})_{t_{\varepsilon_0}})+\frac\epsilon2   \leq m(a_1)+ \epsilon
 \end{align*}
finishing the proof since $\epsilon>0$ is arbitrary.
\end{proof}

Now, we are ready to exhibit the proof of Theorem \ref{maintheorem1}.

\begin{proof}[\emph{\textbf{Proof of Theorem \ref{maintheorem1}}}]
According to Lemma \ref{PSsequence},
there is a Palais-Smale sequence for $E$ restricted on $S(a)$
at the level $m(a)$ and we denoted it by $\{u_n\}$. Using Lemmas \ref{bounded}
and \ref{2bounded}, we derive the two sequences appeared in \eqref{lambda}
and \eqref{un}. To show the proof clearly, we split it into three steps.

\textbf{Step 1.} Without loss of generality, we could assume that the weak limit $u_0\neq0$.

Since $\{u_n\}\subset H^1(\R^2)$ is uniformly bounded, for some $\rho>0$, then we have the following two cases
\[
\text{Case I}: \lim_{n\to\infty}\sup_{y\in\R^2}\int_{B_\rho(y)}|u_n|^2dx=0;~
\text{Case II}: \lim_{n\to\infty}\sup_{y\in\R^2}\int_{B_\rho(y)}|u_n|^2dx>0.
\]
If the Case I occurs, then $u_n\to0$ in $L^p(\R^2)$ for all $2<p<+\infty$ by the Vanishing lemma.
Exploiting some simple calculations, we obtain
\[
\lim_{n\to\infty}\int_{\R^{2}} F(u_n) dx=0~\text{and}\lim_{n\to\infty}\int_{\R^{2}}f(u_n)u_n dx=0
\]
which indicate that
\[
m(a)=\lim_{n\to\infty}[E(u_n)-\frac12J(u_n)]= \frac{1}{2}\lim_{n\to\infty}\int_{\R^{2}}\left[f(u_n)u_n-4 F(u_n)\right]dx=0.
\]
It is impossible by Lemma \ref{positive}. So, the Case (II) must occur and there is a sequence $\{y_n\}\subset\R^2$
such that $\lim\limits_{n\to\infty}\int_{B_\rho(0)}|v_n|^2dx>0$, where $v_n=u_n(\cdot-y_n)$.
Recalling both $E$ and $J$ are translation-invariant, then $\{v_n\}\subset \mathcal{M}(a)$ is a Palais-Smale sequence for $E$ restricted on $S(a)$
at the level $m(a)$. Repeating the previous arguments, there is a $v\in H^1(\R^2)$ such that
$v_n\rightharpoonup v$ in $H^1(\R^2)$, $v_n\to v$ in $L_{\text{loc}}^p(\R^2)$ for all $2<p<+\infty$ and $v_n\to v$
a.e. in $\R^2$. Clearly, we deduce that $v\neq0$. Regardless of the abuse of notations,
we adopt the sequence $\{u_n\}$ and $u_0\neq0$. The proof is done.

\textbf{Step 2.} $E(u_0)=m(a)$ and $\lim\limits_{n\to\infty}|\nabla u_n|_2^2=|\nabla u_0|_2^2$.

In view of lemma \ref{solution}, $u_0$ is a solution of \eqref{mainequation1}
with $\lambda=\lambda_0$ in \eqref{lambda} and so $J(u_0)=0$ by Lemma \ref{Pohozaev}.
Since $\{u_n\}\subset \mathcal{M}(a)$ which is $J(u_n)=0$
and then $\lim\limits_{n\to\infty}J(u_n-u_0)=0$ by \eqref{EBrezis-Lieb2}.
Thus, we have
\begin{align}\label{1proof1}
\nonumber E(u_n-u_0)&  =E(u_n-u_0)-\frac12 J(u_n-u_0) +o_n(1)\\
\nonumber   &= \frac{1}{2}\int_{\R^{2}}\left[f(u_n-u_0)(u_n-u_0)-4 F(u_n-u_0)\right]dx +o_n(1)\\
   &  \geq \frac{\theta-4}{2}\int_{\R^{2}} F(u_n-u_0) dx+o_n(1)\geq o_n(1)
\end{align}
It follows from the Fatou's lemma that $|u_0|_2^2\leq \lim\limits_{n\to\infty}|u_n|_2^2\leq a^2$.
In light of $u_0\neq0$ by Step 1, we are derived from \eqref{EBrezis-Lieb1} and Lemma \ref{non-increasing} that
\begin{equation}\label{1proof2}
  \lim_{n\to\infty}E(u_n-u_0)+E(u_0)=\lim_{n\to\infty}E(u_n)=m(a)\leq m(|u_0|_2^2)\leq E(u_0).
\end{equation}
 It follows from \eqref{1proof1} and \eqref{1proof2} that
 \[
\lim_{n\to\infty}\int_{\R^{2}} F(u_n-u_0) dx=0
\]
and so
\[
\lim_{n\to\infty}\int_{\R^{2}} f(u_n-u_0) (u_n-u_0)dx=0.
\]
Combining the above two formulas and $\lim\limits_{n\to\infty}J(u_n-u_0)=0$ again,
we conclude that
\[
\lim_{n\to\infty}\int_{\R^{2}}[ |\nabla u_n-\nabla u_0|^2+(A_1^2[u_n-u_0]+A_2^2[u_n-u_0])|u_n-u_0|^2]dx=0.
\]
Furthermore, we obtain
$\lim\limits_{n\to\infty}E(u_n-u_0)=0$ jointly with \eqref{1proof1} gives that $E(u_0)=m(a)$.

\textbf{Step 3.} There is a small $a^*>0$ such that $\lambda_0>0$ for all $a\in(0,a^*]$.
So, $|u_0|_2^2=a^2$.

 It follows from \eqref{lambda0} and \eqref{lambda} as well the calculations in Step 2 that
\[
\lambda_0=\frac2{a^2}\left\{\int_{\R^2} |\nabla u_0 |^2  dx+3\int_{\R^2} F(u_0 )dx
-\int_{\R^2} f (u_0 )u_0 dx\right\}.
\]
Since $\{u_n\}$ is uniformly bounded in $H^1(\R^2)$, saying it $\sup_{n\in \mathbb{N}^+}\|u_n\|_{H^1(R^2)}^2\leq \bar{K}$,
then one sees that $\|u_0\|_{H^1(R^2)}^2\leq \bar{K}$ by Fatou's lemma. Because $\alpha>0$ is arbitrary in \eqref{growth1},
we could choose it to satisfy $2\bar{K}\alpha<4\pi$. Taking the very similar calculations in \eqref{positive1}, for $q=\frac{2+\chi}{2}\geq3$, we obtain
\[
\int_{\R^2} f (u_0 )u_0 dx \leq \varepsilon a^2\bigg(\int_{\R^2} |\nabla u_0|^2 dx\bigg)^{\frac{\chi-2}{2}}
   +\bar{C}_\varepsilon a \bigg(\int_{\R^2} |\nabla u_0|^{2}dx\bigg)^{\frac{q-1}2}.
\]
Note that $\varepsilon=\frac1{2a^2}$ if $\chi=4$ and $\varepsilon=1$ if $\chi>4$,
then the constants $\bar{C}_{1/2a^2}$ and $\bar{C}_1$ are only dependent of $f$, $\chi$ and $\alpha$.
Let us define
\[
a^*\triangleq\left\{
      \begin{array}{ll}
        \frac1{2\bar{C}_{ 1/2a^2}}, &\text{if}~\chi=4, \\
         \min\left\{\frac1{2\bar{C}_1,\frac1{\sqrt{2}}}\right\}\varrho^{\frac{4-\chi}{2}}, &\text{if}~\chi>4,
      \end{array}
    \right.
\]
As a consequence, for all $a\in(0,a^*]$, we can infer that
\[
\lambda_0\geq\frac6{a^2}  \int_{\R^2} F(u_0 )dx>0.
\]
Finally, combining \eqref{2bounded1}, Lemma \ref{solution} and the calculations in Step 2, we have that
\begin{align*}
\lambda_0 |u_n-u_0 |_2^2 & =\lambda_0 |u_n|_2^2-\lambda_0 |u_0 |_2^2+o_n(1)
  =\lambda_n|u_n|_2^2-\lambda_0 |u_0 |_2^2+o_n(1)\\
&=\int_{\R^2} f(u_n )u_n dx -\int_{\R^2}[|\nabla u_n|^2+(A_1^2+A_2^2)u_n^2]dx
 -\int_{\R^2} f(u_0 )u_0 dx\\
 &\ \ \ \  +\int_{\R^2}[|\nabla u_0 |^2+(A_1^2+A_2^2)u_0 ^2]dx+o_n(1)\\
&=o_n(1)
\end{align*}
showing that $|u_0|_2^2=\lim\limits_{n\to\infty}|u_n|_2^2=a^2$ since $\lambda_0\neq0$. Thanks to Lemma \ref{mnifold},
we finish the proof.
\end{proof}

\section{The critical case}\label{Section4}

In this section, we concentrate on the proof of Theorem \ref{maintheorem2}. Since it is the critical exponential case,
it could be much complex than that of Theorem \ref{maintheorem1}.
Nevertheless, we can still find a Palais-Smale sequence $\{u_n\}\subset S(a)$ at the level $m(a)$.
Thereby, we can exploit the results in Section \ref{Section2} directly.

Alternatively, the nonlinearity $f$ has the critical exponential growth at infinity as well as there is no compact result
for $H^1(\R^2)\hookrightarrow L^p(\R^2)$ with $2<p<+\infty$, thus
 we have to pull the minimization value $m(a)$ defined in
\eqref{minimization} down to some threshold value which is crucial in restoring the compactness.
To the aim, motivated by \cite{AS},
for the sufficiently integer $n\geq2$, we require the constant $R_n\geq\frac{2+\log 2}{2(2-\log 2)}\rho$ to satisfy
\begin{equation}\label{Rn}
a^2  =\frac{\rho^2}{16\log n} \left( 2\log^22+2\log2+1-\frac8{n^2}\log n-\frac4{n^2}\right)
+\frac{(2R_n-\rho)(2R_n+3\rho)\log^22}{48\log n}.
\end{equation}
Here the constant $\rho>0$ is suitably large and meets that
\begin{equation}\label{rho}
\rho> \frac{1}{\alpha_0\bar{\beta}_0} \sqrt{\frac{\pi(\theta-2)}{\theta}}.
\end{equation}
Moreover, it is simple to observe that
\begin{equation}\label{Rn2}
\lim_{n\to\infty}R_n=+\infty,
 \lim_{n\to\infty}\frac{R_n}{\log n}=0 ~\text{and}\lim_{n\to\infty}\frac{R_n^2}{\log n}=\frac{12a^2}{\log^22}.
\end{equation}
By the above discussions, inspired by \cite{AY,DSD,Lu2012}, we define the following Moser type functions
\begin{equation}\label{wn}
 w_n(x)\triangleq \frac{1}{\sqrt{2\pi}}
\left\{
  \begin{array}{ll}
\displaystyle   \sqrt{\log n}, & 0\leq|x|\leq \frac \rho n, \\
\displaystyle      \frac{\log(\rho/|x|)}{\sqrt{\log n}}, &\frac \rho n\leq|x|\leq \frac \rho 2, \\
\displaystyle      \frac{2(R_n-|x|)\log 2}{(2R_n-\rho)\sqrt{\log n}}, & \frac\rho2\leq|x|\leq R_n, \\
\displaystyle     0, & |x|\geq R_n.
  \end{array}
\right.
\end{equation}
Clearly, $\{w_n\}\subset H^1(\R^2)$ and one can calculate standardly that
\begin{align*}
 \int_{\R^2}|w_n|^2dx&=\int_0^{\frac\rho n}  r\log n dr+\int_{\frac \rho n}^{\frac \rho 2} \frac{\log^2(\rho/r)}{\log n} rdr
+\int_{\frac\rho2}^{R_n}  \frac{4(R_n-r)^2\log^22}{(2R_n-\rho)^2\log n}rdr    \\
 &=\frac{\rho^2}{16\log n} \left( 2\log^22+2\log2+1-\frac8{n^2}\log n-\frac4{n^2}\right)
 +\frac{(2R_n-\rho)(2R_n+3\rho)\log^22}{48\log n}\\
&=a^2
\end{align*}
and
\begin{align*}
\displaystyle   \int_{\R^2}|\nabla w_n|^2dx &=\frac1{\log n}\int_{\frac\rho n}^{\frac\rho 2}\frac1rdr+\int_{\frac\rho2}^{R_n}
\frac{(2\log 2)^2}{(2R_n-\rho)^2\log n}rdr
   = 1-\frac{\log2}{\log n } + \frac{(2R_n+\rho)\log^22}{2(2R_n-\rho)\log n}\leq 1.
\end{align*}

Due to the negative impact caused by the Chern-Simons term in the energy functional $E$,
defined by \eqref{functional}, we shall follow the ideas exploited in \cite[Lemma 3.10]{SSY} to establish the following result.

\begin{lemma}\label{estimateChern-Simons}
Let $\{w_n\}\subset H^1(\R^2)$ be defined as \eqref{wn}, then, passing to a subsequence if necessary,
it holds that
\[
\lim_{n\to\infty}\int_{\R^2} (A_1^2[w_n]+A_2^2[w_n])w_n^2 dx=0.
\]
\end{lemma}

\begin{proof}
Recalling the definition of $\{w_n\}$, we are able to see that
\[
\int_{|x|\leq\frac\rho n}|w_n|^4dx=\frac{\log^2n}{(2\pi)^2}\int_{|x|\leq\frac\rho n} dx=\frac{\rho^2\log^2n}{4\pi n^2}
=o_n(1),
\]
and
\begin{align*}
 \int_{\frac\rho n\leq |x|\leq\frac\rho 2}|w_n|^4dx &=\frac{1}{(2\pi)^2\log^2n}\int_{\frac\rho n\leq |x|\leq\frac\rho 2}
\log^4(\rho/|x|)dx= \frac{1}{ 2\pi \log^2n}\int_{\frac\rho n}^{\frac\rho2}\log^4(\rho/r)rdr \\
    & =\frac{\rho^2}{8\pi \log^2n}(2r^4-4r^3+6r^2-6r+4)e^{2r}\bigg|_{\log \frac1n}^{\log \frac12}=o_n(1).
\end{align*}
In view of \eqref{Rn2}, there holds
\[
\begin{gathered}
 \int_{\frac\rho 2\leq |x|\leq R_n}|w_n|^4dx  =\frac{16\log^42}{(2\pi)^2\log^2n}\int_{\frac\rho 2\leq |x|\leq R_n}
 \frac{(R_n-|x|)^4}{(2R_n-\rho)^4}dx= \frac{8\log^42}{ \pi \log^2n}\int_{\frac\rho 2}^{R_n}\frac{(R_n-r)^4}{(2R_n-\rho)^4}rdr \hfill\\
\ \ \ \  =\frac{4\log^42}{ 15\pi \log^2n}\big[15(r-R_n)^4r^2-20(r-R_n)^3r^3+15(r-R_n)^2r^4-6(r-R_n) r^5+r^6\big]\bigg|_{\frac\rho 2}^{R_n}\hfill\\
\ \ \ \  =\frac{4\log^42}{15\pi(2R_n-\rho)^4 \log^2n}\bigg[R_n^6-15\bigg(\frac\rho2-R_n\bigg)^4\bigg(\frac\rho2\bigg)^2+20\bigg(\frac\rho2-R_n\bigg)^3
\bigg(\frac\rho2\bigg)^3\hfill\\
\ \ \ \ \ \ \ \ \ \ \ \  \ \ \ \  \ \ \ \  \ \  \ \  \ \  \ \  \ \  \ \  \ \ \ \   -15\bigg(\frac\rho2-R_n\bigg)^2
\bigg(\frac\rho2\bigg)^4+6\bigg(\frac\rho2-R_n\bigg)\bigg(\frac\rho2\bigg)^5-\bigg(\frac\rho2\bigg)^6\bigg] \hfill\\
\ \ \ \  =\frac{4\log^42}{15\pi (2R_n-\rho)^4\log^2n}\big[R_n^6+o_n(1)\big]=o_n(1). \hfill\\
\end{gathered}
\]
As a consequence of the above calculations, we derive $|w_n|_4^4\to$ as $n\to\infty$.
Since $|w_n|_2^2=a^2$, one will see that $|w_n|_p^p\to0$ for each $2<p\leq 4$ by the interpolation inequality.
Due to \eqref{gauge1}, we can finish the proof of this lemma.
\end{proof}

Then, with \eqref{estimateChern-Simons} in hands, it permits us to take
 the estimate for the minimization value $m(a)$ as follows.

\begin{lemma}\label{estimate}
Assume that $f$ satisfies
\eqref{definitioncriticalgrowth} and $(f_1)-(f_5)$, then there holds
\begin{equation}\label{estimate1}
m(a)<c^*\triangleq\frac{2\pi}{\alpha_0}.
\end{equation}
\end{lemma}

\begin{proof}
Due to the definition of $m(a)$, it is obvious to conclude that $m(a)\leq \inf\limits_{u\in S_a}\max\limits_{t>0}E(u_t)$.
Since $w_n\in S(a)$ for all integer $n\geq2$, then it suffices to show that
$\max\limits_{t>0}E((w_{n_0})_t)<c^*$ for some integer $n_0\geq2$.
Suppose it by a contradiction, we assume that $\max\limits_{t>0}E((w_{n})_t)\geq c^*$ for all integer $n\geq2$.
It follows from some direct calculations in Lemma \ref{unique} that there is a $t_n>0$ such that
\[
J((w_{n})_{t_n})=0~\text{and}~E((w_{n})_{t_n})=\max\limits_{t>0}E((w_{n})_t)\geq c^*.
\]
Combining $|\nabla w_n|_2^2\leq1$ and $J((w_{n})_{t_n})=0$ jointly with
$(f_2)$, one has that
\begin{equation}\label{estimate2}
t_{n}^{4} +t_{n}^{4}\int_{\R^2} (A_1^2[w_n]+A_2^2[w_n])w_n^2 dx\geq  \frac{\theta-2}{\theta} \int_{\mathbb{R}^{2}}
F\left(t_{n} w_{n}\right)  dx.
\end{equation}
Since $F(s)\geq0$ for all $s\in\R$, using $|\nabla w_n|_2^2\leq1$ again and $E((w_{n})_{t_n})\geq c^*$, there holds
\begin{equation}\label{estimate3}
t_n^{2}+t_n^{2} \int_{\R^2} (A_1^2[w_n]+A_2^2[w_n])w_n^2 dx\geq 2c^*,~\text{for all integer}~n\geq2.
\end{equation}
Invoking from $\left(f_{5}\right)$, for all $\varepsilon \in\left(0, \beta_{0}\right)$, there is
a constant $R_{\varepsilon}>0$ such that
$$
F(s)   \geq  \left(\beta_{0}-\varepsilon\right)  e^{\alpha_{0} s^{2}}, \forall s \geq R_{\varepsilon},
$$
which together with \eqref{estimate2} implies that
\begin{equation}\label{estimate4}
\begin{gathered}
t_{n}^{4} +t_{n}^{4}\int_{\R^2} (A_1^2[w_n]+A_2^2[w_n])w_n^2 dx
\geq  \frac{\theta-2}{\theta} F\left(\frac{t_{n} \sqrt{\log n}}{\sqrt{2 \pi}  }\right)  \int_{B_{\frac\rho n}(0)}
d x \hfill\\
\ \ \ \   \geq \frac{\left(\theta-2\right) \left(\bar{\beta}_{0}-\varepsilon\right)}{\theta}
\left(e^{\alpha_{0} t_{n}^{2} (2\pi)^{-1} \log n}\right)
 \left|B_{\frac\rho n}(0)\right|
 =\frac{\pi(\theta-2)\left(\bar{\beta}_{0}-\varepsilon\right)\rho^2}{\theta}
e^{\left[\alpha_{0} t_{n}^{2} (2\pi)^{-1} -2\right]
\log n} .\hfill\\
\end{gathered}
\end{equation}
So, with the help of \eqref{gauge1}, we obtain
\begin{equation}\label{estimate5}
4\log t_n  +\log[\bar{C}_r(a^2+2)] \geq \log\bigg[\frac{\pi(\theta-2)\left(\bar{\beta}_{0}-\varepsilon\right)\rho^2}{\theta}
\bigg]
+[\alpha_0t_n^2(2\pi) ^{-1}  -2]\log n.
\end{equation}
If $\{t_n\}$ is unbounded, up to a subsequence if necessary, we can assume that
$t_n\to+\infty$ and then
\[
\frac{4\log t_n+ \log[\bar{C}_r(a^2+2)]}{t_n^2}  \geq t_n^{-2}\log\bigg[\frac{\pi(\theta-2)\left(\bar{\beta}_{0}-\varepsilon\right)\rho^2}{\theta}
\bigg]
+[\alpha_0  (2\pi) ^{-1} -2t_n^{-2} ]\log n
\]
which yields a contradiction
if we tend $n\to\infty$. Thereby, passing to a subsequence if necessary, there exists a
positive constant $t_0^{2}\geq 2c^*$ determined by Lemma \ref{estimateChern-Simons} and \eqref{estimate3} such that
\[
\lim_{n\to\infty} t_n^{2}=t_0^{2}\geq 2c^*.
\]
 Moreover, we derive
$t_0^{2}= 2c^*$. Otherwise, we arrive at a contradiction by letting
$n\to\infty$ in \eqref{estimate5}. Let us tend $n\to\infty$ and then $\varepsilon\to0^+$ in \eqref{estimate4},
using Lemma \ref{estimateChern-Simons} again, there holds
\[
(2c^*)^2=t_0^{4}\geq  \frac{\pi(\theta-2)\bar{\beta}_{0}\rho^2}{\theta}
\]
which contradicts with \eqref{rho}.
The proof is completed.
\end{proof}

To deal with the lack of compactness, we also need the following type of compact result.

 \begin{lemma}\label{2compact}
Let $f$ meets
\eqref{definitioncriticalgrowth} and $(f_1)-(f_4)$. If the sequence $\{u_n\}\subset H^1(\R^2)$ satisfying
$u_n\to 0$ in $L^p(\R^2)$ for all $2<p<+\infty$, $u_n\to 0$ a.e. in $\R^2$
and
\begin{equation}\label{2compact1}
\sup_{n\in \mathbb{N}}\int_{\R^2} f(u_n)u_ndx\leq C,
\end{equation}
then we have that $F(u_n)\to 0$ in $L^1(\R^2)$ as $n\to\infty$ along a subsequence.
\end{lemma}
\begin{proof}
The proof is standard and its detail is omitted, we refer the reader to \cite{Figueiredo} for example.
\end{proof}

Note that the map $a\mapsto m(a)$ is non-increasing by Lemma \ref{non-increasing}, whereas
taking into account that the nonlinearity $f$ admits the critical exponential growth as the infinity,
we are going to prove the following result which is obviously stronger.

\begin{lemma}\label{decreasing}
Let $\tilde{a}\neq0$ and $\tilde{u}\in \mathcal{M}(\tilde{a})$ be a minimizer of $m(\tilde{a})$.
If there is a $\tilde{\lambda}>0$ such that $\tilde{u}$ is a solution of Eq. \eqref{mainequation1} with $\lambda=\tilde{\lambda}$,
then $m(\tilde{a})$ is strictly decreasing in the right neighborhood of $\tilde{a}$.
\end{lemma}
\begin{proof}
We follow the method exploited in \cite[Proposition 5.1]{GZ}
to exhibit the proof. For each $t,\gamma>0$, define $\tilde{u}_{t,\gamma}\triangleq\gamma t\tilde{u}(t\cdot)$.
With the help of the two functionals $E,J:S(a)\to\R$ before, we define
the following two new functionals
$\Phi_E(t,\gamma)\triangleq E(\tilde{u}_{t,\gamma})$ and $\Phi_J(t,\gamma)\triangleq J(\tilde{u}_{t,\gamma})$
which are given by
\[
\Phi_E(t,\gamma)=\frac{t^2}2\int_{\R^2}[ \gamma^2|\nabla \tilde{u}|^2+
 \gamma^6 (A_1^2[\tilde{u}]+A_2^2[\tilde{u}])|\tilde{u}|^2]dx-
t^{-2}\int_{\R^2}F(\gamma t\tilde{u})dx
\]
and
\[
\Phi_J(t,\gamma)= t^2 \int_{\R^2}[ \gamma^2|\nabla \tilde{u}|^2+
 \gamma^6 (A_1^2[\tilde{u}]+A_2^2[\tilde{u}])|\tilde{u}|^2]dx-
t^{-2}\int_{\R^2}[f(\gamma t\tilde{u})\gamma t\tilde{u}-2F(\gamma t\tilde{u})]dx.
\]
Since $\tilde{u}\in \mathcal{M}(\tilde{a})$ which shows that $J(\tilde{u})=0$, then we have that
$\frac{\partial\Phi_E}{\partial t}(1,1)=J(\tilde{u})=0$ and
\begin{align*}
 \frac{\partial\Phi_E}{\partial \gamma}(1,1)&=\int_{\R^2} |\nabla \tilde{u}|^2dx+3\int_{\R^2}
  (A_1^2[\tilde{u}]+A_2^2[\tilde{u}])|\tilde{u}|^2dx-
 \int_{\R^2}f( \tilde{u})\tilde{u}dx   \\
   & = -\tilde{\lambda}\int_{\R^2}|\tilde{u}|^2dx=-\tilde{\lambda} \tilde{a}^2<0,
\end{align*}
where we choose $\tilde{u}$ as a text function of Eq. \eqref{mainequation1} with
$u=\tilde{u}$ and $\lambda=\tilde{\lambda}$ as well as
  \eqref{gauge0}.

Moreover, combining $J(\tilde{u})=0$ and Lemma A.3 in the Appendix below,
 we can compute that
\begin{align*}
 \frac{\partial^2\Phi_E}{\partial t^2}(1,1)&=\int_{\R^2}[ |\nabla \tilde{u}|^2dx+
  (A_1^2[\tilde{u}]+A_2^2[\tilde{u}])|\tilde{u}|^2]dx-
 \int_{\R^2}[f^\prime( \tilde{u})\tilde{u}^2-4f( \tilde{u})\tilde{u}+6F(\tilde{u})]dx   \\
   & = - \int_{\R^2}[f^\prime( \tilde{u})\tilde{u}^2-5f( \tilde{u})\tilde{u}+8F(\tilde{u})]dx<0.
\end{align*}
Thanks to the above three facts, for all  $|\iota_t|>0$ small enough and $\iota_\gamma>0$, there holds
\begin{equation}\label{decreasing1}
E(\tilde{u}_{1+\iota_t,1+\iota_\gamma})=  \Phi_E(1+\iota_t,1+\iota_\gamma)<\Phi_E(1,1)=E(\tilde{u}).
\end{equation}
According to $J(\tilde{u})=0$ again, one sees that $\Phi_J(1,1)=0$. Again
adopting Lemma A.3 in the Appendix below and $J(\tilde{u})=0$, there holds
\begin{align*}
  \frac{\partial\Phi_J}{\partial t}(1,1)  & =2\int_{\R^2}[  |\nabla \tilde{u}|^2+
 (A_1^2[\tilde{u}]+A_2^2[\tilde{u}])|\tilde{u}|^2]dx-
\int_{\R^2}[f^\prime( \tilde{u})\tilde{u}^2-3f( \tilde{u})\tilde{u}+4F(\tilde{u})]dx   \\
    & = - \int_{\R^2}[f^\prime( \tilde{u})\tilde{u}^2-5f( \tilde{u})\tilde{u}+8F(\tilde{u})]dx<0.
\end{align*}
It therefore concludes from the implicit function theorem that there exist a constant $\varepsilon>0$
and a continuous function $g:[1-\varepsilon,1+\varepsilon]\to \R$ satisfying $g(1)=1$
such that $\Phi_J(g(\gamma),\gamma)=0$ for every $\gamma\in[1-\varepsilon,1+\varepsilon]$.
Particularly, there holds $\tilde{u}_{g(1+\varepsilon),1+\varepsilon}\in \mathcal{M}((1+\varepsilon)^2\tilde{a})$.
Consequently, we infer from \eqref{decreasing1} that
\[
m((1+\varepsilon)^2\tilde{a})\leq \inf_{u\in \mathcal{M}((1+\varepsilon)^2\tilde{a})}E(u)\leq E(\tilde{u}_{g(1+\varepsilon),1+\varepsilon})
<E(\tilde{u})=m(\tilde{a})
\]
showing the desired result. The proof is completed.
\end{proof}

At this point, we can exhibit the detailed proof of Theorem \ref{maintheorem2}.

\begin{proof}[\emph{\textbf{Proof of Theorem \ref{maintheorem2}}}]
Arguing as before, we can derive that there exists a Palais-Smale sequence
$\{u_n\}\subset \mathcal{M}(a)$ for $E$ restricted on $S(a)$
at the level $m(a)$. In fact, all of the conclusions in Section \ref{Section2}
remain true in this situation and we shall exploit them directly when there is no misunderstanding.
Following the proof of Theorem \ref{maintheorem1}, we continue to divide the proof into three steps.

\textbf{Step I.} The weak limit $u_0\neq0$.

In view of the Step 1 in the proof of Theorem \ref{maintheorem1},
the proof will be done by ruling out the Case I. Suppose it by contradiction that,
if the Case I holds true, then $u_n\to0$ in $L^p(\R^2)$ for all $2<p<+\infty$.
Since we obtain \eqref{2compact1} by Lemma \ref{bounded},
then $F(u_n)\to 0$ in $L^1(\R^2)$ as $n\to\infty$ by Lemma \ref{2compact}.
Hence, taking $E(u_n)\to m(a)$, \eqref{gauge1} and Lemma \ref{estimate} into account, we have that
\[
\limsup_{n\to\infty}\int_{\R^2}|\nabla u_n|^2dx=2\limsup_{n\to\infty}E(u_n)=2
m(a)<\frac{ 4 \pi}{\alpha_0}.
\]
Thereby, we shall choose $\alpha>\alpha_0$ sufficiently close to $\alpha_0$ and $\nu$ sufficiently close to 1 in such
a way that $\frac1\nu+\frac1{\nu^\prime}=1$ and
\[
 |\nabla u_n|_2^2  <\frac{4\pi(1-\epsilon)}{\nu\alpha}~\text{for some suitable}~\epsilon\in(0,1).
\]
Define $\bar{u}_n=\sqrt{\frac{\nu\alpha}{4\pi(1-\epsilon)}}u_n$, then $|\nabla \bar{u}_n|_2^2<1$ and $|\bar{u}_n|_2^2
=\frac{\nu\alpha}{4\pi(1-\epsilon)}a^2<+\infty$.
With these observations, we adopt \eqref{TM2} to obtain
\[
\sup_{n\in \mathbb{N}^+}\int_{\R^2}(e^{\nu\alpha |u_n|^2}-1)dx=
\sup_{n\in \mathbb{N}^+}\int_{\R^2}(e^{4\pi(1-\epsilon)|\bar{u}_n|^2}-1)dx\leq C<+\infty.
\]
which together with \eqref{growth3} with $q>2$ gives that
\[
\int_{\R^2}f(u_n)u_ndx\leq
C \int_{\R^2}|u_n|^{\chi}dx+C\bigg(\int_{\R^2}|u_n|^{q\nu^\prime}dx\bigg)^{\frac1{\nu^\prime}}.
\]
Recalling $J(u_n)=0$, we take advantage of the above fact and Lemma \ref{2compact}
to arrive at $|\nabla u_n|_2^2\to0$. However, it contradicts with Lemma
\ref{positive} and so the Case I cannot occur.

   \textbf{Step II.} $E(u_0)=m(a)$ and $\lim\limits_{n\to\infty}|\nabla u_n|_2^2=|\nabla u_0|_2^2$.

Since $u_0\neq0$, one easily sees that $0<|u_0|_2^2\triangleq \tilde{a}^2\leq
\liminf\limits_{n\to\infty}|u_n|^2_2= a^2$ by Fatou's lemma.
Recalling Lemma \ref{non-increasing} still holds true
in this case, $J(u_0)=0$ by Lemmas \ref{Pohozaev} and \ref{solution},
it then concludes
  from $(f_2)$ and Fatou's lemma again that
\begin{align*}
m(a) & \leq m(\tilde{a})\leq E(u_0)=E(u_0)-\frac12J(u_0)= \frac{1}{2}\int_{\R^{2}}\left[f(u_0 )u -4 F(u_0 )\right]dx \\
    & \leq\frac{1}{2}\liminf_{n\to\infty}\int_{\R^{2}}\left[f(u_n)u_n-4 F(u_n)\right]dx=\liminf_{n\to\infty}[E(u_n)-\frac12J(u_n)]\\
    &=\liminf_{n\to\infty} E(u_n)=m(a)
\end{align*}
showing that $E(u_0)=m(a)=m(\tilde{a})$ and
\begin{equation}\label{2proof1}
 \lim_{n\to\infty}\int_{\R^{2}}\left[f(u_n)u_n-4 F(u_n)\right]dx=\int_{\R^{2}}\left[f(u_0 )u_0 -4 F(u_0 )\right]dx.
\end{equation}
With $E(u_0)=m(a)$ in hands, combining Fatou's lemma and $J(u_0)=0$, we also obtain that
\begin{align*}
	  m(a) &=\lim_{n\to\infty}E(u_n)=\limsup_{n\to\infty}[E(u_n)-\frac{1}{\theta-2}J(u_n)] \\
	& =\limsup_{n\to\infty}\left\{ \frac{\theta-4}{2(\theta-2)}
	\int_{\R^2}[|\nabla u_n|^2+(A_1^2+A_2^2)|u_n|^2]dx+\frac{1}{\theta-2}\int_{\R^{2}}\left[f(u_n)u_n-\theta F(u_n)\right]dx\right\}\\
 & \geq \liminf_{n\to\infty}\left\{ \frac{\theta-4}{2(\theta-2)}
\int_{\R^2}[|\nabla u_n|^2+(A_1^2+A_2^2)|u_n|^2]dx+\frac{1}{\theta-2}\int_{\R^{2}}\left[f(u_n)u_n-\theta F(u_n)\right]dx\right\}\\
 &\geq \frac{\theta-4}{2(\theta-2)}\int_{\R^2}[|\nabla u_0 |^2+(A_1^2+A_2^2)|u_0 |^2]dx+\frac{1}{\theta-2}\int_{\R^{2}}\left[f(u_0)u_0-\theta F(u_0)\right]dx\\
&=E(u_0 )-\frac{1}{\theta-2}J(u _0)=E(u_0 )=m(a).
\end{align*}
As a direct consequence of the above equality, we derive
\[
\lim_{n\to\infty}\int_{\R^2} |\nabla u_n|^2 dx
=\int_{\R^2} |\nabla u_0 |^2 dx,~
\lim_{n\to\infty}\int_{\R^2} (A_1^2+A_2^2)|u_n|^2 dx
=\int_{\R^2} (A_1^2+A_2^2)|u_0 |^2dx,
\]
and
\begin{equation}\label{2proof2}
\lim_{n\to\infty}\int_{\R^{2}}\left[f(u_n)u_n-\theta F(u_n)\right]dx=\int_{\R^{2}}\left[f(u_0 )u_0 -\theta F(u_0 )\right]dx.
\end{equation}
Hence, the Step II concludes.

\smallskip

 \textbf{Step III.} There is a small $a_*>0$ such that $\lambda_0>0$ for all $a\in(0,a_*]$.
So, $|u_0|_2^2=a^2$.

Since $|\nabla u_n-\nabla u_0|_2^2=o_n(1)$, we could continue to choose $\alpha>\alpha_0$
and $\nu,\nu^\prime>1$ adopted in Step II to satisfy
\[
 |\nabla u_n-\nabla u_0|_2^2  <\frac{4\pi}{\nu\alpha(1+\bar{\epsilon})^4}~\text{for some suitable}~\bar{\epsilon}\in(0,1).
\]
Define $\tilde{u}_n=\sqrt{\frac{\nu\alpha(1+\bar{\epsilon})^4}{4\pi}}(u_n-u_0)$, then $|\nabla \tilde{u}_n|_2^2<1$ and $|\tilde{u}_n|_2^2
\leq\frac{\nu\alpha(1+\bar{\epsilon})^4}{4\pi}4a^2<+\infty$.
Besides, for the above fixed $\bar{\epsilon}\in(0,1)$, we need the following two types of Young's inequality
\[
 |a+b|^2\leq(1+\bar{\epsilon})|a|^2+(1+\bar{\epsilon}^{-1})|b|^2,~\forall a,b\in\R
\]
and
\[
e^{a+b}-1\leq \frac{1}{1+\bar{\epsilon}}\big[e^{(1+\bar{\epsilon})a}-1\big]
+\frac{\bar{\epsilon}}{1+\bar{\epsilon}}\big[e^{(1+\bar{\epsilon}^{-1})b}-1\big],~\forall a,b\in\R.
\]
By means of the above facts together with \eqref{TM2}, we derive
 \begin{align*}
\int_{\R^2} (e^{ \nu\alpha u_n^2}-1) dx&\leq\frac{1}{1+\bar{\epsilon}}
\int_{\R^2} (e^{\frac{4\pi}{(1+\bar{\epsilon})^2} |\tilde{u}_n|^2}-1) dx
+\frac{\bar{\epsilon}}{1+\bar{\epsilon}}\int_{\R^2} (e^{\nu\alpha(1+\bar{\epsilon}^{-1})^2 u_0^2}-1) dx \\
  & \leq \frac{C_1}{1+\bar{\epsilon}}+\frac{C_2\bar{\epsilon}}{1+\bar{\epsilon}}\leq C_3 <+\infty,~\forall n\in \mathbb{N}^+.
\end{align*}
Letting $q=\frac{2+\chi}{2}\geq3$ in \eqref{growth3}, we then make full use of the above fact and \eqref{GN} to deduce that
\[
\int_{\R^2} f (u_n)u_ndx \leq \varepsilon a^2\bigg(\int_{\R^2} |\nabla u_n|^2 dx\bigg)^{\frac{\chi-2}{2}}
   +\tilde{C}_\varepsilon a \bigg(\int_{\R^2} |\nabla u_n|^{2}dx\bigg)^{\frac{q-1}2}.
\]
It follows from \eqref{2proof1} and \eqref{2proof2} that
\[
\lim_{n\to\infty}\int_{\R^{2}} f(u_n)u_n dx=\int_{\R^{2}} f(u_0 )u_0  dx.
\]
From which, adopting $|\nabla u_n|_2^2\to|\nabla u_0|_2^2$ as $n\to\infty$, there holds
\[
\int_{\R^2} f (u_0 )u_0 dx \leq \varepsilon a^2\bigg(\int_{\R^2} |\nabla u_0|^2 dx\bigg)^{\frac{\chi-2}{2}}
   +\tilde{C}_\varepsilon a \bigg(\int_{\R^2} |\nabla u_0|^{2}dx\bigg)^{\frac{q-1}2}.
\]
Repeating the calculations of Step 2 in the proof of Theorem \ref{maintheorem1},
we could determine a suitable $a_*>0$ such that $\lambda_0>0$ for all $a\in(0,a_*]$.
At this stage, we reach a conclusion that $u_0\in \mathcal{M}(\tilde{a})$ is a minimizer of
$m(\tilde{a})$ and it is a solution of Eq. \eqref{mainequation1} with $\lambda=\lambda_0>0$.
Owing to Lemma \ref{decreasing}, $m(\tilde{a})$ is strictly decreasing in the right neighborhood of $\tilde{a}$.
So, we must have $\tilde{a}^2=a^2$. Otherwise, there is a sufficiently small
$\varepsilon>0$ such that $(1+\varepsilon)^2\tilde{a}<a$ and $(1+\varepsilon)^2\tilde{a}$ locates
in a right neighborhood of $\tilde{a}$. Therefore, there holds
$m(\tilde{a})>m((1+\varepsilon)^2\tilde{a})\geq m(a)$ which is impossible because of the Step II. So
$u_0\in \mathcal{M}(a)$ is a minimizer of
$m(a)$ and $(u_0,\lambda_0)$ is a couple of weak solutions of problems \eqref{mainequation1}-\eqref{mainequation1a} when $a\in(0,a_*]$.
The proof is completed.
\end{proof}

 \section{The supercritical case}\label{Section5}

In this section, we shall turn to the supercritical case for problems \eqref{mainequation1}-\eqref{mainequation1a}
and it is the main topic in the present article. Note that if there is no misunderstanding,
it is always supposed that the $\mathcal{C}^1$ function $h$ that vanishes in $(-\infty,0]$ satisfies $(h_1)-(h_3)$
for simplicity.

To conclude Theorem \ref{maintheorem4},
as explained before, we begin with verifying the necessary assumptions for
Theorems \ref{maintheorem1} and \ref{maintheorem2}.

Now, we shall deduce that the nonlinearities $f^{R,{\delta}}$ and $f^{R,2}$ given in \eqref{fR}
satisfy some growth conditions which are the counterparts of \eqref{growth1},
\eqref{growth2}, \eqref{growth3} and \eqref{growth4}, respectively.
On the one hand, for every fixed $R>0$, $q>2$, $\alpha>0$ and $\varepsilon>0$, there is a positive constant $M_\varepsilon^R$
which is dependent of $R>0$ and $\varepsilon>0$ such that
\begin{equation}\label{2growth1}
  |f^{R,\delta}(s)|\leq \varepsilon |s|^{\chi-1}+M_\varepsilon^R|s|^{q-1}(e^{\alpha |s|^2}-1),~\forall s\in\R
\end{equation}
which together with $(h_2)$ implies that
\begin{equation}\label{2growth2}
  |F^{R,\delta}(s)|\leq \varepsilon |s|^{\chi }+M_\varepsilon^R|s|^{q }(e^{\alpha |s|^2}-1),~\forall s\in\R.
\end{equation}
Actually, exploiting $(h_1)$ and $(h_3)$ with $\delta\in(0,2)$, we have that
\[
\lim_{s\to0}\frac{f^{R,\delta}(s)}{s^{\chi-1}}=\lim_{s\to0}\frac{h(s)}{s^{\chi-1}}=0~\text{uniformly in}~R>0
\]
and
\[
0\leq \lim_{|s|\to+\infty}\frac{|f^{R,\delta}(s)|}{e^{\alpha |s|^2}-1}\leq
\lim_{|s|\to+\infty}\frac{Me^{(\gamma+\bar{\alpha}_0R^{\tau-\delta})|s|^\delta}}{e^{\alpha |s|^2}-1}
=0
\]
implying the desired result \eqref{2growth1}.
Similarly, we fix $R>0$, $q>2$ and $\varepsilon>0$ to find a $\bar{M}>M$
which is independent of $R$ such that
\begin{equation}\label{2growth3}
  |f^{R,2}(s)|\leq \varepsilon |s|^{\chi-1}+\bar{M} |s|^{q-1}(e^{(\gamma +\bar{\alpha}_0R^{\delta-2}) |s|^2}-1),~\forall s\in\R,
\end{equation}
and
\begin{equation}\label{2growth4}
  |F^{R,2}(s)|\leq \varepsilon |s|^{\chi }+\bar{M} |s|^{q }(e^{(\gamma +\bar{\alpha}_0R^{\delta-2}) |s|^2}-1),~\forall s\in\R.
\end{equation}

Moreover, we have to show the counterparts of $(f_2)$ and $(f_3)$ for the nonlinearity $f^{R,\bar{\delta}}$
which are exhibited as follows.

\begin{lemma}\label{AR}
Let $f$ be defined in \eqref{form} with the $\mathcal{C}^1$ function $h$ satisfying $(h_2)$. Then, for all fixed $R>0$ one has
\[
0<F^{R,\bar{\delta}}(t)\leq \theta f^{R,\bar{\delta}}(t)t,~\forall t\in \R\backslash\{0\}.
\]
\end{lemma}

\begin{proof}
The reader can refer to \cite{AS1,AS2}, whereas we should exhibit the detailed proof for the completeness.
For all $t\in[0,R]$,
 we can see that $f^{R,\bar{\delta}}(t)=f(t)$
 and so $F^{R,\bar{\delta}}(t)=F(t)$, then the lemma is done for $t\in[0,R]$.
Indeed, for all $t\in[0,R]$, one exploits $(h_2)$ to deduce that
\[
F^{R,\bar{\delta}}(t)=F(t)=\int_0^th(s)e^{\bar{\alpha}_0s^\tau}ds
\leq e^{\bar{\alpha}_0t^\tau}\int_0^th(s)ds=e^{\bar{\alpha}_0t^\tau}H(t)\leq
\theta e^{\bar{\alpha}_0t^\tau}h(t)t=\theta f^{R,\bar{\delta}}(t)t.
\]
Given a $t \in [R,+\infty)$, then
\begin{align*}
 F^{R,\bar{\delta}}(t)& =  \int_{0}^{t}f^{R,\bar{\delta}}(s) ds= \int_{0}^{R}f^{R,\bar{\delta}}(s) ds+
\int_{R}^{t}h(s)e^{\bar{\alpha}_0 R^{\tau-\bar{\delta}}s^{\bar{\delta}}} ds \\
   & = F^{R,\bar{\delta}}(R)+  \int_{R}^{t}h(s)e^{\bar{\alpha}_0 R^{\tau-\bar{\delta}}s^{\bar{\delta}}} ds \\
&=F^{R,\bar{\delta}}(R)+ \int_{0}^{t}h(s)e^{\bar{\alpha}_0 R^{\tau-\bar{\delta}}s^{\bar{\delta}}}ds
-\int_{0}^{R}h(s)e^{\bar{\alpha}_0 R^{\tau-\bar{\delta}}s^{\bar{\delta}}}ds.
\end{align*}
We note that
$$
F^{R,\bar{\delta}}(R)=\int_{0}^{R}h(s)e^{\bar{\alpha}_0 s^{\tau}} ds\leq \int_{0}^{R}h(s)e^{\bar{\alpha}_0 R^{\tau-\bar{\delta}}s^{\bar{\delta}}} ds
$$	
leading to
$$
F^{R,\bar{\delta}}(t) \leq \int_{0}^{t}h(s)e^{\bar{\alpha}_0 R^{\tau-\bar{\delta}}s^{\bar{\delta}}}ds
\leq e^{\alpha_0 R^{\tau-\bar{\delta}}t^{\bar{\delta}}}\int_{0}^{t}h(s) ds= e^{\bar{\alpha}_0 R^{\tau-\bar{\delta}}t^{\bar{\delta}}}H(t).
$$
Thereby, it follows from $(h_2)$ that
$$
\theta F^{R,\bar{\delta}}(t) \leq \theta e^{\bar{\alpha}_0 R^{\tau-\bar{\delta}}t^{\bar{\delta}}}H(t)
 \leq e^{\bar{\alpha}_0 R^{\tau-\bar{\delta}}t^{\bar{\delta}}} h(t)t=f^{R,\bar{\delta}}(t)t,~  \forall  t \geq R.
$$
The proof is completed.
\end{proof}

\begin{lemma}\label{monotone}
Let $f$ be defined in \eqref{form} with the $\mathcal{C}^1$ function $h$ satisfying $(h_2)-(h_4)$. Then, for all fixed $R>0$ one has
\[
 \bar{F}^{R,\bar{\delta}}(t)/t^{4}~\text{is strictly increasing in}~(0,+\infty),
\]
where $\bar{F}^{R,\bar{\delta}}(t)=f^{R,\bar{\delta}}(t)t-2F^{R,\bar{\delta}}(t)$ for all $t\in\R$.
\end{lemma}

\begin{proof}
First of all, one simply observes that $f^{R,\bar{\delta}}(t)=f(t)$ for all $t\in[0,R]$,
then $(f^{R,\bar{\delta}})^\prime(t)=f^\prime(t)$ and $F^{R,\bar{\delta}}(t)=F(t)$ for all $t\in[0,R]$. Then,
for any fixed $t\in(0,R]$,
\begin{align*}
F(t)& =\int_0^th(s)e^{\bar{\alpha}_0s^\tau}ds=H(t)e^{\bar{\alpha}_0 t^\tau}
-\int_0^tH(s)d e^{\bar{\alpha}_0 s^\tau}=H(t)e^{\bar{\alpha}_0 t^\tau}
-\bar{\alpha}_0\tau \int_0^tH(s)s^{\tau-1} e^{\bar{\alpha}_0 s^\tau}ds \\
    &
\geq H(t)e^{\bar{\alpha}_0 t^\tau}
-\bar{\alpha}_0\tau H(t)e^{\bar{\alpha}_0 t^\tau} \int_0^ts^{\tau-1} ds
 =H(t)e^{\bar{\alpha}_0 t^\tau}
-\bar{\alpha}_0 H(t)t^\tau e^{\bar{\alpha}_0 t^\tau}.
\end{align*}
Hence, for all $t\in(0,R]$, using $(h_2)$, we deduce that
\[\begin{gathered}
  (f^{R,\bar{\delta}})^\prime(t)t^2-5f^{R,\bar{\delta}}(t)t+8F^{R,\bar{\delta}}(t)=f^\prime(t)t^2-5f(t)t+8F(t)\hfill\\
  \ \ \ \    \geq  \big[h^\prime(t)t^2-5h(t)t+8H(t)\big]e^{\bar{\alpha}_0 t^\tau}
+\bar{\alpha}_0 t^{\tau } e^{\bar{\alpha}_0 t^\tau}[\tau h(t)t-8H(t)]   \hfill\\
\ \ \ \  >\bar{\alpha}_0 t^{\tau } e^{\bar{\alpha}_0 t^\tau}(\tau\theta-8)H(t)>0,   \hfill\\
\end{gathered}\]
where we have applied a very similar result in Lemma A.3 in the Appendix below to $(h_3)$
in the last second inequality. So, $\forall t\in(0,R]$, it has that
\begin{equation}\label{monotone1}
 \frac d{dt}\left(\frac{\bar{F}^{R,\bar{\delta}}(t)}{t^{4}} \right)
=\frac{(f^{R,\bar{\delta}})^\prime(t)t^2-5f^{R,\bar{\delta}}(t)t+8F^{R,\bar{\delta}}(t)}{t^5}>0,
\end{equation}
showing the conclusions for all $t\in(0,R]$.

We next consider the case for all $t\in[R,+\infty)$.
Firstly, one can calculate that
\begin{align*}
  F^{R,\bar{\delta}}(R) &= \int_{0}^{R}h(s)e^{\bar{\alpha}_0 s^{\tau}} ds = \int_{0}^{R} e^{\bar{\alpha}_0 s^{\tau}} dH(s)
  =H(R)e^{\bar{\alpha}_0 R^{\tau}} -\bar{\alpha}_0\tau \int_{0}^{R} H(s)s^{\tau-1}e^{\bar{\alpha}_0 s^{\tau}}ds \\
   & \geq  H(R)e^{\bar{\alpha}_0 R^{\tau}} -\bar{\alpha}_0 H(R)e^{\bar{\alpha}_0 R^{\tau}}\int_{0}^{R}\tau s^{\tau-1}ds
   =H(R)e^{\bar{\alpha}_0 R^{\tau}} -\bar{\alpha}_0 R^{\tau}H(R)e^{\bar{\alpha}_0 R^{\tau}}.
\end{align*}
Similarly, one has that
 \begin{align*}
 \int_{R}^{t}h(s)e^{\bar{\alpha}_0 R^{\tau-\bar{\delta}}s^{\bar{\delta}}}ds& =
\int_{R}^{t}e^{\bar{\alpha}_0 R^{\tau-\bar{\delta}}s^{\bar{\delta}}}dH(s)
=H(t)e^{\bar{\alpha}_0 R^{\tau-\bar{\delta}}t^{\bar{\delta}}}-H(R)e^{\bar{\alpha}_0 R^{\tau}}
-\int_{R}^{t}H(s)de^{\bar{\alpha}_0 R^{\tau-\bar{\delta}}s^{\bar{\delta}}} \\
   & =H(t)e^{\bar{\alpha}_0 R^{\tau-\bar{\delta}}t^{\bar{\delta}}}-H(R)e^{\bar{\alpha}_0 R^{\tau}}
-\bar{\alpha}_0 R^{\tau-\bar{\delta}}\bar{\delta}\int_{R}^{t}H(s)s^{\bar{\delta}-1}e^{\bar{\alpha}_0 R^{\tau-\bar{\delta}}s^{\bar{\delta}}}ds\\
&\geq H(t)e^{\bar{\alpha}_0 R^{\tau-\bar{\delta}}t^{\bar{\delta}}}-H(R)e^{\bar{\alpha}_0 R^{\tau}}
-\bar{\alpha}_0  R^{\tau-\bar{\delta}}H(t)e^{\bar{\alpha}_0 R^{\tau-\bar{\delta}}t^{\bar{\delta} }}\int_{R}^{t}\bar{\delta}s^{\bar{\delta}-1}ds\\
&\geq H(t)e^{\bar{\alpha}_0 R^{\tau-\bar{\delta}}t^{\bar{\delta}}}-H(R)e^{\bar{\alpha}_0 R^{\tau}}
-\bar{\alpha}_0  R^{\tau-\bar{\delta}}H(t)t^{\bar{\delta} }e^{\bar{\alpha}_0 R^{\tau-\bar{\delta}}t^{\bar{\delta} }}
+\bar{\alpha}_0  R^{\tau}H(R)e^{\bar{\alpha}_0 R^{\tau}}.
\end{align*}
From which, we obtain
\begin{align*}
 F^{R,\bar{\delta}}(t) & =F^{R,\bar{\delta}}(R)+ \int_{R}^{t}h(s)e^{\bar{\alpha}_0 R^{\tau-\bar{\delta}}s^{\bar{\delta}}}ds
  \geq H(t)e^{\bar{\alpha}_0 R^{\tau-\bar{\delta}}t^{\bar{\delta}}}
 -\bar{\alpha}_0  R^{\tau-\bar{\delta}}H(t)t^{\bar{\delta} }e^{\bar{\alpha}_0 R^{\tau-\bar{\delta}}t^{\bar{\delta} }}
\end{align*}
Recalling $\delta \in [8\theta^{-1},2)$ in $(h_4)$ which implies that $\bar{\delta} -8\theta\geq0$, adopting $(h_2)-(h_3)$,
 so there holds
\[\begin{gathered}
  (f^{R,\bar{\delta}})^\prime(t)t^2-5f^{R,\bar{\delta}}(t)t+8F^{R,\bar{\delta}}(t) \hfill\\
  \ \ \ \    \geq  \big[h^\prime(t)t^2-5h(t)t+8H(t)\big]e^{\bar{\alpha}_0 R^{\tau-\bar{\delta}}t^{\bar{\delta}}}
+\bar{\alpha}_0  R^{\tau-\bar{\delta}}t^{\bar{\delta} }e^{\bar{\alpha}_0 R^{\tau-\bar{\delta}}t^{\bar{\delta} }}
[\bar{\delta} h(t)t-8H(t)]   \hfill\\
\ \ \ \  >\bar{\alpha}_0  R^{\tau-\bar{\delta}}t^{\bar{\delta} }e^{\bar{\alpha}_0 R^{\tau-\bar{\delta}}t^{\bar{\delta} }}
(\bar{\delta} -8\theta)H(t)\geq0,   \hfill\\
\end{gathered}\]
yielding that \eqref{monotone1} remains true for all $t\in[R,+\infty)$. The proof is completed.
\end{proof}

With \eqref{2growth1}, \eqref{2growth2}, \eqref{2growth3}, and \eqref{2growth4}
as well as Lemmas \ref{AR}-\ref{monotone}
 in hands, we could
proceed as the proofs in Section \ref{Section2} to get that,
for all fixed $R>0$, there are two sequences $\{u_n\}\subset S_a$ and $\{\lambda_n^{R,\bar{\delta}}\}\subset \R$ such that
\begin{equation}\label{PSsequence1}
  E^{R,\bar{\delta}}(u_n)=m^{R,\bar{\delta}}(a)+o_n(1)~\text{as}~n\to\infty,
\end{equation}
\begin{equation}\label{PSsequence2}
  -\Delta u_n+\lambda_n^{R,\bar{\delta}}u_n+ A_0u_n+\sum_{j=1}^2A_j^2[u_n] u_n- f^{R,\bar{\delta}}(u_n )
=o_n(1)~\text{as}~n\to\infty,
\end{equation}
\begin{equation}\label{PSsequence3}
J^{R,\bar{\delta}}(u_n)= o_n(1)~\text{as}~n\to\infty,
\end{equation}
where $E^{R,\bar{\delta}}:S(a)\to \R$ is given by \eqref{functional2} and
\begin{equation}\label{montainpassvalue}
  m^{R,\bar{\delta}}(a)\triangleq\inf_{u\in \mathcal{M}^{R,\bar{\delta}}(a)} E^{R,\bar{\delta}}(u)>0,
 ~\mathcal{M}^{R,\bar{\delta}}(a)=\big\{u\in S(a):J^{R,\bar{\delta}}(u)=0\big\}.
\end{equation}
Here the auxiliary energy functional $J^{R,\bar{\delta}}:S(a)\to\R$ is defined by
\[
J^{R,\bar{\delta}}(u)= \int_{\R^2}[|\nabla u|^2 + (A_1^2+A_2^2)|u|^2]dx-\int_{\R^2}[f^{R,\bar{\delta}}(u)u-2F^{R,\bar{\delta}}(u)]dx.
\]

 Combining
\eqref{PSsequence1}, \eqref{PSsequence2} and \eqref{PSsequence3}
as well as Lemma \ref{AR}, we can repeat the proof of Lemmas \ref{bounded} and \ref{2bounded}
to see that the sequence $\{u_n\}\subset S(a)$ is uniformly bounded in $H^1(\R^2)$ and
\begin{equation}\label{2bounded1}
\lambda_n^{R,\bar{\delta}}
=\frac2{a^2}\left\{\int_{\R^2} |\nabla u_n|^2  dx+3\int_{\R^2} F^{R,\bar{\delta}}(u_n)dx
-\int_{\R^2} f ^{R,\bar{\delta}}(u_n)u_ndx\right\}.
\end{equation}
Moreover, we can derive that \eqref{un} and \eqref{lambda}
 remain true in some suitable forms in this situation.

In what follows, we shall handle the case $\bar{\delta}=\delta$
and $\bar{\delta}=2$ in two subsections.

\subsection{The case $\bar{\delta}=\delta$}\label{subcritical}\

In this subsection, we know that the nonlinearity $f^{R,\bar{\delta}}=f^{R,\delta}$.
According to the above discussions,
we can give the proof of the first part (subcritical exponential case) of Theorem \ref{maintheorem4}.

\begin{proof}[\textbf{\emph{Proof of Theorem \ref{maintheorem4}}}]
The proof is totally similar to that of Theorem \ref{maintheorem1},
so we omit it here.
\end{proof}

Before presenting the proof of Theorem \ref{maintheorem3},
we have to study the $L^\infty$-estimate for the nontrivial solution $(\bar{u}_0^R,\bar{\lambda}_0^R)$
established in Theorem \ref{maintheorem4}.
We first show that any sequence satisfying
\eqref{PSsequence1}, \eqref{PSsequence2} and \eqref{PSsequence3} is uniformly bounded in $n\in \mathbb{N}^+$ and $R>0$
 for some suitable $\bar{\alpha}_0>0$.

\begin{lemma}\label{10bounded}
 There is a constant
 $\bar{\alpha}_0^*=\frac{1}{R^{\tau-\delta}}>0$ such that for all $\bar{\alpha}_0\in(0,\bar{\alpha}_0^*)$ and $\tau>2$,
the sequence $\{u_n\}\subset H^1(\R^2)$ satisfying \eqref{PSsequence2}, \eqref{PSsequence3} and \eqref{PSsequence3} is uniformly bounded in $n\in \mathbb{N}^+$
and $R>0$, that is,
there is a constant $\Sigma>0$ independent of $n\in \mathbb{N}^+$ and $R>0$ such that
\begin{equation}\label{10bounded01}
\sup_{n\in \mathbb{N}}|\nabla u_n |_{2}^2\leq\Sigma<+\infty~\text{\emph{and}}~0<\lambda_n^{R,\delta}\leq \Sigma<+\infty.
\end{equation}
\end{lemma}

\begin{proof}
We claim that there is a constant $\bar{m}(a)>0$ which is independent of $R>0$ such that
\begin{equation}\label{1bounded4}
0<m^{R,{\delta}}(a)\leq \bar{m}(a)<+\infty,~\forall R>0.
\end{equation}
Indeed, in view of the definition of $f^{R,\delta}$, one concludes that $F^{R,\delta}(t)\geq H(t)$
for all $t\in\R$, and so, $E^{R,\delta}(u)\leq I(u)$ for all $u\in H^1(\R^2)$, where the
energy functional $I:H^1(\R^2)\to\R$ is defined by
\begin{equation}\label{hhh}
I(u)=\frac{1}{2}\int_{\R^2}[|\nabla u|^2 + (A_1^2+A_2^2)|u|^2]dx-
\int_{\R^2}H(u)dx.
\end{equation}
Here, it is enough to choose the constant $\bar{m}(a)$ to be a minimization of $I|_{\mathcal{M}(a)}$.

Secondly, we improve \eqref{2growth1} and \eqref{2growth2}
in the sense that the positive constant $M_\varepsilon^R$ is independent of $R>0$
 under some suitable $\bar{\alpha}_0>0$. Indeed, for each $\bar{\alpha}_0\in(0,\bar{\alpha}_0^*)$ with
$\bar{\alpha}_0^*=\frac{1}{R^{\tau-\delta}}>0$, by $(h_4)$ with $\delta\in(0,2)$, one has that
\[
|f^{R,\delta}(s)|\leq Me^{\gamma |s|^\delta}e^{\bar{\alpha}_0^* R^{\tau-\delta} |s|^\delta}
\leq M e^{(\gamma + 1)|s|^\delta}, \quad \forall s \in \mathbb{R}.
\]
Thus, given $\alpha>0$, there is $M_1=M_1(\alpha)>0$ independent of $R>0$, such that
\[
|f^{R,\delta}(s)| \leq M_1e^{\alpha |s|^2}, ~\forall s\in\R.
\]
Therefore, there is $M_\varepsilon>0$ independent of $R>0$ such that
\begin{equation}\label{22growth1}
  |f^{R,\delta}(s)|\leq \varepsilon |s|^{\chi-1}+M_\varepsilon|s|^{q-1}(e^{\alpha |s|^2}-1),~\forall s\in\R
\end{equation}
and
\begin{equation}\label{22growth2}
  |F^{R,\delta}(s)|\leq \varepsilon |s|^{\chi }+M_\varepsilon|s|^{q }(e^{\alpha |s|^2}-1),~\forall s\in\R.
\end{equation}

Finally, by virtue of \eqref{1bounded4} and \eqref{22growth1}-\eqref{22growth2},
we could repeat the calculations exploited in Lemma \ref{bounded}
to derive the first part of \eqref{10bounded01}. From it and \eqref{1bounded4} combined with \eqref{PSsequence1}, we obtain the
second part of \eqref{10bounded01}. Hence,
the proof of this lemma is completed.
 \end{proof}

Thanks to the growth conditions \eqref{22growth1} and \eqref{22growth2} with respect to the nonlinearity $f^{R,\delta}$,
we recall the choice of $\bar{a}^*_R>0$ which is very similar to its counterpart in the Step
3 in the proof of Theorem \ref{maintheorem1}, it would conclude that
 $\bar{a}^*_R$ is independent of $R>0$.
 Furthermore, we have the following result.

 \begin{lemma}\label{3A}
 Let $(\bar{u}_0^R,\bar{\lambda}_0^R)$
established by Theorem \ref{maintheorem4} for all fixed $R>0$, then there is a constant
 $\bar{\alpha}_0^*=\frac{1}{R^{\tau-\delta}}>0$ such that for all $\bar{\alpha}_0\in(0,\bar{\alpha}_0^*)$ and $\tau>2$, we have
 \[
 \sup_{R>0}|A_0[\bar{u}_0^R]+A_1^2[\bar{u}_0^R]+A_2^2[\bar{u}_0^R]|_\infty\leq C_A,
 \]
 where $C_A\in(0,+\infty)$ is a constant independent of $R>0$.
 \end{lemma}

 \begin{proof}
 According to the above observations, we derive that the constant $a\in(0,\bar{a}^*_R]$
 is independent of $R>0$. On the one hand, using \eqref{CSS2e1} and \eqref{CSS2e2} for $j=1,2$, there holds
 \begin{align*}
|A_j[\bar{u}_0^R]|_\infty& \leq \frac1{4\pi}\int_{\R^2}\frac{|\bar{u}_0^R|^2}{|x-y|}dy
=\frac1{4\pi}\int_{|x-y|<1}\frac{|\bar{u}_0^R|^2}{|x-y|}dy
+\frac1{4\pi}\int_{|x-y|\geq1}\frac{|\bar{u}_0^R|^2}{|x-y|}dy \\
  &\leq  \frac1{\sqrt[3]{4\pi}}\left(\int_{\R^2} |\bar{u}_0^R|^6dy\right)^{\frac13}
+\frac1{4\pi}\int_{\R^2} |\bar{u}_0^R|^2 dy\leq \frac1{\sqrt[3]{4\pi}}\Sigma^{\frac23}a^{\frac23}+\frac1{4\pi}a^2,
\end{align*}
where we have used \eqref{GN} and \eqref{10bounded01}. We exploit \eqref{CSS2d} to see that
\[
|A_0[\bar{u}_0^R]|_\infty\leq \frac{|A_j[\bar{u}_0^R]|_\infty}\pi\int_{\R^2}\frac{|\bar{u}_0^R|^2}{|x-y|}dy.
\]
Combining the above two formulas, we can finish the proof of this lemma.
 \end{proof}

At this position, we are ready to give the proof of Theorem \ref{maintheorem4}.

\begin{proof}[\emph{\textbf{Proof of Theorem \ref{maintheorem3}}}]
Firstly, with \eqref{10bounded01} and \eqref{22growth1} in hands,
we can easily prove that $|f^{R,\delta}(\bar{u}_0^R)|_2$
is uniformly bounded in $R>0$.

Then, we would show that there is a constant $C_0>0$ which is independent of $R>0$
such that $|\bar{u}_0^R|_\infty\leq C_0$.
To obtain it, since
 $\bar{u}_0^R$ is a nontrivial solution of Eq. \eqref{mainequation1} with a suitable $\bar{\lambda}_0^R$, namely there holds
$$
 -\Delta \bar{u}_0^R +\bar{\lambda}_0^R\bar{u}_0^R
 +(A_0[\bar{u}_0^R]+A_1^2[\bar{u}_0^R]+A_2^2[\bar{u}_0^R])\bar{u}_0^R=f^{R,\delta}(\bar{u}_0^R) ~ \mbox{in}~ \R^2.
 $$
Since $h$ is nonnegative, without loss of generality, we can assume that $\bar{u}_0^R\geq0$ for all $x\in\R^2$,
and so,
$$
 -\Delta \bar{u}_0^R + \bar{u}_0^R\leq (1+\bar{\lambda}_0^R)\bar{u}_0^R
 + |A_0[\bar{u}_0^R]+A_1^2[\bar{u}_0^R]+A_2^2[\bar{u}_0^R]|\bar{u}_0^R+ f^{R,\delta}(\bar{u}_0^R) ~ \mbox{in}~ \R^2.
 $$
According to \eqref{10bounded01}, we have that $\bar{\lambda}_a^R\leq \bar{C}<+\infty$
for some constant $\bar{C}>0$ which is independent of $R>0$.
As a consequence  of Lemma \ref{3A}, the Lax-Milgram
theorem gives the existence of a function $\bar{w}_0^R\in H^{2}(\R^2)$ such that
$$
 -\Delta \bar{w}_0^R+\bar{w}_0^R=(1+\bar{\lambda}_0^R)\bar{u}_0^R
 + |A_0[\bar{u}_0^R]+A_1^2[\bar{u}_0^R]+A_2^2[\bar{u}_0^R]|\bar{u}_0^R+ f^{R,\delta}(\bar{u}_0^R) ~ \mbox{in} ~ \R^2.
 $$
Next, we fix the test function
 $$
 v_r(x)=\phi(x/r)(\bar{u}_0^R-\bar{w}_0^R)^{+}(x) \in H^{1}(\R^2),
 $$
 where $\phi \in C_{0}^{\infty}(\R^2)$ satisfies
 $$
 0 \leq \phi(x) \leq 1,~\forall x \in \R^2;~\phi(x)=1, ~ \forall x \in B_1(0); ~\mbox{and}~\phi(x)=0~\forall x \in B^{c}_{2}(0).
 $$
Using the function test $v_r$
 on $-\Delta(\bar{u}_0^R-\bar{w}_0^R)+(\bar{u}_0^R-\bar{w}_0^R)\leq0$ in $\R^2$, we get the inequality below
 $$
 \int_{\R^2}\nabla\big(\bar{u}_0^R-\bar{w}_0^R)\nabla v_r + (\bar{u}_0^R-\bar{w}_0^R)v_r\big] dx \leq 0,
 $$
Since
$$
v_r \to (\bar{u}_0^R-\bar{w}_0^R)^+ ~\mbox{as} ~ r \to +\infty ~\mbox{in} ~H^1(\R^2),
$$
by the Lebesgue's Dominated Convergence theorem,
we arrive at
$$
 \int_{\R^2}|\nabla (\bar{u}_0^R-\bar{w}_0^R)^+|^2+|(\bar{u}_0^R-\bar{w}_0^R)^+|^2  \,dx \leq 0,
$$
implying that
\[
0\leq \bar{u}_0^R\leq \bar{w}_0^R,~ \forall x \in \R^2.
\]
By using the continuous Sobolev embedding $H^2(\R^2) \hookrightarrow L^{\infty}(\R^2)$, there is a $C_4>0$ independent of $R>0$ such that
$$
|\bar{w}_0^R|_\infty \leq C_4\|\bar{w}_0^R\|_{H^{2}(\R^2)}, ~\forall R>0
$$
which together with the last fact gives that
$$
|\bar{u}_0^R|_\infty \leq C_4\|\bar{w}_0^R\|_{H^{2}(\R^2)}, \quad \forall R>0.
$$
On the other hand, by Br\'{e}zis \cite[Theorem 9.25]{Brezis}, there is a $C_5>0$ independent of $R>0$ such that
$$
\|\bar{w}_0^R\|_{H^{2}(\R^2)} \leq C_5\big|(1+\bar{\lambda}_0^R)  \bar{u}_0^R
+|A_0[\bar{u}_0^R]+A_1^2[\bar{u}_0^R]+A_2^2[\bar{u}_0^R]|\bar{u}_0^R + f^{R,\delta}(\bar{u}_0^R)\big|_2, ~ \forall R>0,
$$
from where it follows that
$$
\|\bar{w}_0^R\|_{H^{2}(\R^2)} \leq C_6, ~ \forall R>0
$$
for some $C_6>0$ independent of $R>0$. Have this in mind, we must have
$$
|\bar{u}_0^R|_\infty \leq C_0, \quad \forall R>0,
$$
for some $C_0>0$ independent of $R>0$, showing the desired result.

Finally, we shall conclude that $\bar{u}_0^R\in H^1(\R^2)$ is a nontrivial solution of the original Eq. \eqref{mainequation1}
with $\lambda=\bar{\lambda}_0^R$
by choosing $R=C_0>0$ in \eqref{fR} and hence $\bar{\alpha}_0^*= C_0^{ \delta-\tau} >0$.
In other words, the couple $(\bar{u}_0^R,\bar{\lambda}_0^R)$ is a weak solution of problems \eqref{mainequation1}-\eqref{mainequation1a}.
 The proof is completed.
\end{proof}

\subsection{The case $\bar{\delta}=2$}\label{cricritical}\

In this subsection, we have the nonlinearity $f^{R,\bar{\delta}}=f^{R,2}$.
Firstly, we derive the following estimate for the mountain-pass value $m^{R,2}(a)$
defined in \eqref{montainpassvalue}.

\begin{lemma}\label{2estimate}
Suppose $(h_5)$ additionally, then there is a sufficiently large $\xi_0^R>0$
such that for all $\xi>\xi_0^R$, there holds
\begin{equation}\label{estimate1}
  m^{R,2}(a)<\frac{2\pi}{ \gamma +\bar{\alpha}_0R^{\delta-2} }.
\end{equation}
\end{lemma}

\begin{proof}
Given $a>0$, define $\psi(x)=a \pi^{-\frac12}e^{-\frac12|x|^2}$
for all $x\in\R^2$
and  so $|\nabla \psi|_2^2=|\psi|_2^2=a^2$. By \eqref{gauge1},
\[
\int_{\R^2}  (A_1^2[\psi]+A_2^2[\psi])|\psi|^2 dx\leq 16\bar{C}_r^2a^6.
\]
For all $t>0$, since $f^{R,2}(t)\geq h(t)\geq \xi t^{p-1}$ for all $t>0$, we have that
\begin{align*}
E^{R,2}(\psi_t) &
\leq \max_{t>0}\bigg(\frac{t^2}{2}\int_{\R^2}[|\nabla \psi|^2 + (A_1^2+A_2^2)|\psi|^2]dx-\frac{\xi t^{p-2}}{p}
\int_{\R^2}|\psi|^pdx\bigg)\\
&\leq \frac{p-4}{2(p-2)}\left[\frac{p}{(p-2)\xi}\right]^{\frac2{p-4}}(1+16\bar{C}_r^2a^3)^{\frac{2(p-2)}{p-4}}
\pi^{\frac{p-2}{p-4}}a^{-\frac{2p}{p-4}},
\end{align*}
where we have used $|\psi|_p^p\geq a^p\pi^{-\frac{p-2}{2}}$.
By Lemma \ref{unique}, one sees that $m^{R,2}(a)\leq \max\limits_{t>0}E^{R,2}(\psi_t)$.
Let us choose the constant $\xi_0(R)$ ro satisfy
\[
\xi_0^R=
\frac{p(1+16\bar{C}_r^2a^3)^{ p-2 }}{(p-2)a^{p}}
\pi^{\frac{p-2}{2}}
\left[\frac{(p-4)(\gamma +\bar{\alpha}_0R^{\delta-2})}{4\pi(p-2)}\right]^{\frac{p-4}2}
\]
which indicates the desired result. The proof is completed.
\end{proof}

Now, we can exhibit the proof of the second part (critical exponential case) of Theorem \ref{maintheorem4}.

\begin{proof}[\textbf{\emph{Proof of Theorem \ref{maintheorem4}}}]
\textbf{(Completed).}
Due to Lemma \ref{2estimate}, the proof is totally similar to that of Theorem \ref{maintheorem2} and
so we just show the counterpart of Step I in this situation.

If the case I occurs, then we can derive
\[
\limsup_{n\to\infty}\int_{\R^2}|\nabla u_n|^2dx=2\limsup_{n\to\infty}E^{R,2}(u_n)=2m^{R,2}
(a)<\frac{4\pi}{\gamma +\bar{\alpha}_0R^{\delta-2}}.
\]
Choosing $\nu>1$ sufficiently close to 1 in such
a way that $\frac1\nu+\frac1{\nu^\prime}=1$ and
\[
 |\nabla u_n|_2^2  <\frac{4\pi(1-\epsilon)}{\nu(\gamma +\bar{\alpha}_0R^{\delta-2})}~\text{for some suitable}~\epsilon\in(0,1).
\]
Define $\bar{u}_n=\sqrt{\frac{\nu(\gamma +\bar{\alpha}_0R^{\delta-2})}{4\pi(1-\epsilon)}}u_n$, then $|\nabla \bar{u}_n|_2^2<1$ and $|\bar{u}_n|_2^2
=\frac{\nu(\gamma +\bar{\alpha}_0R^{\delta-2})}{4\pi(1-\epsilon)}a^2<+\infty$. Using the above facts,
we depend on \eqref{TM2} to see that
\[
\int_{\R^2} (e^{ \nu(\gamma +\bar{\alpha}_0R^{\delta-2}) u_n^2}-1) dx
=\int_{\R^2} (e^{ 4\pi (1-\epsilon)\bar{u}_n^2}-1) dx\leq C<+\infty,
\]
The remaining part is trivial and we omit it here. The proof is completed.
\end{proof}

Then, arguing as before, we begin considering the $L^\infty$-estimate for the nontrivial solution $(\bar{u}_0^R,\bar{\lambda}_0^R)$
established in Theorem \ref{maintheorem4}. To the end,
we prove the following result.

\begin{lemma}\label{20bounded}
Let $(h_5)$ be satisfied.
 There is a constant
 $\tau^*=2+\frac{1}{R}>0$, then for all for every $\alpha>0$, $\tau\in[2,\tau^*)$ and $R>e$,
each sequence $\{u_n\}\subset H^1(\R^2)$ satisfying \eqref{PSsequence1}, \eqref{PSsequence2} and \eqref{PSsequence3} is uniformly bounded in $n\in \mathbb{N}$
and $R>e$, that is,
there is a constant $\Pi>0$ independent of $n\in \mathbb{N}$ and $R>0$ such that
\begin{equation}\label{20bounded01}
\sup_{n\in \mathbb{N}} |\nabla u_n |_{2}^2\leq\Pi<+\infty~\text{\emph{and}}~0<\lambda_n^{R,\delta}\leq \Pi<+\infty,
\end{equation}
provided $\xi_0>0$ in $(h_5)$ is sufficiently large,
where the constant $\Pi>0$ independent of $R>e$ satisfies
\begin{equation}\label{20bounded00}
 \Pi<\frac{\pi}{2\alpha e^{\frac1e}}.
\end{equation}
\end{lemma}

\begin{proof}
We claim that there is a constant $\tilde{m}(a)>0$ which is independent of $R>0$ such that
\begin{equation}\label{2bounded4}
0<m^{R,{2}}(a)\leq \tilde{m}(a)<+\infty,~\forall R>0.
\end{equation}
Indeed, in view of the definition of $f^{R,2}$, one concludes that $F^{R,2}(t)\geq H(t)$
for all $t\in\R$, and so, $E^{R,2}(u)\leq I(u)$ for all $u\in H^1(\R^2)$, where the
energy functional $I:H^1(\R^2)\to\R$ is defined by \eqref{hhh}.
Here, it is enough to choose the constant $\tilde{m}(a)$ to be a minimization of $I|_{\mathcal{M}(a)}$.
Exploiting the very similar calculations in Lemma \ref{2estimate}, we can find such a $\Pi>0$ in \eqref{20bounded00}
to satisfy
\begin{equation}\label{2bounded4444}
\tilde{m}(a)<\frac{\theta-4}{ 2(\theta-3)} \Pi.
\end{equation}

Secondly, we shall improve \eqref{2growth3} and \eqref{2growth4}.
Obviously, $\lim\limits_{R\to+\infty}R^{\frac{1}{R}}=1$
and the function $R^{\frac{1}{R}}$ is strictly decreasing in $R\in(e,+\infty)$,
then $0<R^{\frac{1}{R}}\leq e^{\frac{1}{e}}$ for all $R\in(e,+\infty)$. Consider
$\tau^*=2+\frac{1}{R}>0$ for the fixed $R>e$, $\varepsilon>0$ and $q\geq2$, there is a
  constant $M_\varepsilon>0$ independent of $R>e$ such that
\begin{equation}\label{22growth3}
  |f^{R,2}(s)|\leq \varepsilon |s|^{\chi-1}+M_\varepsilon|s|^{q-1}(e^{2\alpha e^{\frac1e} |s|^2}-1),~\forall s\in\R.
\end{equation}
Actually, using $(h_1)$, by some elementary calculations,
\[
0\leq \frac{ |f^{R,2}(s)| }{|s|^{\chi-s}}\leq \frac{ |h(s)|e^{\alpha R^{\tau-2}|s|^2} }{|s|^{\chi-1}}
\leq \frac{ |h(s)|e^{\alpha R^{\tau^*-2}|s|^2} }{|s|^{\chi-1}}
\leq \frac{ |h(s)|e^{\alpha R^{\frac1R}|s|^2} }{|s|^{\chi-1}}
\leq \frac{ |h(s)|e^{\alpha e^{\frac1e}|s|^2} }{|s|^{\chi-1}}\to0
\]
uniformly in $s\to0$. On the other hand, there is a $s_0>0$ independent of $R>e$ such that
\begin{align*}
|f^{R,2}(s)| &
\leq Me^{\gamma |s|^{\delta}}e^{\alpha e^{\frac1e} |s|^{2}}
\leq  \bar{M}   e^{2\alpha e^{\frac1e} |s|^{2}}
 \leq   \tilde{M}  |t|^q(e^{2\alpha e^{\frac1e} t^{2}} -1), ~\forall |s|\geq s_0,
\end{align*}
where the constants $\tilde{M}>\bar{M}>M$ are independent of $R>e$ and we have used $(h_3)$.

\[
 0\leq f^{R,\delta}(t)\leq Me^{\gamma t^\delta}e^{\alpha_0^* R^{\tau-\delta} t^\delta}
\leq M e^{(\gamma + 1)t^\delta} \leq M_1e^{\alpha t^2}, ~\forall t\in\R,
\]
where $M_1>0$ is a constant independent of $R>0$.
Similarly, there holds
\begin{equation}\label{22growth4}
  |F^{R,\delta}(s)|\leq \varepsilon |s|^{\chi}+M_\varepsilon|s|^{q}(e^{2\alpha e^{\frac1e} |s|^2}-1),~\forall s\in\R.
\end{equation}

Finally, we show \eqref{20bounded01}. In fact, we make full use \eqref{PSsequence1} and \eqref{PSsequence3} jointly with Lemma \ref{AR}
to get
\begin{align*}
  \int_{\R^2}|\nabla u_n|^2dx &=2E^{R,2}(u_n)-\int_{\R^2}(A_1^2+A_2^2)|u_n|^2dx+\int_{\R^2}F(u_n)dx+o_n(1) \\
    &\leq  2E^{R,2}(u_n)+\frac{2}{\theta-4}\big[E^{R,2}(u_n)-\frac12J^{R,2}(u_n)\big]+o_n(1)\\
 &=  \frac{2(\theta-3)}{\theta-4} E^{R,2}(u_n) +o_n(1)
=  \frac{2(\theta-3)}{\theta-4}  m^{R,2}(a)+o_n(1)
\end{align*}
which together with \eqref{2bounded4} and \eqref{2bounded4444} indicates the first part of \eqref{20bounded01}.
The remaining part is trivial, we omit it here. The proof is completed.
\end{proof}

\begin{proof}[\textbf{\emph{Proof of Theorem \ref{maintheorem3}}}]\textbf{(Completed).}
We
 claim that $|f^{R,2}(\bar{u}_0^R)|_2$ is uniformly bounded in $R>e$.
Let us contemplate $\bar{w}_0^R=\Pi^{-\frac{1}{2}}\bar{u}_0^R$,
then $|\nabla\bar{w}_0^R|^2_{2}\leq 1$ and $|\bar{w}_0^R|^2_{2}=\Pi^{-1}a^2<+\infty$. According to the definition of $\tau^*$,
we deduce that \eqref{22growth3} still remains true.
 Since $q\geq2$ and \eqref{20bounded00} gives that
\[
 8\alpha e^{\frac1e} \Pi< 4\pi ,
\]
we then apply \eqref{TM2} to conclude that
\[
 \int_{\R^2} |\bar{u}_0^R|^{2(q-1)}(e^{4\alpha e^{\frac1e} |\bar{u}_a^R|^2}-1) dx
\leq \bigg(\int_{\R^2} |\bar{u}_0^R|^{4(q-1)} dx\bigg)^{\frac12}
\bigg(\int_{\R^2} (e^{8\alpha e^{\frac1e}\Pi |\bar{w}_0^R|^2}-1) dx\bigg)^{\frac12}\leq C
\]
for some $C>0$ independent of $R>e$. Besides, $2(\chi-1)\geq6$, it simply calculates that
\[
\int_{\R^2} |\bar{u}_0^R|^{2(\chi-1)} dx\leq C
\]
for some $C>0$ independent of $R>e$. Exploiting \eqref{22growth3}, we could derive the claim.
Repeating the remaining parts in proof of Theorem \ref{maintheorem3} in Subsection \ref{subcritical}, we must have
$$
|\bar{u}_0^R|_\infty \leq \bar{C}_0, \quad \forall R>e.
$$
for some $\overline{C}_0>0$ independent of $R>e$.
Finally, we obtain that the couple $(\bar{u}_0^R,\bar{\lambda}_0^R)$ is a weak solutions to problems \eqref{mainequation1}-\eqref{mainequation1a}
by choosing $R=\tilde{C}_0=\max\{\bar{C}_0,e\}>0$ in \eqref{fR} and then $\tau^*=2+\frac1{ \tilde{C}_0} >0$. The proof is completed.
\end{proof}

\section*{Appendix}

\noindent \textbf{Lemma A.1.}
Suppose that $f$ satisfies \eqref{definitionsubcriticalgrowth} and $(f_1)$, or
\eqref{definitioncriticalgrowth} and $(f_1)$. Let $u\in H^1(\R^2)$ be a
nontrivial weak solution of Eq. \eqref{mainequation1} with some suitable $\lambda>0$,
then it satisfies the so-called Poho\u{z}aev identity below
\[
 \lambda\int_{\R^2}|u|^2dx
 +2\int_{\R^2}(A_1^2+A_2^2)|u|^2dx-2\int_{\R^2}F(u)dx
=0.
\]

\begin{proof}
Since $u\neq0$, taking \eqref{CSS2d} and \eqref{CSS2e1}-\eqref{CSS2e2} into account,
 define
\[
b(x)=\frac{f(u(x))}{u(x)}-(A_0[u(x)]+A_1^2[u(x)]+A_2^2[u(x)])-\lambda,
\]
by means of \eqref{TM1} and using a similar arguments in Lemma \ref{3A},
we can conclude that $b(x)\in L_{\text{loc}}^1(\R^2)$
and $u$ satisfies the following elliptic equation
 $$
-\Delta u= b(x)u.
$$
In view of the classic Br\'{e}zis-Kato theorem,
one would conclude that $u\in L_{\text{loc}}^s(\R^2)$ for all $1\leq s<+\infty$.
Thus $u\in W^{2.s}_{\text{loc}}(\R^2)$ for all $1\leq s<+\infty$
by the Cald\'{e}ron-Zygmund inequality.

To derive the Poho\u{z}aev type identity for
Eq. \eqref{mainequation1}, we use a truncation argument due to Kavian (see e.g. \cite[Appendix B]{Willem}).
Let $\psi\in C^\infty([0,+\infty),[0, 1])$ such that $\psi(r) = 1$ for
$r\in [0, 1]$ and $\psi(r) = 0$ for $r\in[2,+\infty)$. Define $\psi_n(x) \triangleq \psi(|x|^
2/n^2)$ on $\R^2$ for $n\in \mathbb{N}$. Then
there exists $C_1>0$ such that
\[
0\leq \psi_n\leq C_1~\text{and}~|x||\nabla \psi_n(x)|\leq C_1.
\]
Multiplying Eq. \eqref{mainequation1} by $\psi_n(x,\nabla u)$, we have for every
$n\in \mathbb{N}$,
\begin{equation}\label{Pohozaev1}
0=\big\{-\Delta u+ \lambda u+ \left(A_0[u]+ A_1^2[u]+A_2^2[u]\right)u-
 f(u)\big\}\psi_n(x,\nabla u).
\end{equation}
For every
$n\in \mathbb{N}$, by virtue of the properties of the divergence,
 it is clear to compute that
\[
\left\{
  \begin{array}{ll}
   -\psi_n(x,\nabla u)\Delta u   =-\text{div}\bigg\{\bigg[
\nabla u(x,\nabla u)-x\frac{|\nabla u|^2}{2}\bigg]\psi_n\bigg\}
-\frac{|\nabla u|^2}{2}(x,\nabla \psi_n)+(x,\nabla u)(\nabla \psi_n,\nabla u), \\
  \psi_n(x,\nabla u)u    =\frac{1}{2}\text{div}[xu^2\psi_n]
-u^2\psi_n-\frac{1}{2}u^2(x,\nabla \psi_n),\\
  -\psi_n(x,\nabla u)f(u) =-\text{div}(x\psi_nF(u))+2\psi_nF(u)+F(u)(x,\nabla \psi_n),
  \end{array}
\right.
\]
 and by setting $\phi_u=A_0[u]+A_1^2[u]+A_2^2[u]$ that
\[
\psi_n(x,\nabla u)\phi_uu
=\frac{1}{2}\text{div}\big[\phi_uxu^2\psi_n\big]
 -\phi_uu^2\psi_n-\frac{1}{2}\phi_u u^2(x,\nabla \psi_n)-\frac{1}{2} u^2\psi_n(x,\nabla \phi_u).
\]
It follows from the divergence theorem and \eqref{Pohozaev1} that
\[
\begin{gathered}
\int_{\partial B_{2n}(0)}\bigg\{
\frac{|(x,\nabla u)|^2}{2n}-n|\nabla u|^2-n\lambda u^2-
n\phi_uu^2+2nF(u)
\bigg\}\psi_nd\sigma\hfill\\
\ \ \ \ = -\int_{B_{2n}(0)}\bigg\{\lambda u^2+\phi_uu^2+\frac12 (x,\nabla \phi_u)u^2-2F(u)\bigg\}\psi_ndx
\hfill\\
\ \ \ \ \ \ \ \  -\frac{1}{2}\int_{B_{2n}(0)}\bigg\{|\nabla u|^2+\lambda u^2+\phi_uu^2-2F(u)\bigg\}(x,\nabla\psi_n)dx
\hfill\\
\ \ \ \ \ \ \ \   +\int_{B_{2n}(0)}(x,\nabla u)(\nabla\psi_n,\nabla u)dx.
\hfill\\
\end{gathered}
\]
Note that
\[
\int_{\R^2}(\nabla \phi_u,x)u^2dx=-2\int_{\R^2}(A_1^2[u]+A_2^2[u])u^2dx.
\]
Recalling $\psi_n\equiv0$ on $\partial B_{2n}(0)$ together with the definitions of the cutoff function $\psi_n$, then
exploiting the Dominated Convergence theorem, we obatin
\[
\begin{gathered}
 \int_{\R^2}\lambda u^2dx+2\int_{\R^2}(A_1^2[u]+A_2^2[u])u^2dx
-2\int_{\R^2}F(u)dx\hfill\\
\ \ \ \    =\lim_{n\to\infty}\int_{B_{2n}(0)}\bigg\{\lambda u^2+\phi_uu^2+\frac12(x,\nabla\phi_u)u^2-2F(u)\bigg\}\psi_ndx
\hfill\\
\ \ \ \    =-\frac{1}{2}\lim_{n\to\infty}
\int_{B_{\sqrt{2}n}(0)\backslash B_{2n}(0)}\bigg\{|\nabla u|^2+\lambda u^2+\phi_uu^2-2F(u)\bigg\}(x,\nabla\psi_n)dx
\hfill\\
\ \ \ \ \ \ \ \   +\lim_{n\to\infty}\int_{B_{\sqrt{2}n}(0)\backslash B_{2n}(0)}(x,\nabla u)(\nabla\psi_n,\nabla u)dx
\hfill\\
\ \ \ \ =0\hfill\\
\end{gathered}
\]
showing the desired result, where the Fubini's theorem
 is used in the first equality. So, we accomplish the proof.
\end{proof}

\noindent \textbf{Lemma A.2.} Suppose that $f$ satisfies \eqref{definitionsubcriticalgrowth} and $(f_1)$, or
\eqref{definitioncriticalgrowth} and $(f_1)$. Let $u_n\to u$
in $H^1(\R^2)$ and $u_n\to u$ a.e. in $\R^2$,
then
\[
\lim_{n\to\infty}\int_{\R^2}f(u_n)u_ndx=\int_{\R^2}f(u)udx~\text{and}
\lim_{n\to\infty}\int_{\R^2}F(u_n)dx=\int_{\R^2}F(u)dx.
\]

\begin{proof}
Since $\|u_n-u\|^2_{H^1(\R^2)}\to0$, adopting some very similar calculations in the Step III
in the proof of Theorem \ref{maintheorem2}, we apply \eqref{growth1} and \eqref{TM2} to have that
\[
\int_{\R^2}f(u_n)u_ndx\leq C\int_{\R^2}|u_n|^\chi dx+C\left(\int_{\R^2}|u_n|^{q\nu^\prime} dx\right)^{\frac1{\nu^\prime}}.
\]
Now, one can conclude that $f(u_n)u_n\to f(u)u$ by using a variant of Vitali's Dominated Convergence theorem.
The remaining part is trivial, we omit it here.
\end{proof}

\noindent \textbf{Lemma A.3.} Suppose that $f$ satisfies $(f_1)$ and $(f_3)$, then there holds
\[
f^\prime(s)s^2-5f(s)s+8F(s)>0,~\forall s\in\R\backslash\{0\}.
\]
\begin{proof}
Since $f(s)\equiv0$ for all $s\in(-\infty,0]$, then it suffices to consider $s\in(0,+\infty)$.
According to $(f_1)$ and $(f_3)$, the function $\bar{F}\in \mathcal{C}^1$
and so $\bar{F}^\prime(s)>0$ for all $s\in(0,+\infty)$.
Obviously, there holds
\[
\frac d{ds}\bar{F}^\prime(s)=\frac{f^\prime(s)s^2-5f(s)s+8F(s)}{s^5}>0,~\forall s\in(0,+\infty),
\]
which gives the desired result.
\end{proof}

\noindent \textbf{Lemma A.4.}
For $\Psi= E\circ \zeta$, where $\zeta(u)=t_uu(t_u\cdot)$ for all $u\in S(a)$
and $t_u>0$ comes from Lemma \ref{unique}. Then, $\Psi$ is of $\mathcal{C}^1$ class over $S(a)$
and
\[
\Psi^\prime(u)[v]  =E^\prime(\zeta(u))[v_{t_u}]
\]
for any $u\in S(a)$ and $v\in \mathbb{T}_u$.

\begin{proof}
Recalling the definition of $A_1$ in \eqref{CSS2e1}, given some $u,v\in H^1(\R^2)$ and $t>0$, we have
\begin{align*}
  A_1[u+tv](x) & =-\frac{1}{4\pi }\int_{\R^2}\frac{ x_2-y_2 }{|x-y|^2}(u+tv)^2(y)dy \\
    & =-A_1[u](x)-\frac t{2\pi}\int_{\R^2}\frac{x_2-y_2}{|x-y|^2}u(y)v(y)dy-t^2A_1[v](x)
\end{align*}
leading to
\[
\lim_{t\to0} \frac{A_1[u+tv](x)-A_1[u](x)}{t}=-\frac 1{2\pi}\int_{\R^2}\frac{x_2-y_2}{|x-y|^2}u(y)v(y)dy.
\]
Arguing as a very similar way, by means of \eqref{CSS2e2},  there holds,
\[
\lim_{t\to0} \frac{A_2[u+tv](x)-A_2[u](x)}{t}= \frac 1{2\pi}\int_{\R^2}\frac{ x_1-y_1 }{|x-y|^2}u(y)v(y)dy.
\]
From which, we apply the Fubini's theorem jointly with \eqref{CSS2d} to deduce that
\[\begin{gathered}
\lim_{t\to0}\frac{1}{2t}\int_{\R^2}\sum_{j=1}^2\left(A_j^2[u+tv] -A_j^2[u] \right)|u|^2dx\hfill\\
\ \ \ \   =\lim_{t\to0}\frac{1}{2t}\int_{\R^2}\sum_{j=1}^2\left(A_j [u+tv] +A_j[u])(A_j [u+tv]-A_j[u]) \right)|u|^2dx\hfill\\
\ \ \ \ = \int_{\R^2}A_1[u]\left(-\frac 1{2\pi}\int_{\R^2}\frac{ x_2-y_2 }{|x-y|^2}u(y)v(y)dy\right)u^2dx\hfill\\
\ \ \ \ \ \ \ \   + \int_{\R^2}A_2[u]\left(\frac 1{2\pi}\int_{\R^2}\frac{ x_1-y_1 }{|x-y|^2}u(y)v(y)dy\right)u^2dx\hfill\\
\ \ \ \ = \int_{\R^2}\left(-\frac 1{2\pi}\int_{\R^2}\frac{ x_2-y_2 }{|x-y|^2}A_1[u](x)u^2(x)dx\right)  u(y)v(y)dy\hfill\\
\ \ \ \ \ \ \ \   + \int_{\R^2}\left(\frac 1{2\pi}\int_{\R^2}\frac{(x_1-y_1)}{|x-y|^2}A_2[u](x)u^2(x)dx\right)u(y)v(y)dy\hfill\\
\ \ \ \   = \int_{\R^2}A_0[u](x)u(x)v(x) dx.\hfill\\
\end{gathered}
\]
Due to $E(u_{t_u})=\max_{t>0}E(u_t)$ for all $u\in S(a)$ by Lemma \ref{unique},
then
\begin{align*}
\Psi(u+tv)-\Psi (u) & =E\big( (u+tv)_{t_{u+tv}} \big)-E(u_{t_u})
\leq E\big( (u+tv)_{t_{u+tv}}\big)-E(u_{t_{u+tv}}) \\
  &=  \frac{t_{u+tv}^2}{2}\int_{\R^2}\big[|\nabla(u+tv)|^2+(A^2_1[u+tv]+A^2_2[u+tv] )|u+tv|^2]dx\\
&\ \ \ \  -\frac{t_{u+tv}^2}{2}\int_{\R^2}\big[|\nabla u|^2+(A_1^2[u] +A_2^2[u] )|u|^2 \big]dx\\
&\ \ \ \  -t_{u+tv}^{-2}\int_{\R^2}[F(t_{u+tv} u+tt_{u+tv}v )-F(t_{u+tv} u)]dx\\
&=\frac{t_{u+tv}^2}{2}\int_{\R^2}[2t\nabla u\nabla v+t^2|\nabla v|^2+(A_1^2[u+tv]+A_2^2[u+tv])(2tuv+t^2|v|^2)]dx\\
&\ \ \ \  +\frac{t_{u+tv}^2}{2}\int_{\R^2}(A_1^2[u+tv]+A_2^2[u+tv]-A_1^2[u]-A_2^2[u])|u|^2dx\\
&\ \ \ \  -t_{u+tv}^{-2}\int_{\R^2}f(t_{u+tv} u+  \sigma_t tt_{u+tv}v )tt_{u+tv}vdx,
\end{align*}
where $\sigma_t\in(0,1)$ is determined by the Intermediate Value theorem. Also, we adopt
  $E\big( (u+tv)_{t_{u+tv}} \big)=\max_{t>0}E\big( (u+tv)_{t}\big)$ to get
\begin{align*}
\Psi(u+tv)-\Psi (u) & =E\big( (u+tv)_{t_{u+tv}} \big)-E(u_{t_u})
\geq E\big( (u+tv)_{t_{u}}\big)-E(u_{t_{u}}) \\
  &=  \frac{t_{u}^2}{2}\int_{\R^2}\big[|\nabla(u+tv)|^2+(A^2_1[u+tv]+A^2_2[u+tv] )|u+tv|^2]dx\\
&\ \ \ \  -\frac{t_{u}^2}{2}\int_{\R^2}\big[|\nabla u|^2+(A_1^2[u] +A_2^2[u] )|u|^2 \big]dx\\
&\ \ \ \  -t_{u}^{-2}\int_{\R^2}[F(t_{u} u+tt_{u}v )-F(t_{u} u)]dx\\
&=\frac{t_{u}^2}{2}\int_{\R^2}[2t\nabla u\nabla v+t^2|\nabla v|^2+(A_1^2[u+tv]+A_2^2[u+tv])(2tuv+t^2|v|^2)]dx\\
&\ \ \ \  +\frac{t_{u}^2}{2}\int_{\R^2}(A_1^2[u+tv]+A_2^2[u+tv]-A_1^2[u]-A_2^2[u])|u|^2dx\\
&\ \ \ \  -t_{u}^{-2}\int_{\R^2}f(t_{u} u+  \bar{\sigma}_t tt_{u}v )tt_{u}vdx,
\end{align*}
where $\bar{\sigma}_t\in(0,1)$.
Since the map $u+tv\to t_{u+tv}$ is continuous by Lemma \ref{unique}, one has
\begin{align*}
\Psi^\prime(u)[v]&=\lim_{t\to0}\frac{\Psi(u+tv)-\Psi (u)}{t} \\
& =\frac{t_{u}^2}{2}\int_{\R^2}[2\nabla u\nabla v+(A_1^2[u]+A_2^2[u]+A_0[u])2uv]dx
 -t_{u}^{-2}\int_{\R^2}f(t_{u} u)t_{u}vdx\\
&=
\int_{\R^2} \big[\nabla u_{t_u} \nabla v_{t_u}
+(A_1^2[u_{t_u}]+A_2^2[u_{t_u}]+A_0[u_{t_u}])u_{t_u}v_{t_u}\big]dx - \int_{\R^2}f( u_{t_u})v_{t_u}dx\\
&=E^\prime({u_{t_u}})[v_{t_u}]=E^\prime(\zeta(u))[v_{t_u}]
\end{align*}
which is the desired result. The proof is completed.
\end{proof}

\noindent \textbf{Lemma A.5.} Let $u,v\in H^1(\R^2)$ and suppose
 that $\supp u_t \cap \supp v=\emptyset$ for all $t>0$,
 then
 \[
 \lim_{t\to0^+}\left(\int_{\R^2}A_j^2[u_t+v]|u_t+v|^2dx-\int_{\R^2}A_j^2[u_t]|u_t|^2dx
 - \int_{\R^2}A_j^2[ v]| v|^2dx\right)=0.
 \]

 \begin{proof}
 According to \eqref{CSS2e1} and \eqref{CSS2e2}, for all $t>0$, there holds
 \[
 A_j[u_t+v]=A_j[u_t]+A_j[v],~j=1,2,
 \]
 and
 \[
 |u_t+v|^2=|u_t|^2+|v|^2.
 \]
 It follows from some simple calculations that
 \[\begin{gathered}
 \left|\int_{\R^2}A_j^2[u_t+v]|u_t+v|^2dx-\int_{\R^2}A_j^2[u_t]|u_t|^2dx
 - \int_{\R^2}A_j^2[ v]| v|^2dx\right|
 \hfill\\
 \ \ \ \  = \left|\int_{\R^2}\big(A_j^2[u_t]v^2+A_j^2[v]u_t^2+2A_j[u_t]A_j[v]u_t^2+2A_j[u_t]A_j[v]v^2\big)dx\right| \hfill\\
 \ \ \ \  \leq |A_j[u_t]|_{\widehat{r}}^2 |v|_{\frac{2\widehat{r}}{\widehat{r}-2}}^{2}
 +|A_j[v]|_{\widehat{r}}^2 |u_t|_{\frac{2\widehat{r}}{\widehat{r}-2}}^{2}
 +2 |A_j[u_t]|_{\widehat{r}} |A_j[v]|_{\widehat{r}} |u_t|_{\frac{2\widehat{r}}{\widehat{r}-2}}^{2}
 +|A_j[u_t]|_{\widehat{r}}|A_j[v]|_{\widehat{r}} |v|_{\frac{2\widehat{r}}{\widehat{r}-2}}^{2} \hfill\\
 \ \ \ \  =t^2|A_j[u]|_{\widehat{r}}^2 |v|_{\frac{2\widehat{r}}{\widehat{r}-2}}^{2}
 +t^{\frac4{\widehat{r}}}|A_j[v]|_{\widehat{r}}^2 |u |_{\frac{2\widehat{r}}{\widehat{r}-2}}^{2}
 +2t^{1+\frac4{\widehat{r}}} |A_j[u_t]|_{\widehat{r}} |A_j[v]|_{\widehat{r}} |u_t|_{\frac{2\widehat{r}}{\widehat{r}-2}}^{2}
 +t|A_j[u_t]|_{\widehat{r}}|A_j[v]|_{\widehat{r}} |v|_{\frac{2\widehat{r}}{\widehat{r}-2}}^{2} \hfill\\
 \ \ \ \   \to0 \hfill\\
 \end{gathered}
 \]
 as $t\to0^+$, where we exploited Lemma \ref{gauge1} with $\widehat{r}>2$. The proof is completed.
 \end{proof}


\vskip18pt

\par\noindent {\bf Conflict of interest.} The authors have no competing
interests to declare that are relevant to the content of this
article. 

\bigskip


\vskip0.1in


\begin{thebibliography}{99}

\bibitem{AYA}
Adimurthi, S.L. Yadava, Multiplicity results for semilinear elliptic equations in bounded domain of $\mathbb{R}^{2}$
involving critical exponent, {\it Ann. Scuola Norm. Sup. Pisa Cl. Sci.}, \textbf{17} (1990), 481--504.

\bibitem{AY}
Adimurthi, Y. Yang, An interpolation of Hardy inequality and Trudinger-Moser inequality in $\R^N$
 and its
applications, {\it Int. Math. Res. Not. IMRN,} {\bf13} (2010), 2394--2426.

\bibitem{AJM}
C.O. Alves, C. Ji, O.H. Miyagaki, Normalized solutions for a Schr\"{o}dinger equation with critical growth in $\R^N$,
 {\it Calc. Var. Partial Differential Equations,} {\bf61} (2022), no. 1, Paper No. 18, 24 pp.

\bibitem{AS1}
C.O. Alves, L. Shen,
On existence of solutions  for some classes of elliptic problems with supercritical exponential growth,
accepted by {\it Math. Z.,} (2023).

\bibitem{AS2}
C.O. Alves, L. Shen,
On a class of strongly indefinite Schr\"{o}dinger equations with Stein-Weiss convolution parts and supercritical
exponential growth in $\R^2$, {\it submitted}.

 \bibitem{AS}
C.O. Alves, L. Shen, Normalized solutions for the Schr\"{o}dinger equations with Stein-Weiss convolution parts in $\R^2$: the supercritical case,
{\it submitted}.

\bibitem{AP}
A. Azzollini, A. Pomponio, Positive energy static solutions for the Chern-Simons-Schr\"{o}dinger system under a large-distance fall-off requirement
 on the gauge potentials, {\it Calc. Var. Partial Differential Equations,} {\bf60} (2021), no. 5, Paper No. 165, 30 pp.

 \bibitem{BartschSoave}
T. Bartsch, N. Soave, A natural constraint approach to normalized solutions of nonlinear
Schr\"{o}dinger equations and systems, {\it J. Funct. Anal.,} {\bf272} (2017), 4998--5037.

 \bibitem{BartschSoave2}
T. Bartsch, N. Soave, Multiple normalized solutions for a competing system of Schr\"{o}dinger equations,
{\it Calc. Var. Partial Differential Equations,} {\bf58} (2019), no. 1, Paper No. 22, 24 pp.

\bibitem{Brezis} H. Br\'ezis, {\it Functional analysis, Sobolev spaces and partial differential equations}, New York: Springer, 2011.

\bibitem{Byeon}
J. Byeon, H. Huh, J. Seok, Standing waves of nonlinear Schr\"{o}dinger equations with the gauge field,
{\it J. Funct. Anal.,} {\bf263} (2012), 1575--1608.

\bibitem{Cao}
D. Cao, Nontrivial solution of semilinear elliptic equation with critical exponent in $\R^2$, {\it Commun.
Partial Differential Equations,} {\bf17} (1992), 407--435.


\bibitem{Chang}
K.C. Chang, Methods in nonlinear analysis, \emph{Springer}, 2003.

\bibitem{CZT}
S. Chen, B. Zhang, X. Tang, Existence and concentration of semiclassical ground state solutions for the generalized
 Chern-Simons-Schr\"{o}dinger system in $H^1(\R^2)$, {\it Nonlinear Anal.,} {\bf185} (2019), 68--96.

  \bibitem{Figueiredo}
D.G. de Figueiredo, O.H. Miyagaki, B. Ruf, Elliptic equations in $\R^2$ with
nonlinearities in the critical growth range, {\it Calc. Var. Partial Differential Equations,} {\bf3} (1995), 139--153.

\bibitem{DSD}
	\newblock M. de Souza, J.M. do \'{O},
	\newblock {A sharp Trudinger-Moser type inequality in $\R^2$},
	\newblock {\it Trans. Ame. Math. Soc.,} \textbf{366} (2014), 4513--4549.

\bibitem{DPS}
Y. Deng, S. Peng, W. Shuai, Nodal standing waves for a gauged nonlinear Schr\"{o}dinger equation in $\R^2$,
{\it J. Differential Equations,} {\bf264} (2018), no. 6, 4006--4035.

\bibitem{Dunne}
G. Dunne, Self-Dual Chern-Simons Theories, Springer, 1995.

\bibitem{Esry}
  B.D. Esry, C.H. Greene, Jr., J. P. Burke, J.L. Bohn, Hartree-Fock theory for double condensates, {\it Phys. Rev. Lett.,} {\bf78} (1997) 3594--3597.

\bibitem{Frantzeskakis}
  D.J. Frantzeskakis, Dark solitons in atomic Bose Einstein condensates: From theory to experiments, {\it J. Phys. A: Math. Theor.,}
  {\bf 43}, (2010)
213001.

\bibitem{Ghoussoub}
N. Ghoussoub, Duality and Perturbation Methods in Critical Point Theory, volume 107
of Cambridge Tracts in Mathematics. \emph{Cambridge University Press}, Cambridge, 1993.

\bibitem{GZ}
T. Gou, Z. Zhang, Normalized solutions to the Chern-Simons-Schr\"{o}dinger system, {\it J. Funct. Anal.,}
{\bf 280} (2021), no. 5, Paper No. 108894, 65 pp.

\bibitem{Huh1}
H. Huh, Blow-up solutions of the Chern-Simons-Schr\"{o}dinger equations, {\it Nonlinearity,}
{\bf22} (2009), 967--974.

\bibitem{Huh2}
H. Huh, Standing waves of the Schr\"{o}dinger equation coupled with the Chern-Simons gauge field. {\it J.
Math. Phys.,} {\bf53} (2012), 063702.

 \bibitem{Jackiw1}
R. Jackiw, S. Pi, Classical and quantal nonrelativistic Chern-Simons theory, {\it Phys. Rev. D,} {\bf42} (1990), 3500--3513.

\bibitem{Jackiw2}
R. Jackiw, S. Pi, Soliton solutions to the gauged nonlinear Schr\"{o}dinger equation on the plane,
{\it Phys. Rev. Lett.} {\bf64}
(1990), 2969--2972.

\bibitem{Jackiw3}
R. Jackiw, S. Pi, Self-dual Chern-Simons solitons, {\it Progr. Theoret. Phys. Suppl.} {\bf107} (1992), 1--40.

\bibitem{Jeanjean1997}
L. Jeanjean, Existence of solutions with prescribed norm for semilinear elliptic equations, \textit{Nonlinear Anal.}, \textbf{28} (1997), 1633--1659.

\bibitem{JeanjeanLu1}
L. Jeanjean, S. Lu, Nonradial normalized solutions for nonlinear scalar field equations, {\it Nonlinearity,}
 {\bf 32} (2019), no. 12, 4942--4966.

\bibitem{JeanjeanLu2}
L. Jeanjean, S. Lu, A mass supercritical problem revisited,
{\it Calc. Var. Partial Differential Equations,} {\bf59} (2020), no. 5, Paper No. 174, 43 pp.

\bibitem{JeanjeanLu3}
L. Jeanjean, S. Lu, Normalized solutions with positive energies for a coercive problem and application to the cubic-quintic nonlinear Schro
 dinger equation, {\it Math. Models Methods Appl. Sci.,} {\bf32} (2022), no. 8, 1557--1588.

\bibitem{JCZSA}
 Y. Jiang, T. Chen, J. Zhang, M. Squassina, N. Almousa, Ground states of Schr\"{o}dinger systems with the Chern-Simons gauge fields,
 {\it Adv. Nonlinear Stud.} {\bf23} (2023), no. 1, Paper No. 20230086, 16 pp.


\bibitem{KT}
J. Kang, C. Tang, Existence of nontrivial solutions to Chern-Simons-Schr\"{o}dinger system with indefinite potential,
{\it Discrete Contin. Dyn. Syst. Ser. S,} {\bf14} (2021), no. 6, 1931--1944.

\bibitem{Lu2012}
N.~Lam, G.~Lu, Existence and multiplicity of solutions to equations of $n$-Laplacian type with critical exponential growth in $\mathbb{R}^n$,
 \textit{J. Funct. Anal.}, \textbf{262} 2012), 1132--1165.

\bibitem{LX}
G. Li, X. Luo, Normalized solutions for the Chern-Simons-Schr\"{o}dinger equation in $\R^2$,
{\it Ann. Acad. Sci. Fenn. Math.,} {\bf42} (2017), no. 1, 405--428.

\bibitem{LYY}
L. Li, J.F. Yang, J.G. Yang, Solutions to Chern-Simons-Schr\"{o}dinger systems with external potential,
{\it Discrete Contin. Dyn. Syst. Ser. S,} {\bf14} (2021), no. 6, 1967--1981.

\bibitem{LiXinfu}
X. Li, Existence of normalized ground states for the Sobolev critical Schr\"{o}dinger equation with
combined nonlinearities, {\it Calc. Var. Partial Differential Equations,} {\bf60} (2021), 1--14.

\bibitem{LC}
W. Liang, C. Zhai, Existence of bound state solutions for the generalized Chern-Simons-Schr\"{o}dinger system in $H^1(\R^2)$,
{\it Appl. Math. Lett.,} {\bf100} (2020), 106028, 7 pp.

\bibitem{LM}
E.H. Lieb, M. Loss, Analysis: Graduate Studies in Mathematics. AMS, Providence (2001).

\bibitem{LS1}
B. Liu, P. Smith, Global well posedness of the equivariant Chern-Simons-Schr\"{o}dinger equation,
{\it Rev. Mat. Iberoam.,}
{\bf32} (2016), 751--794.

\bibitem{LS2}
B. Liu, P. Smith, D. Tataru, Local wellposedness of Chern-Simons-Schr\"{o}dinger, {\it Int. Math. Res. Not. IMRN,}
{\bf 23} (2014),
6341--6398.

\bibitem{LOZ}
Z. Liu, Z. Ouyang, J. Zhang, Existence and multiplicity of sign-changing standing waves for a gauged nonlinear
Schr\"{o}dinger equation in $\R^2$, {\it Nonlinearity,} 32 (2019), no. 8, 3082--3111.

\bibitem{Luo}
X. Luo, Multiple normalized solutions for a planar gauged nonlinear Schr\"{o}dinger equation,
{\it Z. Angew. Math. Phys.,} {\bf69} (2018), no. 3, Paper No. 58, 17 pp.

\bibitem{Malomed}
B. Malomed, Multi-Component Bose-Einstein Condensates: Theory, Emergent Nonlinear Phenomena in Bose-Einstein Condensation,
{\it Springer-Verlag}, Berlin, 2008.

\bibitem{PR}
A. Pomponio, D. Ruiz, A variational analysis of a gauged nonlinear Schrodinger equation,
{\it  J. Eur. Math. Soc. (JEMS),} {\bf17} (2015), no. 6, 1463--1486.

\bibitem{PSZZ}
A. Pomponio, L. Shen, X. Zeng, Y. Zhang, Generalized Chern-Simons-Schr\"{o}dinger system with sign-changing steep
potential well: critical and subcritical exponential case, {\it J. Geom. Anal.,} {\bf33} (2023), no. 6, Paper No. 185, 34 pp.

\bibitem{Shen}
L. Shen, Ground state solutions for a class of gauged Schr\"{o}dinger equations with subcritical and critical exponential growth,
{\it Math. Methods Appl. Sci.,} {\bf43} (2020), no. 2, 536--551.

\bibitem{Shen2}
L. Shen, Zero-mass gauged Schr\"{o}dinger equations with supercritical exponential growth, {\it submitted.}

\bibitem{SSY}
L. Shen, M. Squassina, M. Yang, Critical gauged Schr\"{o}dinger equations in $\R^2$ with vanishing potentials,
{\it Discrete Contin. Dyn. Syst. A,} {\bf42} (2022), no. 9, 4415--4438.

\bibitem{Soave1}
 N. Soave, Normalized ground states for the NLS equation with combined nonlinearities,
 {\it  J. Differential Equations,} {\bf269} (2020), no. 9, 6941-6987.

\bibitem{Soave2}
 N. Soave, Normalized ground states for the NLS equation with combined nonlinearities: the Sobolev critical case,
 {\it J. Funct. Anal.,} {\bf279} (2020), no. 6, 108610, 43 pp.

\bibitem{SzulkinWeth}
A. Szulkin, T. Weth, Ground state solutions for some indefinite variational problems,
 {\it J. Funct. Anal.,} {\bf257} (2009), 3802--3822.


\bibitem{WT}
Y. Wan, J. Tan, The existence of nontrivial solutions to Chern-Simons-Schr\"{o}dinger systems,
{\it Discrete Contin. Dyn. Syst. A,} {\bf37} (2017), no. 5, 2765--2786.

\bibitem{WeiWu}
J. Wei, Y. Wu, Normalized solutions for Schr\"{o}dinger equations with critical Sobolev exponent
and mixed nonlinearities, {\it J. Funct. Anal.,} {\bf283} (2022), 1--46.

\bibitem{Willem}
M. Willem, Minimax Theorems, Progress in Nonlinear Differential Equations and their Applications, 24, Birkh\"{a}user Boston Inc., Boston, MA,
1996.

\bibitem{YCS1}
S. Yao, H. Chen, J. Sun, Two normalized solutions for the Chern-Simons-Schr\"{o}dinger system with exponential critical growth,
{\it J. Geom. Anal.,} {\bf33} (2023), no. 3, Paper No. 91, 26 pp.

\bibitem{YCS2}
S. Yao, H. Chen, J. Sun, Normalized solutions to the Chern-Simons-Schr\"{o}dinger system under the nonlinear combined effect,
{\it Sci. China Math.,} {\bf66} (2023), no. 9, 2057--2080.

\bibitem{Yuan}
J. Yuan, Multiple normalized solutions of Chern-Simons-Schr\"{o}dinger system,
{\it NoDEA Nonlinear Differential Equations Appl.,} {\bf22} (2015), no. 6, 1801--1816.

\bibitem{YTC}
S. Yuan, X. Tang, S. Chen, Normalized solutions of Chern-Simons-Schr\"{o}dinger equations with exponential critical growth,
 {\it J. Math. Anal. Appl.,} {\bf516} (2022), no. 2, Paper No. 126523, 22 pp.

\end{thebibliography}
\end{document}